\documentclass{amsart}
\usepackage{amssymb,amsmath,amsthm,mathtools, geometry, amssymb, verbatim, esint, soul}
\usepackage{a4wide}
\usepackage{float}
\usepackage{graphicx}
\usepackage[utf8]{inputenc} 
\usepackage{a4wide}
\usepackage{tikz}
\usetikzlibrary{calc}
\usepackage{ulem}
\usepackage{stackengine,scalerel,graphicx}
\stackMath
\usepackage{amsfonts} 
\usepackage{latexsym}
\usepackage[font=small,format=hang,labelfont={sf,bf}]{caption}
\usepackage{epsfig}
\usepackage{subfig}
\usepackage{url}
\usepackage{varioref}
\usepackage{bm}
\usepackage{color}
\usepackage{mathrsfs}
\usepackage{latexsym}
\usepackage{bbm}
\usepackage{faktor}
\usepackage{yhmath}
\usepackage{empheq}
\usepackage{mathtools}
\usepackage{color}
\usepackage{marginnote}
\usepackage{comment}
\usepackage{cancel}
\usetikzlibrary{arrows.meta}

\allowdisplaybreaks[1]

\usepackage{amssymb,amsmath,mathrsfs, mathtools}
\usepackage[colorlinks=false]{hyperref}
\usepackage{comment}
\usepackage{diagbox}

\usepackage[colorinlistoftodos]{todonotes}
\setuptodonotes{color=green!40, fancyline, size=small}

\newcommand{\rewrite}[1]{\todo[color=red!40]{#1}}

\def\e{\varepsilon}
\def\rr{{\mathbb R}}
\def\dx{\,dx}
\def\NN{{\mathbb N}}
\def\ZZ{{\mathbb Z}}

\def\A{{\mathcal A}}
\newcommand{\I}{I}
\newcommand{\dis}{\mathcal{D}}
\newcommand{\Z}{Z}
\newcommand{\DD}{\mathscr{D}}
\newcommand{\PP}{\mathscr{P}}
\newcommand{\one}{\mathbb{I}}
\newcommand{\zero}{\mathbb{O}}
\newcommand{\enF}{\mathcal{F}^{\tau, \gamma}_\e}
\newcommand{\E}{\mathcal{E}}
\newcommand{\hh}{\mathcal{H}}
\newcommand{\nn}{\mathcal{N}}

\usepackage{stmaryrd}
\newcommand{\dseg}[2]{\llbracket #1, #2 \rrbracket}

\DeclareMathOperator*{\R}{\mathbb{R}}

\DeclareMathOperator*{\argmin}{arg\,min}
\DeclareMathOperator*{\len}{len}
\DeclareMathOperator*{\diam}{diam}

\def\XXint#1#2#3{{\setbox0=\hbox{$#1{#2#3}{\int}$} 
  \vcenter{\hbox{$#2#3$}}\kern-.5\wd0}}

\makeatletter

\@addtoreset{equation}{section}
\makeatother

\newtheorem{theorem}{Theorem}[section]
\newtheorem*{theorem*}{Theorem}
\newtheorem{lemma}[theorem]{Lemma}
\newtheorem{proposition}[theorem]{Proposition}

\theoremstyle{definition}
\newtheorem{definition}[theorem]{Definition}
\newtheorem{remark}[theorem]{Remark}

\title[Crystalline Motion for the Blume-Emery-Griffiths Model]{Crystalline Motion of discrete interfaces in the Blume-Emery-Griffiths Model}

\author[M. Cicalese]{M. Cicalese}
\address[Marco Cicalese]{Technische Universit\"at M\"unchen, Boltzmannstrasse 3, 85748 Garching, Germany	}
\email[]{marco.cicalese@tum.de}

\author[G. Fusco]{G. Fusco}
\address[Giuliana Fusco]{Scuola Superiore Meridionale, via Mezzocannone 4, 80134 Napoli, Italy}
\email[]{g.fusco@ssmeridionale.it}

\author[G. Savarè]{G. Savar\'e}
\address[Giovanni Savar\'e]{Technische Universit\"at M\"unchen, Boltzmannstrasse 3, 85748 Garching, Germany	}
\email[]{giovanni.savare@tum.de}

\date{\today}  

\begin{document}
\begin{abstract}
We study the discrete-to-continuum evolution of a lattice system consisting of two immiscible phases labelled by $-1$ and $+1$ in presence of a surfactant phase labelled by $0$. The system’s energy is described by the classical Blume-Emery-Griffith model on the lattice $\e\ZZ^2$, and its continuum evolution is obtained as $\e\to 0$ through a minimizing-movements scheme with a time step proportional to $\e$. The dissipation functional we choose contains two contributions: a standard Almgren-Taylor-Wang type  term  penalizing the distance between successive configurations of the $+1$ phase, and a term penalizing the variation of the surfactant mass and modeling surfactant evaporation. The latter term depends on a scaling parameter $\gamma>0$, which determines whether the surfactant mass is conserved at each time step. We focus on the case in which the initial configuration consists of a single crystal of phase $1$ completely wetted by the surfactant. For $\gamma>2$ the surfactant can loose mass and the evolution reduces to the crystalline mean curvature flow of an Ising-type model, while for $\gamma<2$ the conservation of the surfactant mass leads to a more complex evolution characterized by stronger non-uniqueness and partial pinning.	
\end{abstract}

\maketitle
	

\tableofcontents

\section{Introduction}

In statistical mechanics and materials science, lattice models are essential for understanding energy-driven physical systems, particularly in their capacity to model phase transitions and microstructure formation. Beyond their descriptive power, they serve as a crucial rigorous framework for the validation and derivation of asymptotic mesoscopic and continuum phenomenological models which can be obtained as the lattice spacing vanishes. A central technique in the description of the asymptotic behaviour of low energy states and in particular ground states of lattice models is the $\Gamma$-convergence. The $\Gamma$-limit defines a rigorous variational coarse-graining and the whole process goes under the name of discrete-to-continuum passage. The class of those discrete systems that give rise to an asymptotic continuum energy model favoring phase separation has been extensively studied in the static framework over the past 20 years (for a comprehensive study the reader is referred to the book \cite{ABCS} and its references). This class includes classical ferromagnetic Ising-type lattice spin systems. These consist of interacting particles occupying the nodes of (a subset of) a lattice with each particle labelled by a value from a finite set (in the classical Ising case, particles can take only two possible values). Each of these values defines a spin phase which tends to be connected to save energy, and, as a result, the entire system can be described in terms of a finite partition into connected subsets of the lattice (the phases) interacting via an appropriate perimeter functional. Within this description the asymptotic continuum energy is a (possibly anisotropic) perimeter-type energy of a partition of sets of finite perimeter. Finding the phase transition energy of the coarse grained system is therefore equivalent to solving a classical minimal partitioning problem in the calculus of variations.

The relation between the perimeter functional and the Ising system, discussed above in the static setting, extends to the dynamic setting. The appearance of interfacial energies as an approximation of spin energies in fact suggests to consider a continuum description of discrete evolution problems in terms of geometric motions of interfaces. For the ferromagnetic Ising model, where the spin variable takes only two values, the continuum description involves two homogeneous phases separated by an interface with crystalline perimeter energy (see \cite{ABC}). In \cite{BGN}, Braides, Gelli, and Novaga introduced a continuum geometric evolution of these phases by combining the Almgren-Taylor-Wang time-discrete minimizing-movements scheme \cite{Almgren-Taylor-Wang} (see also \cite{LS}) with a coarse-graining procedure. This pioneering result revealed the emergence of an evolution characterized by new and intriguing phenomena arising from the possible interaction between the two approximation parameters - namely, the lattice and time scales. This line of research has since been further developed by several authors, who have shown that the presence in the model of anisotropies or additional scales, such as in the case of periodic inclusions or oscillating forcing terms, can give rise to additional new and nontrivial features in the resulting motion (see \cite{BSc, BMN, BST, MN, S} and \cite{BS} for an introduction to the subject).\\

In this paper, following the variational approach outlined above, we begin the analysis of the discrete Blume–Emery–Griffiths (BEG) model introduced in \cite{BEG} (see also \cite{EOT} and refer to the book \cite{R} for an introduction to the physics of surfactants in interfacial phenomena). This model describes a three-phase Ising-type system, representing two immiscible phases (for instance, water and oil) in the presence of a third phase acting as a surface-active agent (surfactant). The surfactant, by adsorbing onto the interfaces between the two immiscible phases, can significantly modify the interfacial energy and, consequently, the morphology and evolution of the system.

In the discrete setting, a rigorous variational coarse-graining of this class of surfactant models— including the BEG energy—was obtained in \cite{ACS} through $\Gamma$-convergence (see also \cite{LGGZ1, LGGZ2} for related, though less rigorous, studies on ground-state behavior). Earlier continuum descriptions of surfactant effects in phase separation problems were developed in \cite{FMS} (see also \cite{CH1, CH2} for the vectorial and non-local version of the problem) and later extended in \cite{AB}. Building on these results, we now turn to the study of the evolution problem associated with the BEG model, and in particular to its dependence on the conservation (or loss) of surfactant mass during the motion. This effect is captured through a time-discrete minimizing-movements scheme of the type introduced in \cite{BGN}, suitably modified by a dissipation potential that incorporates a parameter controlling the degree of surfactant conservation.

From a physical perspective, numerous experimental and theoretical studies have shown that surfactants can profoundly alter the dynamics of interfaces in immiscible fluids. The accumulation of surfactant at the contact line may generate Marangoni stresses that stabilize the interface and induce pinning, whereas surfactant evaporation or redistribution can promote depinning and motion \cite{D1997, K2019, Z2017,  L2016, S2013, S2017}. These competing mechanisms—pinning through surfactant accumulation and depinning through surfactant loss—have a direct analogue in our model, where they correspond to two distinct regimes determined by the value of the control parameter. For small values of this parameter, the surfactant mass is effectively conserved, leading to pinned crystalline evolutions; for large values, evaporation dominates and the motion reduces to crystalline mean curvature flow. The ultimate goal of this paper is to provide a first rigorous description of these phenomena by constructing an explicit example of a single-crystal evolution that qualitatively changes depending on whether surfactant evaporation is allowed or not.

We further notice that many technical difficulties in our problem, and in primis the careful analysis needed to prove the connectedness of the phases along the flow, come from the fact that it involves three (and not only two) phases. Even in the continuum framework, where there is no lattice scale, the study of network evolutions arising as the boundaries of partitions with more than two phases is still not completely understood, despite the many interesting contributions obtained in the last years (see \cite{MNPS} and the references therein). In the discrete setting, rigorous results on the evolution by minimizing movements of multi-phase interfacial systems are quite scarce. A first example was obtained by the first author in \cite{BCY} (see also \cite{BS} and the references therein). The scarcity of rigorous results in the lattice case can be attributed to the intrinsic complexity of the problem: minimizers may be highly non-unique and may develop nontrivial lattice-scale oscillations, usually referred to as microstructures. We end this first part of the introduction by mentioning that, if the aim of the discrete-to-continuum evolution scheme is that of approximating the crystalline mean curvature flow, then other algorithms can be used in order to avoid microstructure formation, hence pinning phenomena or spatial drift (see e.g., \cite{CDGM})\\

\noindent {\bf The Blume-Emery-Griffith surfactant model}\\

\noindent Given a bounded open set $\Omega \subset \mathbb{R}^2$ with Lipschitz boundary, on the square lattice $\Omega_\e=\e\ZZ^2\cap \Omega$ we consider the set ${\mathcal{A}}_\e=\{u:\Omega_\e\rightarrow \{\pm1,0\}\}$ of those functions whose values on the nodes of the lattice correspond to particles of water and oil for $u=+1$ and $u=-1$, respectively or of surfactant for $u=0$. For a given configuration of particles, the energy of the system is given by
\begin{equation}
\label{BEG_energy}\E^{latt}_\e(u)=\sum_{n.n.}\e^2(-u(p)u(q)+k(u(p)u(q))^2)
\end{equation}
where $n.n.$ stands for nearest neighbors and it means that the sum is performed over those $p, q\in \Omega_\e$ such that $|p-q| = \e$. Here $k>0$ is a constant which represents the quotient between bi-quadratic and quadratic exchange interaction strengths.  In \cite{ACS} it was shown that, as $\e\rightarrow 0$, upon identifying arrays $\{u(p)\}, p \in \Omega_\e$ with their piecewise-constant interpolations on the cell of the lattice, the $\Gamma$-limit (with respect to the $L^{1}$-convergence) of $\mathcal{E}^{latt}_\e$ is finite on $L^{1}(\Omega;[-1,1])$ and it is given by the constant value $2|\Omega|(-1+k)\wedge 0$. Such a value is achieved by one of the three uniform states, namely $u=1$, $u=-1$ and $u=0$. As a consequence, on choosing $k<1$ one can select the uniform states $u=\pm1$ to be ground states. In the spirit of \textit{development by $\Gamma$-convergence} (see \cite{BT} for a more general treatment of this topic) one can refer the energy to its minimum $m_\e:=\sum_{n.n.}\e^2(k-1)$ and further scale the functionals by $\e$ and obtain the new family of functionals
\begin{equation}
    \label{BEG_riscalata}
    \E_\e(u):=\frac{\E^{latt}_\e(u)-m_\e}{\e}=\sum_{n.n.}\e(1-u(p)u(q)-k(1-(u(p)u(q))^2)).
\end{equation}
Observe that in \eqref{BEG_riscalata} the interaction energy of two neighboring particles taking both the value $+1$ (or $-1$) is zero, while the interaction of a surfactant particle $0$ with all the other particles costs the positive value $(1-k)$. For this reason it is also said that the BEG functional describes a \textit{repulsive surfactant model}. At this scaling, a simple computation shows that, on choosing $\frac{1}{3}<k<1$, interposing a particle in the phase $0$ between two particles with opposite phases decreases the energy. Hence one can expect that the phase $0$ acts as a surfactant phase in this range of values of $k$. This educated guess is confirmed by computing the $\Gamma$-limit as $\e\rightarrow 0$ (with respect to the $L^1$ convergence) of the energies $\E_\e$ and proving that it is given by the interfacial-type energy
\begin{equation}
    \label{BEG_riscalato_gamma_limite}
    \E(u)=\int_{S(u)}\psi(\nu_u)\, d\mathcal{H}^1
\end{equation}
where $u \in BV(\Omega;\{\pm1\}),\psi(\nu)=(1-k)(3|\nu_1|\vee|\nu_2|+|\nu_1|\wedge|\nu_2|)$ denotes the anisotropic surface tension, $S(u)$ is the jump set of $u$, namely the interface between the phases $\{u=1\}$ and $\{u=-1\}$ and $\nu_u$ is the measure theoretic inner normal to $S(u)$.\\

\noindent{\bf Time-discrete evolution of the Blume-Emery-Griffith surfactant model}\\

\noindent In our work we are interested in energy-driven motions deriving from the functional $\E_\e$. We follow the approach by Braides, Gelli, Novaga \cite{BGN} and Braides, Solci \cite{BS}, who coupled the minimizing movements approach a la Almgren, Taylor and Wang (\cite{Almgren-Taylor}, \cite{Almgren-Taylor-Wang}) with a discrete-to-continuous analysis. In our setting the minimizing movements scheme can be described as follows. Given a lattice spacing $\e>0$, we first introduce a time step $\tau = \tau(\e)$ and an initial datum $u^\e_0\in\A_\e$. Then, we define a discrete motion by successive minimizations: given the spin configuration at the $j$-th time step $u^\e_{j}$ we obtain the next configuration $u^\e_{j+1}$ as
\begin{equation}\label{intro:schema}
u^\e_{j+1}\in\text{ argmin}\Bigl\{\E_\e(u^\e_{j+1})+\frac{1}{\tau}\dis_{\e, \gamma}(u^\e_{j+1}, u^\e_j)\Bigl\},
\end{equation}
where $\dis_{\e, \gamma}$ is a dissipation term which penalizes configurations $u^\e_{j+1}$ which are distant from $u^\e_j$. In the present case it is given as the sum of two contributions $\dis_{\e, \gamma}\coloneqq \dis_\e^1+\e^\gamma \dis^0_\e$. The first one (see (\ref{eq:dis_1})) penalizes large variations between $\{u^\e_{j+1}=1\}$ and $\{u^\e_{j}=1\}$, and coincides with the dissipation introduced in \cite{BGN}. The second term, instead, penalizes large  variations of the total amount of surfactant, i.e. $\#\{u^\e_j=0\}$, see (\ref{eq:dis_0}). The parameter $\gamma>0$ acts as a scaling factor for the second contribution, thereby tuning its influence in the overall dissipation.  As it will be further stressed in the following of this introduction, the parameter $\gamma$ has a critical role in the conservation of the surfactant mass during the discrete evolution and it is here introduced to account for the influence of the presence (or absence) of surfactant evaporation on the geometry of the evolving phases. We observe that, without this second contribution, i.e. if the amount of surfactant is free to change at every step, then the minimizing movements we considered would coincide with the minimizing movements originated by functional (\ref{BEG_riscalato_gamma_limite}) and already analyzed in \cite[Section 4.1.5]{BS}. 
We then associate to every function $u^\e_j$ the set $A_j^\e\subset\rr^2$ given by the union of all the squares of side length $\e$ centered at points $p$ where $u^\e_j(p) = 1$ (see (\ref{eq:set_A})). We focus on the case in which lattice and time scale interact most, that is $\tau(\e)=\zeta\e$ for some $\zeta>0$, and we choose a single crystal initial datum $u^\e_0$, i.e., such that $A^\e_0\cap\e\ZZ^2$ is the discretization of an octagon (the Wulff shape of the functional (\ref{BEG_riscalato_gamma_limite})). Then for every $t\ge 0$ we define the set function $A^\e(t):=A^\e_{\lfloor t/\tau\rfloor}$. We prove that $A^\e(t)$ converges uniformly locally in time to a limit flow $A(t)$, and we prove several results which describe how $A(t)$ depends on the parameter $\gamma$. \\

The article is organized as follows. In Section \ref{sec:formulation of the probelm} we present the setting of the problem and we introduce some notation and definition. In Section \ref{sec:connection} we show that the set $A^\e_j$ is connected as long as its vertical and horizontal diameter are sufficiently large and we prove that at each time step $j$, $A^\e_{j+1}\subset A^\e_j$. Then, in the following sections, we discuss two qualitatively different motions corresponding to $\gamma>2$ and $\gamma<2$. In Section \ref{sec:gamma>2} we deal with the case $\gamma>2$. We denote by $\I^\e_j:=\{u^\e_j=1\}$ and $\Z^\e_j:=\{u^\e_j=0\}$
and for a set $J\subset\e\ZZ^2$ we denote by $\partial^+J\subset\e\ZZ^2$ the set of points not belonging to $J$ which are adjacent to a point of $J$. 
With this notation, we prove that at every time step $j$ the variation of surfactant, i.e. 
\begin{equation*}
\frac{\e^\gamma}{\tau}|\#Z^\e_{j}-\#Z^\e_{j+1}|
\end{equation*}
gives a negligible contribution to the dissipation, and that the minimizing set $I^\e_{j+1}$ is a discrete octagon completely surrounded by surfactant, i.e. $Z_{j+1}^\e = \partial^+I_{j+1}^\e$ (see Figure \ref{Figure_1_intro-2}). Finally, in Theorem \ref{teo:free_surfactant_movement} we characterize the limit flow $A(t)$ and we find that it agrees with the one described in \cite[Section 4.1.5]{BS} for the energy (\ref{BEG_riscalato_gamma_limite}).

\begin{figure}[H]
    \centering
    \resizebox{0.50\textwidth}{!}{\input{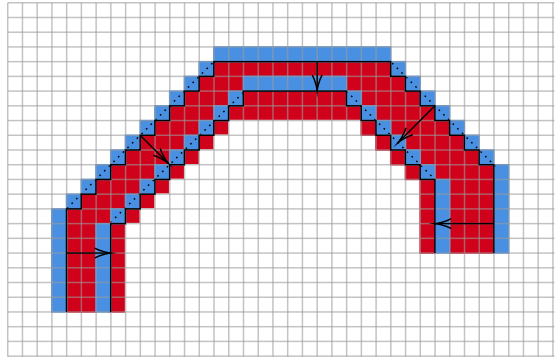}}
    \caption{The case $\gamma >2$: Discrete flow of a Wulff-type set (in red $u=1$ and in blue $u=0$).}
    \label{Figure_1_intro-2}
\end{figure}

In Section \ref{sec:gamma<2} we consider the case $\gamma<2$ and we analyze the minimizing movements under the additional assumption that at the starting time step $j=0$ the surfactant completely wets a Wulff shape, that is the set $I^\e_0$ is a discrete octagon completely surrounded by surfactant (i.e. $\partial^+I^\e_0\subset Z^\e_0$). In a forthcoming paper \cite{CFS2} we will deal with the remaining situations, that is with $\gamma = 2$, and with $\gamma<2$ in the non-complete wetting case.
The results that we obtain in Section \ref{sec:gamma<2} are described below. We focus on the case that $\e^2\#Z^\e_0\to 0$, that is we consider the minimizing movements under the assumption that the two dimensional measure of the surfactant is negligible as $\e\to 0$ (see \eqref{eq:5.1} and the beginning of Section \ref{sec:gamma<2} for the precise assumption). Note that if this assumption is not fulfilled, the surfactant phase would not disappear in the continuum limit and the motion should be described by a different set of equations taking into account also the evolution of the surfactant phase. First, we prove that at each time step the amount of surfactant remains constant and the minimizer set $I_j^\e$ is surrounded by surfactant.   For a set $J\subset\e\ZZ^2$ we denote by $R_J\subset\e\ZZ^2$ the smallest discrete rectangle which contains $J$. In Lemma \ref{lemma:energy_surfactant} we prove that the energy of $u^\e_j$ is given by 
\[\E_\e(u^\e_j) = 2\e(1-k)\#Z^\e_j+\frac12(1-k)\Bigl(Per(Z^\e_j\cup I^\e_j)+Per(I_j^\e)\Bigl).\]
Recalling that $\#Z^\e_j$ is constant with respect to $j$ at $\e$ fixed, we can explain how the minimization scheme works by simplifying the problem and considering in place of $\E_\e$ the energy
\begin{equation}\label{eq:reduced_energy}    
\tilde{\E}_\e(u^\e_j):=Per(Z^\e_j\cup I^\e_j)+Per(I_j^\e) = Per(R_{Z^\e_j\cup I^\e_j}) + Per(R_{I^\e_j}).
\end{equation}
The evolution, obtained via minimizing movements, consists of two stages. In the first stage, $I^\varepsilon_j$ is a discrete octagon having at least one of its sloped sides sufficiently long with respect to the scale parameter $\e$ (for the sake of simplicity in this introduction we can think of such a length to be of order $1$ and we refer the reader to \ref{prop:hausdorff_distance_from_boundary} for the precise assumption). In this case, we show that the minimizing set $I^\e_{j+1}$ is still an octagon with long sloped sides. However, differently from the situation analysed in Section \ref{sec:gamma>2}, while the sides which are parallel to the coordinate axes move inward from time step $j$ to time step $j+1$ (see Proposition \ref{prop:surrounded_shape_minimizer_beginning}), the long sloped ones remain pinned, i.e. they do not move (see Figure \ref{Figure_2_intro-2}).

\begin{figure}[H]
    \centering
    \resizebox{0.50\textwidth}{!}{\input{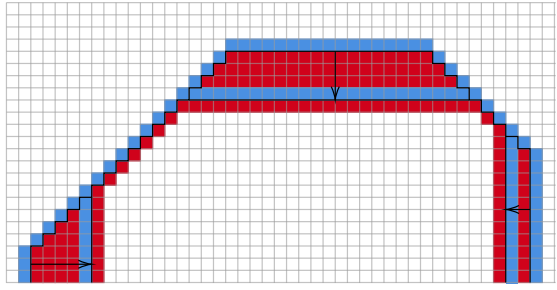}}
    \caption{The case $\gamma<2$. The first stage: pinning of the diagonals (in red $\{u=1\}$ and in blue $\{u=0\}$).}
    \label{Figure_2_intro-2}
\end{figure}

We give an heuristic explanation of the pinning. Using that there exists a long sloped side of $I^\e_{j+1}$, and recalling that $\e^2\#Z^\e_0\to 0$, one can show that the minimizer $u^\e_{j+1}$ satisfies $Z^\e_{j+1}\cup I^\e_{j+1}\subset R_{I^\e_{j+1}\cup\partial^+I^\e_{j+1}}.$ Since we already know that $I^\e_{j+1}$ is surrounded by surfactant, it follows that $Per(R_{Z^\e_{j+1}\cup I^\e_{j+1}}) = Per(R_{I^\e_{j+1}\cup\partial^+I^\e_{j+1}})$, and therefore from \eqref{eq:reduced_energy} we deduce that the energy of any minimizer is (up to a constant term and a multiplicative factor)
\begin{equation}\label{eq:reduced_energy_tilde}
\tilde{\E}_\e(u^\e_{j+1}) = \Bigl(Per(R_{I^\e_{j+1}})+8\e\Bigl)+Per(R_{I^\e_{j+1}}) = 8\e+2Per(R_{I^\e_{j+1}}).
\end{equation}
In particular the energy depends only on the perimeter of the minimizing set $I^\e_{j+1}$, and not on the set $Z^\e_{j+1}$. As a consequence of that and of the a priori estimate on the Hausdorff distance between consecutive optimal shapes (see Proposition \ref{prop:hausdorff_distance_from_boundary}), we obtain the pinning of the long diagonals. In order to give an idea of the argument, let us consider, for a given $u^\e_j$, the minimization scheme in \eqref{intro:schema} with $\tilde\E_\e$ in place of $\E_\e$. Let ${u}^\e_{j+1}$ be a minimizer and let $\tilde{u}^\e_{j+1}$ be a competitor which verifies $\tilde{I}^\e_{j+1}:=\{\tilde{u}^\e_{j+1} = 1\} = R_{I^\e_{j+1}}\cap I^\e_j$. First, we deduce from \eqref{eq:reduced_energy_tilde} that $\tilde{\E}_\e(u^\e_{j+1}) = \tilde{\E}_\e(\tilde{u}^\e_{j+1})$. In addition to that, as a consequence of Proposition \ref{prop:hausdorff_distance_from_boundary}, we also infer that it can not happen that $R_{I^\varepsilon_{j+1}}\subsetneq I^\varepsilon_j$. In other words, the sloped sides of  $\tilde{I}^\varepsilon_{j+1}$ are pinned with respect to the sloped sides of $I^\varepsilon_{j}$.  Finally from by the definition of $\tilde{I}^\e_{j+1}$ we have that $\dis^1_\e(u^\e_{j+1}, u^\e_j)\ge \dis^1_\e(\tilde{u}^\e_{j+1}, u^\e_j)$, and that a strict inequality holds if $I^\e_{j+1}\subsetneq\tilde{I}^\e_{j+1}$. By the minimality of $u^\e_{j+1}$ we eventually deduce that $I^\e_{j+1} = \tilde{I}^\e_{j+1}.$ In particular every long sloped side of the discrete octagon $I^\e_{j+1}$ is pinned. As the motion proceeds, some diagonals may vanish. When this happens, the surfactant which is no more adjacents to some point of $ I ^\varepsilon$ tends to build up along the upper diagonals of the long sloped side (see Figure \ref{Figure_3_intro-2}).
\begin{figure}[H]
    \centering
    \resizebox{0.50\textwidth}{!}{\input{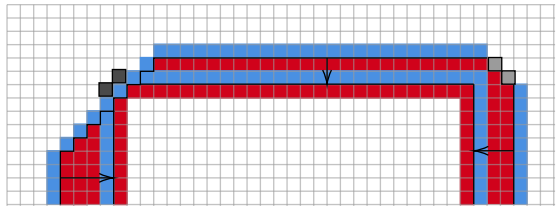}}
    \caption{The case $\gamma<2$: one diagonal disappears (in light gray, we represent the surfactant on the short diagonal at step $j$; in dark gray, we represent such surfactant that build up along the upper diagonal at the step $j+1$).}
    \label{Figure_3_intro-2}
\end{figure}
If the motion does not interrupt, the set $I_j^\e$ will become an octagon with sloped sides having length vanishing with $\e$ and the second stage of the motion starts.
In the second stage, the minimizer $I^\e_{j+1}$ resembles a discrete rectangle at each time step $j$ (see Definition \ref{def:quasi_rectangle} for a precise description of its shape). A precise description of the set $I_{j+1}^\e$ is more difficult, because it may happen that either $Z^\e_{j+1}\cup I^\e_{j+1}\subset R_{I^\e_{j+1}\cup\partial^+I^\e_{j+1}}$ (see Figure \ref{Figure_4_intro-2}) or $Z^\e_{j+1}\cup I^\e_{j+1}\subsetneq R_{I^\e_{j+1}\cup\partial^+I^\e_{j+1}}$ (see figure \ref{Figure_5_intro}). 
\begin{figure}[H]
    \centering
    \resizebox{0.50\textwidth}{!}{\input{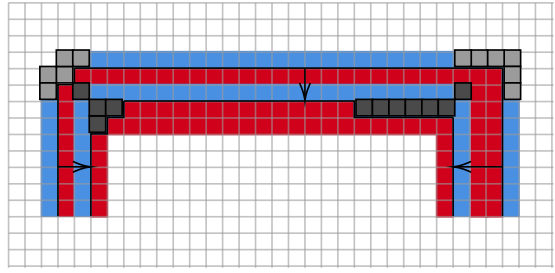}}
    \caption{The case $\gamma<2$: all the sloped sides are short and $Z^\varepsilon_{j+1}\cup I^\varepsilon_{j+1}\subset R_{I^\varepsilon_{j+1}\cup \partial^+I^\varepsilon_{j+1}}$ (in light gray, we represent the surfactant at step $j$ that is not adjacent to the sides parallel to the coordinate axes at step $j+1$. A possible rearrangement of such surfactant particles at step $j+1$ is represented in dark gray).}
    \label{Figure_4_intro-2}
\end{figure}

\begin{figure}[H]
    \centering
    \resizebox{0.80\textwidth}{!}{\input{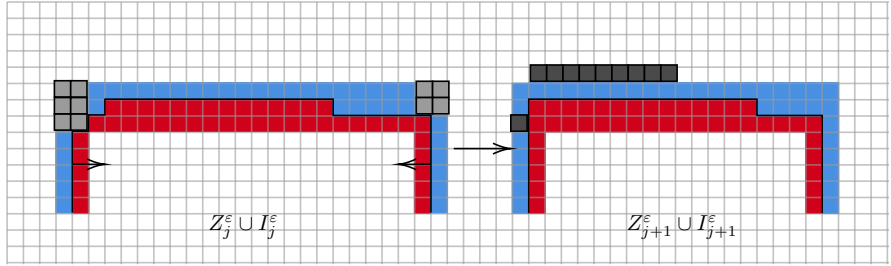}}
    \caption{The case $\gamma<2$: all the sloped sides are short and $Z^\varepsilon_{j+1}\cup I^\varepsilon_{j+1}\subsetneq R_{I^\varepsilon_{j+1}\cup \partial^+I^\varepsilon_{j+1}}$. To make it easier for the reader, $Z^\varepsilon_{j}\cup I^\varepsilon_{j}$ and $Z^\varepsilon_{j+1}\cup I^\varepsilon_{j+1}$ are shown non-overlapping, differently from the other figures (in light gray, we represent the surfactant at step $j$ that is rearranged at step $j+1$ as the dark gray particles on the right hand side).}
    \label{Figure_5_intro}
\end{figure}
The energy \eqref{eq:reduced_energy} favours minimizers $u^\e_{j+1}$ in which both $\I^\e_{j+1}$ and $\I^\e_{j+1}\cup\Z^\e_{j+1}$ are discrete rectangles, and the balance between these two geometric tendencies makes the description harder. In order to overcome this problem we need to construct proper barriers to the motion which eventually help us in obtaining a continuum description in Theorem \ref{teo:final teo A_j surrounded}.\\

We can summarize the results described above by saying that they match the main features of surfactant-induced-pinning phenomena as those observed in physical systems involving surfactant-mediated interface motion. In this context of evaporating or surfactant-laden droplets, the removal or redistribution of surfactant can indeed lead either to smooth motion of the contact line or to its partial arrest (see, e.g., \cite{BaSi}). In other words, in phase separation phenomena in presence of surfactant, mass conservation or slow exchange of surfactant may hinder domain evolution and induce metastable configurations. These observations are consistent with the two limiting regimes captured by our discrete model: for $\gamma>2$ the surfactant effectively evaporates and the crystal evolves by crystalline mean curvature flow, whereas for $\gamma<2$ the conservation of surfactant mass produces constrained, partially pinned motions.

\section{Formulation of the problem}\label{sec:formulation of the probelm}
\paragraph{Notation}
We denote by $\NN$ the sets of natural numbers starting from $1$. We denote by $\e\ZZ^2$ the square lattice with lattice spacing $\e>0$. We will call the points of $\e\ZZ^2$ with the letters $p, q$. When dealing with more than two points we name them as $p^i$, and we denote by $p^i = (p^i_1, p^i_2)$ their components. We write $p\preccurlyeq q$ to say that $p_1\le q_1$ and $p_2\le q_2$. The symmetric difference of two sets $A, B\subset \mathbb{R}^2$ is denoted by $A \triangle B$, their Hausdorff distance by $d_{\mathcal{H}}(A,B)$. For $i=1, 2$ we denote by $e_i$  the standard basis of $\rr^2$. Given $p\in \e\ZZ^2$ we define $\mathcal{N}(p):=\{p\pm\e e_i:\:i=1, 2\}$ so that, in particular, $\#\mathcal{N}(p) = 4$ for every $p\in\e\ZZ^2$. For \( x \in \mathbb{R} \), we denote by \( \lfloor x \rfloor \) and \( \lceil x \rceil \) the floor and ceiling of \( x \), respectively. When no ambiguity arises, we write \( [x] \) to denote the closest integer to \( x \). In the proofs we will often use the symbol $c$ to denote a generic constant, whose value may change form line to line. 

\paragraph{Setting of the problem} On the scaled square lattice $\e\ZZ^2$ we consider the set \[\A_\e\coloneqq\Bigl\{u\colon\e\ZZ^2\to\{\pm1,0\}\text{ such that }\#\{p\in\e\ZZ^2:u(p) = 0\}<+\infty\Bigl\}\]
and the family of energies $\E_\e:\A_\e\mapsto\rr$ defined as follows
\begin{equation}\label{eq:energy}
\E_\e(u) = \sum_{n.n.}\e(1-u(p)u(q)-k(1-(u(p)u(q))^2)),    
\end{equation}
where $n.n.$ means that the sum is taken over those $p, q\in \e\ZZ^2$ such that $|p-q| = \e$ and $k\in (1/3, 1)$.
Given $I, J\subset\e\ZZ^2$ and $u\in\A_\e$ we also use the following notation for a local version of \eqref{eq:energy}, namely
\begin{equation}\label{eq:local_energy}
\E_\e(u, I, J):=\sum
_{\substack{n.n.:\\p\in I, q\in J\\p\preccurlyeq  q}}\e\Bigl(1-u(p)u(q)-k\bigl(1-(u(p)u(q))^2\bigl)\Bigl),
\end{equation}
where the constraint $p\preccurlyeq q$ is needed to avoid double counting.
In the case $J = \e\ZZ^2$ we will simply write $\E_\e(u, I)$ in place of $\E_\e(u, I, \e\ZZ^2)$. In the following, for $I\subset\e\ZZ^2$ we set \begin{equation}\label{eq:set_A}
A_I:=\bigcup_{p\in I}Q_\e(p), \quad Q_\e(p):=p+\e[-1/2, 1/2]^2.
\end{equation} 
For $u^\e\in\A_\e,$ and for any sequence $(u^\e_j)_j\subset\A_\e$, we set
\begin{equation}\label{eq:I_j}
    \begin{split}
          &I_{u^\e}:=\{p\in\e\ZZ^2:u^\e(p)=1\},\;\;Z_{u^\e}:= \{p\in\e\ZZ^2:u^\e(p)=0\},\;\;A_{u^\e}:=A_{I_{u^\e}},\\
        &I_j^\e = I_{u_{j}^\e},\,Z^\e_j  = \Z_{u^\e_j},\, A^\e_j = A_{u_j^\e}.
    \end{split}
\end{equation}
Given $I\subset\e\ZZ^2$ and $s=1, +\infty$ we define the discrete $L^s$-distance of a point $p\in\e\ZZ^2$ from $I$ as
\begin{equation}\label{eq:distance_1}
d^\e_s(p, I) = \inf\{||p-q||_s:q\in I\},
\end{equation}
where $||p||_1 = |p_1|+|p_2|$ and $||p||_\infty = |p_1|\vee|p_2|$. We also introduce  the discrete $L^s$-distance of $p$ from the discrete boundary of $I$
\begin{equation}\label{eq:distance_2}
d^\e_s(p, \partial I) := \begin{cases}
    \inf\{||p-q||_s:q\in I\}&\text{ if }p\not\in I,\\
    \inf\{||p-q||_s:q\in\e\ZZ^2\setminus I\}&\text{ if }p\in I.
\end{cases}    
\end{equation}
We observe that, up to a set of negligible measure, we can extend this distance to all $x\in\rr^2$ by setting 
\[d^\e_s(x, \partial I) := d^\e_s(p, \partial I)  \text{ if }x\in Q_\e(p).\]
The discrete diameter of a set $I\subset\e\ZZ^2$ is given by $\diam(I):=\sup\{d^\e_\infty(p, q):p,\,q\in I\}$, while its exterior and interior boundary are
\[\partial^+ I\coloneqq\{p\in\e\ZZ^2\setminus I:d^\e_1(p, \partial I) = \e\},\quad\partial^-I\coloneqq\{p\in I:d^\e_1(p, \partial I) = \e\},\] respectively.

For $\gamma>0$ we introduce the dissipation 
$\dis_{\e, \gamma}:\:\A_\e\times\A_\e\to [0,+\infty]$ defined as \[\dis_{\e, \gamma}(u, w):=\dis^1_\e(\I_{u}, \I_w)+\e^\gamma\dis^0_\e(Z_{u}, Z_w),\] where $\dis^1_\e$ and $\dis^0_\e$ are defined as
\begin{equation}\label{eq:dis_1}
\dis_\e^1(\I_{u},\I_{w}):=\sum_{p\in\I_{u}\triangle \I_{w}}\e^2d^\e_1(p, \partial\I_{w}) = \int_{A_{\I_{u}}\triangle A_{\I_w}}d^\e_1(x, \partial A_{\I_w})\dx    
\end{equation}
 and 
\begin{equation}\label{eq:dis_0}
\dis^0_\e(\Z_{u}, \Z_{w}):=|\#\Z_{u}-\#\Z_{w}|.
\end{equation}
For $\tau,\gamma>0$ we  also introduce the total energy $\enF:\A_\e\times\A_\e\to\rr$  given by
\begin{equation}\label{eq:functional}
\enF(u, w)\coloneqq \E_\e(u)+\frac{1}{\tau}\dis_{\e, \gamma}(u, w).
\end{equation}
In the following we always consider $\tau = \zeta\e$, for a given parameter $\zeta>0$.

\begin{definition}\label{def:minimizing_movement}
    Given $u^\e_0\in\A_\e$, we say that a sequence $(u^\e_j)_j\subset\A_\e$ defined for every $j\in\NN\cup\{0\}$ is a minimizing movement with starting datum $u^\e_0$ if for every $j$ it holds 
\begin{equation}\label{eq:definition_u_j}    u^\e_{j+1}\in\text{argmin}\Bigl(u\mapsto \enF(u, u_j)\Bigl).
\end{equation}
In the Sections \ref{sec:gamma>2} and \ref{sec:gamma<2} we will always suppose that the set $I^\e_0$ associated to the initial datum $u^\e_0$ is the discretization of an octagon, in the sense of Definition \ref{def:discrete_octagon} in Section \ref{sec:gamma>2}.
\end{definition}

In the following proofs in order to deduce some conditions on the geometry of the sets $A_j$, it will often be useful to compare a minimizer $u_j$ with a new function $\tilde{u}_j$ obtained from $u_j$ changing its values only at a few points. For $\I\subset\e\ZZ^2$ we use the notation $u_j(\I)\mapsto\xi$ to indicate that for every point $p\in\I$ we replace the value of $u_j$ at $p$ with the new value $\xi\in\{\pm1, 0\}$, and that we do not modify $u_j$ anywhere else.

In all the pictures that follow, $Z_u$ is colored in blue, $\I_u$ is colored in red and $\{p:u(p)=-1\}$ is colored in green.

\subsection{Discrete sets}
\begin{definition}\label{def:discrete_segment}
    Let $p, q\in\e\ZZ^2$ and $n\in\NN\cup\{0\}$ be such that 
    \[q = p+n\e e_i\text{ for }i=1\text{ or } i=2,\quad\text{ or }\;q = p+n\e(e_1\pm e_2).\]
    We define the (discrete) segment connecting $p$ and $q$ as $\dseg{p}{q}:=\{p+s\e e_i:s = 0, \dots, n\}\subset\e\ZZ^2$ in the first case, and as $\dseg{p}{q}:=\{p+s\e(e_1\pm e_2):s = 0, \dots, n\}\subset\e\ZZ^2$ in the second case.
\end{definition}

\begin{definition}\label{def:discrete_path}

    We say that a finite ordered set $\pi = \{p^1, \dots, p^n\}\subset\e\ZZ^2$ is a \emph{weakly continuous discrete path} (resp. \emph{strongly continuous discrete path}) of \emph{length} $\text{len}(\pi) = n$ connecting $p^1$ and $p^n$ if $d^\e_\infty(p^i, p^{i+1}) = 1$ (resp. $d^\e_1(p^i, p^{i+1}) = 1$) for every $i = 1, \dots, n-1.$ We also denote by $\Gamma^w_{p, q}(\I)$ and $\Gamma^s_{p, q}(\I)$ the space of all paths $\pi$ connecting $p$ and $q$, which are weakly and strongly connected respectively, and such that $\pi\subset\I$.
\end{definition}

\begin{definition}\label{def:connection}
    We say that a set $\I\subset\e\ZZ^2$ is \emph{weakly connected} (resp. \emph{strongly connected}) if for every pair $p, q\in\I$ there exists $\pi\in\Gamma^w_{p, q}(I)$ (resp. $\pi\in\Gamma_{p, q}^s(I)$).
\end{definition}

\begin{definition}\label{def:convexity}
    We say that a set $\I\subset\e\ZZ^2$ is $horizontally$ (resp. $vertically$) $convex$ if for every $p\in\I,$ and every $n\in\NN$ such that the point $q := p+n\e e_1 (\text{resp.}\ q:= p+n\e e_2)$ belongs to $\I$, then the entire discrete segment $\dseg{p}{q}$ is contained in $\I$.
\end{definition}

\begin{definition}\label{def:staircase_set}
     We say that $\I\subset\e\ZZ^2$ is a \emph{staircase set} if it is strongly connected, and horizontally and vertically convex. We say that $I$ is a discrete rectangle if $A_I\subset\rr^2$ is a rectangle. Moreover we define $R_\I\subset\e\ZZ^2$ as the smallest discrete rectangle which contains $\I$, and we say that $\I$ is a \emph{non degenerate staircase set} if the set $R_\I\setminus\I=:T$ has four (not empty) strongly connected components $T_i, i=1, \dots, 4.$ We label the components $T_i$ clockwise setting $T_1$ to be the lowest left component (see Figure \ref{fig:staircasedef}). 
\begin{figure}[H]
    \centering
    \resizebox{0.40\textwidth}{!}{\tikzset{every picture/.style={line width=0.75pt}} 

\begin{tikzpicture}[x=0.75pt,y=0.75pt,yscale=-1,xscale=1]

\draw  [draw opacity=0] (250.33,50) -- (481.33,50) -- (481.33,280) -- (250.33,280) -- cycle ; \draw  [color={rgb, 255:red, 155; green, 155; blue, 155 }  ,draw opacity=0.6 ] (250.33,50) -- (250.33,280)(260.33,50) -- (260.33,280)(270.33,50) -- (270.33,280)(280.33,50) -- (280.33,280)(290.33,50) -- (290.33,280)(300.33,50) -- (300.33,280)(310.33,50) -- (310.33,280)(320.33,50) -- (320.33,280)(330.33,50) -- (330.33,280)(340.33,50) -- (340.33,280)(350.33,50) -- (350.33,280)(360.33,50) -- (360.33,280)(370.33,50) -- (370.33,280)(380.33,50) -- (380.33,280)(390.33,50) -- (390.33,280)(400.33,50) -- (400.33,280)(410.33,50) -- (410.33,280)(420.33,50) -- (420.33,280)(430.33,50) -- (430.33,280)(440.33,50) -- (440.33,280)(450.33,50) -- (450.33,280)(460.33,50) -- (460.33,280)(470.33,50) -- (470.33,280)(480.33,50) -- (480.33,280) ; \draw  [color={rgb, 255:red, 155; green, 155; blue, 155 }  ,draw opacity=0.6 ] (250.33,50) -- (481.33,50)(250.33,60) -- (481.33,60)(250.33,70) -- (481.33,70)(250.33,80) -- (481.33,80)(250.33,90) -- (481.33,90)(250.33,100) -- (481.33,100)(250.33,110) -- (481.33,110)(250.33,120) -- (481.33,120)(250.33,130) -- (481.33,130)(250.33,140) -- (481.33,140)(250.33,150) -- (481.33,150)(250.33,160) -- (481.33,160)(250.33,170) -- (481.33,170)(250.33,180) -- (481.33,180)(250.33,190) -- (481.33,190)(250.33,200) -- (481.33,200)(250.33,210) -- (481.33,210)(250.33,220) -- (481.33,220)(250.33,230) -- (481.33,230)(250.33,240) -- (481.33,240)(250.33,250) -- (481.33,250)(250.33,260) -- (481.33,260)(250.33,270) -- (481.33,270) ; \draw  [color={rgb, 255:red, 155; green, 155; blue, 155 }  ,draw opacity=0.6 ]  ;
\draw   (280.33,80) -- (450.33,80) -- (450.33,250) -- (280.33,250) -- cycle ;
\draw [color={rgb, 255:red, 208; green, 2; blue, 27 }  ,draw opacity=1 ][line width=0.75]    (320.33,80) -- (390.33,80) -- (390.33,90) -- (410.33,90) -- (410.33,110) -- (420.33,110) -- (420.33,120) -- (450.33,120) -- (450.33,160) -- (440.33,160) -- (440.33,190) -- (430.33,190) -- (430.33,200) -- (420.33,200) -- (420.33,220) -- (400.33,220) -- (400.33,230) -- (400.33,250) -- (330.33,250) -- (330.33,230) -- (320.33,230) -- (320.33,220) -- (310.33,220) -- (310.33,210) -- (300.33,210) -- (300.33,200) -- (290.33,200) -- (290.33,190) -- (280.33,190) -- (280.33,140) -- (290.33,140) -- (290.33,120) -- (310.33,120) -- (310.33,90) -- (320.33,90) -- cycle ;

\draw (362.33,153.4) node [anchor=north west][inner sep=0.75pt]  [font=\small]  {${\textcolor[rgb]{0.82,0.01,0.11}{I}}$};
\draw (425.33,90.0) node [anchor=north west][inner sep=0.75pt]  [font=\footnotesize]  {$T_{3}$};
\draw (425.33,220.4) node [anchor=north west][inner sep=0.75pt]  [font=\footnotesize]  {$T_{4}$};
\draw (287.33,220.4) node [anchor=north west][inner sep=0.75pt]  [font=\footnotesize]  {$T_{1}$};
\draw (287.33,90.0) node [anchor=north west][inner sep=0.75pt]  [font=\footnotesize]  {$T_{2}$};
\draw (442.33,53.4) node [anchor=north west][inner sep=0.75pt]  [font=\normalsize]  {$R_{I}$};

\end{tikzpicture}}
    \caption{In red an example of non degenerate staircase set.}
    \label{fig:staircasedef}
\end{figure}

    \noindent Finally, we denote by
    \begin{equation}
    \label{horizontal_slices}
        H^i = \dseg{p^{h, i}}{q^{h, i}},\,\quad \text{for}\ i=1 \dots, n_h,
    \end{equation}
    \begin{equation}
    \label{vertical_slices}
        V^i = \dseg{p^{v, i}}{q^{v, i}}\,\quad \text{for}\ i=1, \dots, n_v,
    \end{equation} 
    for some $n_h, n_v \in \NN$,
    the horizontal and vertical slices of $\I$, respectively. We label $H^i$ in ascending order starting from the lowest slice, i.e. writing $p^{i} = (p^{i}_1, p^{i}_2)$, it holds $p^{h, i}_2 = q^{h, i}_2$ for every $i$, and $p^{h,i}_2<p^{h, j}_2$ if $i<j$. We label the slices $V^i$ from left to right, so that $p^{v, i}_1<p^{v, j}_1$ if $i<j$. We also set \begin{align}\label{eq:parallel_sides_and_slices}
        \PP_1 := H^{1},\,\PP_2 := V^{1},\,\PP_3 := H^{n_h},\,\PP_4 := V^{n_v},\,
    \end{align}
    to denote the most ``external'' slices. 
We observe that the perimeter of $A_\I$ is \begin{equation}\label{eq:length_perimeter}
        Per(A_\I) = 2\e(n_h+n_v),
\end{equation}
while it holds
\begin{equation}\label{eq:length_external_boundary}
     \#\partial^+\I = \frac{Per(A_	I)}{\e}-\#\{p\in\ZZ^2\setminus\I:\#(\mathcal{N}(p)\cap\I) = 2\}.
\end{equation}
With a slight abuse of notation, for $I\subset \e\ZZ^2$ we write $Per(I)$ instead of $Per(A_I).$
\end{definition}

\section{Connection and inclusion of the minimizer}\label{sec:connection}
We consider the minimizing movement scheme defined by \eqref{def:minimizing_movement}. The goal of this section is to prove that $\I_{j+1}^\e$ is connected and contained into  $\I^\e_j$, where the sets $I_j^\e$ are given by \eqref{eq:I_j}. In order to prove such a result we need $I_j^\e$ to be a ``sufficiently regular set'' as explained in the next definition. From now on, in the rest of this article we will systematically omit the dependence on $\e$ when no confusion is possible; for instance we write $I_j$ in place of $I^\e_j$, and $(u_j)_j\subset\A_\e$ in place of $(u^\e_j)_j\subset\A_\e$.

\begin{definition}\label{def:large_staircase_set}
    Let $\bar{c}>0$ be a positive constant. We say that $(\I^\e)_\e,\,\e>0$ is a \emph{family of staircase sets of width $\bar{c}$} if for every $\e>0$ it holds that $\I^\e\subset\e\ZZ^2$ is a staircase set such that 
    \[\min\{\#\PP_i^\e:\:i=1, \dots, 4\}\ge \frac{\bar{c}}{\e}.\]
    We observe that if there exists a constant $\bar{R}>0$ such that $\sup_\e\diam(\I^\e)\le\bar{R}$, then there exists a constant $\tilde{c}$, which depends only on $\bar{c}$ and $\bar{R}$, such that for every $\e$ and for every pair of points $p, q\in\partial^-\I^\e$ it holds   \begin{equation}\label{eq:3_geodesic_distance}
    \inf\{\len(\pi):\:\pi\in\Gamma^w_{p, q}(\partial^-\I^\e)\}\le\tilde{c}\cdot\inf\{\len(\pi):\:\pi\in\Gamma^w_{p, q}(\I^\e)\}.
    \end{equation}
\end{definition}

\begin{definition}
\label{def: corner unit}
    Let $u\in\A_\e$. We say that $p\in \varepsilon\ZZ^2$ is a \emph{surfactant in the corner} if, up to rotation, it holds 
    \[u(p) = u(p+\e e_2) = u(p +\e e_1) = 0,\,\text{ and }u(p-\e e_1) = u(p-\e e_2)\neq 0.\]
    
    \noindent Moreover, given $u_j\in\A_\e,$ we define 
    \[\one_j:=\{p\in\partial^-\I_{j}:\#(\mathcal{N}(p)\cap\I_j) = 2\},\;\;\zero_j:=\{p\in\partial^+\I_j:\#(\mathcal{N}(p)\cap\I_j) = 2\}.\]
\end{definition}

\begin{remark}\label{rem:rectangle}
Let $G_i$ for $i=1,\dots,n$ be the weakly connected components of $I$, so that $\I = \cup_{i=1}^nG_i$. Then there exists $p\in\e\ZZ^2\setminus I$ such that $\#\bigl(\mathcal{N}(p)\cap\I\bigl)\ge 2$ if and only if there exists one component $G_i$ which is not a rectangle.
\end{remark}

\subsection{Preliminaries}
In this subsection we prove some auxiliary results which will be used several times in the following.


\begin{lemma}\label{lemma:shape_optimality}
    {\rm[Shape optimality]} Let $u_0,\,u_1\in\A_\e$ be such that $u_1$ is a minimizer of $\enF(\cdot, u_0)$. Then, for every $\e>0$ the following properties hold true.
    \begin{itemize}
        \item[1.1)] Let $p \in (\partial^+\I_1)\cap\I_0$ be a surfactant in the corner (Definition \ref{def: corner unit}) for $u_1$. Then, every weakly connected component of the set $\Z_1\cup \I_1$ is a discrete rectangle.
        \item[1.2)]If $p\in\zero_1\cap I_0$ then $u_1(p) = 0$.
    \end{itemize}
    
    Now let $\I$ be a strongly connected component of $\I_1$, suppose that $I_1\subset I_0$, that $I$ is a staircase set, and let $H^{i} = \dseg{p^{i}}{ q^{i}},\,i=1, \dots, n_h$ be the horizontal slices of $\I$. Assume that there exists $i\ge 2$ such that $p^{i}_1+2\e e_1\le p^{i-1}_1$ (so that in particular $\{p^i, p^i+\e e_1\}\subset H^i\cap\partial^-I$). Then (at least) one of the following holds:
    \begin{itemize}
        \item[(i)] $\dseg{p^i-\e e_2}{p^{i-1}-\e e_1}\subset\{u_0\neq 1\}$,
        \item[(ii)] $\dseg{p^i-\e e_2}{p^{i-1}-\e e_1}\subset\Z_1$ and $\mathcal{N}(p^{i-1}-\e e_1-\e e_2)\cap\Z_1 = \{p^{i-1}-\e e_1\}$,
        \item[(iii)] the weakly connected components of $\I_1\cup\Z_1$ are rectangles and $\partial^+I_1\subset Z_1$.
    \end{itemize}
    Finally suppose that case (ii) occurs, that the set $I_1\cup Z_1$ has only one strongly connected component and that $\dseg{p^i-\e e_2}{p^{i-1}-\e e_1}\not\subset\{u_0\neq 1\}$. Then it holds $i=2$.
\end{lemma}
\begin{proof}
    1.1) Let $\Z_1\cup\I_1 =: \cup_{i=1}^nG_i$, where $G_i$ are the weakly connected components of $\Z_1\cup\I_1$, and let $p$ be a surfactant in the corner for $u_1$. Suppose that there exists $\bar{i}$ such that $G_{\bar{i}}$ is not a rectangle. Then from Remark \ref{rem:rectangle} there exists $q\in\partial^+G_{\bar{i}}$ such that $\#(\mathcal{N}(q)\cap G_{\bar{i}})\ge 2.$ We define a competitor $\tilde{u}_1$ replacing $u_1(p)\mapsto 1$ and $u_1(q)= 0$. Since the dissipation of $\tilde{u}_1$ is strictly smaller than the dissipation of $u_1$, and since $\E_\e(\tilde{u}_1)\le \E_\e(u_1)$, we conclude that $u_1$ cannot be a minimizer.\\
    1.2) If by contradiction we assume that $u_1(p) = -1$ then, replacing $u_1(p) \mapsto 1$, we do not increase the energy, and we strictly decrease the dissipation.\vspace{5pt}\\
    We now prove the second part of the Lemma. If $u_0(p^{i-1}-\e e_1) \neq 1$ then, since $\I_0$ is a staircase set (and in particular horizontally convex), and since $\I_1\subset \I_0$, we deduce that $\dseg{p^i-\e e_2}{p^{i-1}-\e e_1}\subset\{u_0\neq 1\}$, which is (i). Suppose instead that $u_0(p^{i-1}-\e e_1) = 1.$ Since $p^{i-1}-\e e_1\in\zero_1\cap I_0$, then, in view of 1.2), it holds $u_1(p^{i-1}-\e e_1) = 0.$ If there exists $\bar{p}\in\dseg{p^i-\e e_2}{p^{i-1}-\e e_1}$ such that $u_1(\bar{p}) = -1,$ then, up to possibly changing $\bar{p}$, we may suppose that $u_1(\bar{p}+\e e_1) = 0$. Replacing $u_1(p^{i-1}-\e e_1)\mapsto1$ and $u_1(\bar{p})\mapsto 0$ (see Figure \ref{fig:(ii)_shape_optimality}) we strictly reduce the dissipation, and we do not increase the energy, which is in contradiction with the minimality of $u_1$.

    \begin{figure}[H]
         \centering
  \resizebox{0.50\textwidth}{!}{\tikzset{every picture/.style={line width=0.75pt}} 

\begin{tikzpicture}[x=0.75pt,y=0.75pt,yscale=-1,xscale=1]

\draw  [color={rgb, 255:red, 208; green, 2; blue, 27 }  ,draw opacity=1 ][fill={rgb, 255:red, 208; green, 2; blue, 27 }  ,fill opacity=1 ] (190.33,100.33) -- (250.33,100.33) -- (250.33,110.33) -- (190.33,110.33) -- cycle ;
\draw  [color={rgb, 255:red, 208; green, 2; blue, 27 }  ,draw opacity=1 ][fill={rgb, 255:red, 208; green, 2; blue, 27 }  ,fill opacity=1 ] (360.33,100.33) -- (410.33,100.33) -- (410.33,110.33) -- (360.33,110.33) -- cycle ;
\draw  [color={rgb, 255:red, 208; green, 2; blue, 27 }  ,draw opacity=1 ][fill={rgb, 255:red, 208; green, 2; blue, 27 }  ,fill opacity=1 ] (250.33,120.33) -- (270.33,120.33) -- (270.33,130.33) -- (250.33,130.33) -- cycle ;
\draw  [color={rgb, 255:red, 208; green, 2; blue, 27 }  ,draw opacity=1 ][fill={rgb, 255:red, 208; green, 2; blue, 27 }  ,fill opacity=1 ] (420.33,120.33) -- (440.33,120.33) -- (440.33,130.33) -- (420.33,130.33) -- cycle ;
\draw  [color={rgb, 255:red, 208; green, 2; blue, 27 }  ,draw opacity=1 ][fill={rgb, 255:red, 208; green, 2; blue, 27 }  ,fill opacity=1 ] (170.33,80.33) -- (180.33,80.33) -- (180.33,90.33) -- (170.33,90.33) -- cycle ;
\draw  [color={rgb, 255:red, 208; green, 2; blue, 27 }  ,draw opacity=1 ][fill={rgb, 255:red, 208; green, 2; blue, 27 }  ,fill opacity=1 ] (180.33,90.33) -- (190.33,90.33) -- (190.33,100.33) -- (180.33,100.33) -- cycle ;
\draw  [color={rgb, 255:red, 208; green, 2; blue, 27 }  ,draw opacity=1 ][fill={rgb, 255:red, 208; green, 2; blue, 27 }  ,fill opacity=1 ] (240.33,110.33) -- (250.33,110.33) -- (250.33,120.33) -- (240.33,120.33) -- cycle ;
\draw  [color={rgb, 255:red, 208; green, 2; blue, 27 }  ,draw opacity=1 ][fill={rgb, 255:red, 208; green, 2; blue, 27 }  ,fill opacity=1 ] (340.33,80.33) -- (350.33,80.33) -- (350.33,90.33) -- (340.33,90.33) -- cycle ;
\draw  [color={rgb, 255:red, 208; green, 2; blue, 27 }  ,draw opacity=1 ][fill={rgb, 255:red, 208; green, 2; blue, 27 }  ,fill opacity=1 ] (350.33,90.33) -- (360.33,90.33) -- (360.33,100.33) -- (350.33,100.33) -- cycle ;
\draw  [color={rgb, 255:red, 208; green, 2; blue, 27 }  ,draw opacity=1 ][fill={rgb, 255:red, 208; green, 2; blue, 27 }  ,fill opacity=1 ] (400.33,110.33) -- (420.33,110.33) -- (420.33,120.33) -- (400.33,120.33) -- cycle ;
\draw  [color={rgb, 255:red, 74; green, 144; blue, 226 }  ,draw opacity=1 ][fill={rgb, 255:red, 74; green, 144; blue, 226 }  ,fill opacity=1 ] (190.33,110.33) -- (220.33,110.33) -- (220.33,120.33) -- (190.33,120.33) -- cycle ;
\draw  [color={rgb, 255:red, 74; green, 144; blue, 226 }  ,draw opacity=1 ][fill={rgb, 255:red, 74; green, 144; blue, 226 }  ,fill opacity=1 ] (230.33,110.33) -- (240.33,110.33) -- (240.33,120.33) -- (230.33,120.33) -- cycle ;
\draw  [color={rgb, 255:red, 126; green, 211; blue, 33 }  ,draw opacity=1 ][fill={rgb, 255:red, 126; green, 211; blue, 33 }  ,fill opacity=1 ] (220.33,110.33) -- (230.33,110.33) -- (230.33,120.33) -- (220.33,120.33) -- cycle ;
\draw  [color={rgb, 255:red, 74; green, 144; blue, 226 }  ,draw opacity=1 ][fill={rgb, 255:red, 74; green, 144; blue, 226 }  ,fill opacity=1 ] (360.33,110.33) -- (400.33,110.33) -- (400.33,120.33) -- (360.33,120.33) -- cycle ;
\draw  [draw opacity=0] (140.33,60.33) -- (460.44,60.33) -- (460.44,151.33) -- (140.33,151.33) -- cycle ; \draw  [color={rgb, 255:red, 155; green, 155; blue, 155 }  ,draw opacity=0.77 ] (140.33,60.33) -- (140.33,151.33)(150.33,60.33) -- (150.33,151.33)(160.33,60.33) -- (160.33,151.33)(170.33,60.33) -- (170.33,151.33)(180.33,60.33) -- (180.33,151.33)(190.33,60.33) -- (190.33,151.33)(200.33,60.33) -- (200.33,151.33)(210.33,60.33) -- (210.33,151.33)(220.33,60.33) -- (220.33,151.33)(230.33,60.33) -- (230.33,151.33)(240.33,60.33) -- (240.33,151.33)(250.33,60.33) -- (250.33,151.33)(260.33,60.33) -- (260.33,151.33)(270.33,60.33) -- (270.33,151.33)(280.33,60.33) -- (280.33,151.33)(290.33,60.33) -- (290.33,151.33)(300.33,60.33) -- (300.33,151.33)(310.33,60.33) -- (310.33,151.33)(320.33,60.33) -- (320.33,151.33)(330.33,60.33) -- (330.33,151.33)(340.33,60.33) -- (340.33,151.33)(350.33,60.33) -- (350.33,151.33)(360.33,60.33) -- (360.33,151.33)(370.33,60.33) -- (370.33,151.33)(380.33,60.33) -- (380.33,151.33)(390.33,60.33) -- (390.33,151.33)(400.33,60.33) -- (400.33,151.33)(410.33,60.33) -- (410.33,151.33)(420.33,60.33) -- (420.33,151.33)(430.33,60.33) -- (430.33,151.33)(440.33,60.33) -- (440.33,151.33)(450.33,60.33) -- (450.33,151.33)(460.33,60.33) -- (460.33,151.33) ; \draw  [color={rgb, 255:red, 155; green, 155; blue, 155 }  ,draw opacity=0.77 ] (140.33,60.33) -- (460.44,60.33)(140.33,70.33) -- (460.44,70.33)(140.33,80.33) -- (460.44,80.33)(140.33,90.33) -- (460.44,90.33)(140.33,100.33) -- (460.44,100.33)(140.33,110.33) -- (460.44,110.33)(140.33,120.33) -- (460.44,120.33)(140.33,130.33) -- (460.44,130.33)(140.33,140.33) -- (460.44,140.33)(140.33,150.33) -- (460.44,150.33) ; \draw  [color={rgb, 255:red, 155; green, 155; blue, 155 }  ,draw opacity=0.77 ]  ;
\draw    (170.33,80.33) -- (170.33,90.33) -- (180.33,90.33) -- (180.33,100.33) -- (190.33,100.33) -- (190.33,110.33) -- (240.33,110.33) -- (240.33,120.33) -- (250.33,120.33) -- (250.33,130.33) -- (270.33,130.33) ;
\draw    (340.33,80.33) -- (340.33,90.33) -- (350.33,90.33) -- (350.33,100.33) -- (360.33,100.33) -- (360.33,110.33) -- (400.33,110.33) -- (400.33,120.33) -- (420.33,120.33) -- (420.33,130.33) -- (440.33,130.33) ;
\draw    (280.33,110.33) -- (328.33,110.33) ;
\draw [shift={(330.33,110.33)}, rotate = 180] [color={rgb, 255:red, 0; green, 0; blue, 0 }  ][line width=0.75]    (10.93,-3.29) .. controls (6.95,-1.4) and (3.31,-0.3) .. (0,0) .. controls (3.31,0.3) and (6.95,1.4) .. (10.93,3.29)   ;

\draw (192.33,96.73) node [anchor=north west][inner sep=0.75pt]  [font=\scriptsize]  {$p^{i}$};
\draw (242.33,106.73) node [anchor=north west][inner sep=0.75pt]  [font=\scriptsize]  {$p^{i-1}$};
\draw (362.33,95.73) node [anchor=north west][inner sep=0.75pt]  [font=\scriptsize]  {$p^{i}$};
\draw (412.33,107.73) node [anchor=north west][inner sep=0.75pt]  [font=\scriptsize]  {$p^{i-1}$};
\draw (222.33,111.73) node [anchor=north west][inner sep=0.75pt]  [font=\scriptsize]  {$\overline{p}$};
\draw (391.33,111.73) node [anchor=north west][inner sep=0.75pt]  [font=\scriptsize]  {$\overline{p}$};
\draw (202.33,70.73) node [anchor=north west][inner sep=0.75pt]    {$u_{1}$};

\end{tikzpicture}}
    \caption{$\dseg{p^i-\e e_2}{p^{i-1}-\e e_1}\subset\Z_1.$ The configuration on the right is the one that in (ii) of Lemma \ref{lemma:shape_optimality} strictly reduces the dissipation without increasing the energy.} \label{fig:(ii)_shape_optimality}
     \end{figure}
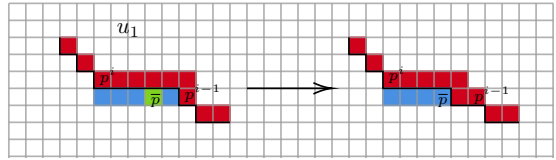    
    Therefore if $p^{i-1}-\e e_1\in I_0$ we have $\dseg{p^i-\e e_2}{p^{i-1}-\e e_1}\subset\Z_1.$
    Now, if moreover $u_1(p^{i-1}-\e e_1+\e e_2) = 0$, then case (iii) occurs. Indeed we have $p^{i-1}-\e e_1\in\zero_1\cap I_0$ and therefore from $1.1)$ we deduce that the weakly connected components of $\I_1\cup\Z_1$ are rectangles. Moreover, if by contradiction it were $\partial^+I_1\not\subset Z_1$, then there would exist $p\in\partial^+I_1$ such that $u_1(p) = -1$. But then the competitor $\tilde{u}_1$ obtained by replacing $u_1(p)\mapsto 0$ and $u_1(p^{i-1}-\e e_1)\mapsto 1$ has strictly lower energy and strictly lower dissipation with respect to $u_1$, which contradicts the minimality of $u_1$. Therefore we conclude $\partial^+I_1\subset Z_1$, which completes the proof of (iii). Finally, if $u_1(p^{i-1}-\e e_1-\e e_2) = -1$, and if there exists 
    \[q\in\mathcal{N}(p^{i-1}-\e e_1-\e e_2)\cap(\Z_1\setminus\{p^{i-1}-\e e_1\}),\]
    then we could replace $u_1(p^{i-1}-\e e_1)\mapsto 1$ and $u_1(p^{i-1}-\e e_1-\e e_2)\mapsto 0$  (see Figure \ref{fig:(ii)_1}) strictly reducing the dissipation of $u_1,$ and not increasing its energy, which is a contradiction.
     \begin{figure}[H]
         \centering
  \resizebox{0.50\textwidth}{!}{\tikzset{every picture/.style={line width=0.75pt}} 

\begin{tikzpicture}[x=0.75pt,y=0.75pt,yscale=-1,xscale=1]

\draw  [color={rgb, 255:red, 74; green, 144; blue, 226 }  ,draw opacity=1 ][fill={rgb, 255:red, 74; green, 144; blue, 226 }  ,fill opacity=1 ] (400.33,120.33) -- (410.33,120.33) -- (410.33,130.33) -- (400.33,130.33) -- cycle ;
\draw  [color={rgb, 255:red, 74; green, 144; blue, 226 }  ,draw opacity=1 ][fill={rgb, 255:red, 74; green, 144; blue, 226 }  ,fill opacity=1 ] (410.33,120.33) -- (420.33,120.33) -- (420.33,130.33) -- (410.33,130.33) -- cycle ;
\draw  [color={rgb, 255:red, 74; green, 144; blue, 226 }  ,draw opacity=1 ][fill={rgb, 255:red, 74; green, 144; blue, 226 }  ,fill opacity=1 ] (239.33,120.33) -- (249.33,120.33) -- (249.33,130.33) -- (239.33,130.33) -- cycle ;
\draw  [color={rgb, 255:red, 126; green, 211; blue, 33 }  ,draw opacity=1 ][fill={rgb, 255:red, 126; green, 211; blue, 33 }  ,fill opacity=1 ] (230.33,120.33) -- (240.33,120.33) -- (240.33,130.33) -- (230.33,130.33) -- cycle ;
\draw  [color={rgb, 255:red, 208; green, 2; blue, 27 }  ,draw opacity=1 ][fill={rgb, 255:red, 208; green, 2; blue, 27 }  ,fill opacity=1 ] (190.33,100.33) -- (250.33,100.33) -- (250.33,110.33) -- (190.33,110.33) -- cycle ;
\draw  [color={rgb, 255:red, 208; green, 2; blue, 27 }  ,draw opacity=1 ][fill={rgb, 255:red, 208; green, 2; blue, 27 }  ,fill opacity=1 ] (360.33,100.33) -- (410.33,100.33) -- (410.33,110.33) -- (360.33,110.33) -- cycle ;
\draw  [color={rgb, 255:red, 208; green, 2; blue, 27 }  ,draw opacity=1 ][fill={rgb, 255:red, 208; green, 2; blue, 27 }  ,fill opacity=1 ] (250.33,120.33) -- (270.33,120.33) -- (270.33,130.33) -- (250.33,130.33) -- cycle ;
\draw  [color={rgb, 255:red, 208; green, 2; blue, 27 }  ,draw opacity=1 ][fill={rgb, 255:red, 208; green, 2; blue, 27 }  ,fill opacity=1 ] (420.33,120.33) -- (440.33,120.33) -- (440.33,130.33) -- (420.33,130.33) -- cycle ;
\draw  [color={rgb, 255:red, 208; green, 2; blue, 27 }  ,draw opacity=1 ][fill={rgb, 255:red, 208; green, 2; blue, 27 }  ,fill opacity=1 ] (170.33,80.33) -- (180.33,80.33) -- (180.33,90.33) -- (170.33,90.33) -- cycle ;
\draw  [color={rgb, 255:red, 208; green, 2; blue, 27 }  ,draw opacity=1 ][fill={rgb, 255:red, 208; green, 2; blue, 27 }  ,fill opacity=1 ] (180.33,90.33) -- (190.33,90.33) -- (190.33,100.33) -- (180.33,100.33) -- cycle ;
\draw  [color={rgb, 255:red, 208; green, 2; blue, 27 }  ,draw opacity=1 ][fill={rgb, 255:red, 208; green, 2; blue, 27 }  ,fill opacity=1 ] (240.33,110.33) -- (250.33,110.33) -- (250.33,120.33) -- (240.33,120.33) -- cycle ;
\draw  [color={rgb, 255:red, 208; green, 2; blue, 27 }  ,draw opacity=1 ][fill={rgb, 255:red, 208; green, 2; blue, 27 }  ,fill opacity=1 ] (340.33,80.33) -- (350.33,80.33) -- (350.33,90.33) -- (340.33,90.33) -- cycle ;
\draw  [color={rgb, 255:red, 208; green, 2; blue, 27 }  ,draw opacity=1 ][fill={rgb, 255:red, 208; green, 2; blue, 27 }  ,fill opacity=1 ] (350.33,90.33) -- (360.33,90.33) -- (360.33,100.33) -- (350.33,100.33) -- cycle ;
\draw  [color={rgb, 255:red, 208; green, 2; blue, 27 }  ,draw opacity=1 ][fill={rgb, 255:red, 208; green, 2; blue, 27 }  ,fill opacity=1 ] (400.33,110.33) -- (420.33,110.33) -- (420.33,120.33) -- (400.33,120.33) -- cycle ;
\draw  [color={rgb, 255:red, 74; green, 144; blue, 226 }  ,draw opacity=1 ][fill={rgb, 255:red, 74; green, 144; blue, 226 }  ,fill opacity=1 ] (190.33,110.33) -- (220.33,110.33) -- (220.33,120.33) -- (190.33,120.33) -- cycle ;
\draw  [color={rgb, 255:red, 74; green, 144; blue, 226 }  ,draw opacity=1 ][fill={rgb, 255:red, 74; green, 144; blue, 226 }  ,fill opacity=1 ] (230.33,110.33) -- (240.33,110.33) -- (240.33,120.33) -- (230.33,120.33) -- cycle ;
\draw  [color={rgb, 255:red, 74; green, 144; blue, 226 }  ,draw opacity=1 ][fill={rgb, 255:red, 74; green, 144; blue, 226 }  ,fill opacity=1 ] (220.33,110.33) -- (230.33,110.33) -- (230.33,120.33) -- (220.33,120.33) -- cycle ;
\draw  [color={rgb, 255:red, 74; green, 144; blue, 226 }  ,draw opacity=1 ][fill={rgb, 255:red, 74; green, 144; blue, 226 }  ,fill opacity=1 ] (360.33,110.33) -- (400.33,110.33) -- (400.33,120.33) -- (360.33,120.33) -- cycle ;
\draw  [draw opacity=0] (140.33,60.33) -- (450.44,60.33) -- (450.44,151.33) -- (140.33,151.33) -- cycle ; \draw  [color={rgb, 255:red, 155; green, 155; blue, 155 }  ,draw opacity=0.77 ] (140.33,60.33) -- (140.33,151.33)(150.33,60.33) -- (150.33,151.33)(160.33,60.33) -- (160.33,151.33)(170.33,60.33) -- (170.33,151.33)(180.33,60.33) -- (180.33,151.33)(190.33,60.33) -- (190.33,151.33)(200.33,60.33) -- (200.33,151.33)(210.33,60.33) -- (210.33,151.33)(220.33,60.33) -- (220.33,151.33)(230.33,60.33) -- (230.33,151.33)(240.33,60.33) -- (240.33,151.33)(250.33,60.33) -- (250.33,151.33)(260.33,60.33) -- (260.33,151.33)(270.33,60.33) -- (270.33,151.33)(280.33,60.33) -- (280.33,151.33)(290.33,60.33) -- (290.33,151.33)(300.33,60.33) -- (300.33,151.33)(310.33,60.33) -- (310.33,151.33)(320.33,60.33) -- (320.33,151.33)(330.33,60.33) -- (330.33,151.33)(340.33,60.33) -- (340.33,151.33)(350.33,60.33) -- (350.33,151.33)(360.33,60.33) -- (360.33,151.33)(370.33,60.33) -- (370.33,151.33)(380.33,60.33) -- (380.33,151.33)(390.33,60.33) -- (390.33,151.33)(400.33,60.33) -- (400.33,151.33)(410.33,60.33) -- (410.33,151.33)(420.33,60.33) -- (420.33,151.33)(430.33,60.33) -- (430.33,151.33)(440.33,60.33) -- (440.33,151.33)(450.33,60.33) -- (450.33,151.33) ; \draw  [color={rgb, 255:red, 155; green, 155; blue, 155 }  ,draw opacity=0.77 ] (140.33,60.33) -- (450.44,60.33)(140.33,70.33) -- (450.44,70.33)(140.33,80.33) -- (450.44,80.33)(140.33,90.33) -- (450.44,90.33)(140.33,100.33) -- (450.44,100.33)(140.33,110.33) -- (450.44,110.33)(140.33,120.33) -- (450.44,120.33)(140.33,130.33) -- (450.44,130.33)(140.33,140.33) -- (450.44,140.33)(140.33,150.33) -- (450.44,150.33) ; \draw  [color={rgb, 255:red, 155; green, 155; blue, 155 }  ,draw opacity=0.77 ]  ;
\draw    (170.33,80.33) -- (170.33,90.33) -- (180.33,90.33) -- (180.33,100.33) -- (190.33,100.33) -- (190.33,110.33) -- (240.33,110.33) -- (240.33,120.33) -- (250.33,120.33) -- (250.33,130.33) -- (270.33,130.33) ;
\draw    (340.33,80.33) -- (340.33,90.33) -- (350.33,90.33) -- (350.33,100.33) -- (360.33,100.33) -- (360.33,110.33) -- (400.33,110.33) -- (400.33,120.33) -- (420.33,120.33) -- (420.33,130.33) -- (440.33,130.33) ;
\draw    (280.33,110.33) -- (328.33,110.33) ;
\draw [shift={(330.33,110.33)}, rotate = 180] [color={rgb, 255:red, 0; green, 0; blue, 0 }  ][line width=0.75]    (10.93,-3.29) .. controls (6.95,-1.4) and (3.31,-0.3) .. (0,0) .. controls (3.31,0.3) and (6.95,1.4) .. (10.93,3.29)   ;

\draw (192.33,96.73) node [anchor=north west][inner sep=0.75pt]  [font=\scriptsize]  {$p^{i}$};
\draw (242.33,106.73) node [anchor=north west][inner sep=0.75pt]  [font=\scriptsize]  {$p^{i-1}$};
\draw (362.33,95.73) node [anchor=north west][inner sep=0.75pt]  [font=\scriptsize]  {$p^{i}$};
\draw (412.33,107.73) node [anchor=north west][inner sep=0.75pt]  [font=\scriptsize]  {$p^{i-1}$};
\draw (242.33,121.73) node [anchor=north west][inner sep=0.75pt]  [font=\scriptsize]  {$q$};
\draw (412.33,121.73) node [anchor=north west][inner sep=0.75pt]  [font=\scriptsize]  {$q$};
\draw (192.33,65.73) node [anchor=north west][inner sep=0.75pt]    {$ \begin{array}{l}
u_{1}\\
\end{array}$};

\end{tikzpicture}}
    \caption{$\mathcal{N}(p^{i-1}-\e e_1-\e e_2)\cap Z_1=\{p^{i-1}-\e e_1\}$. The configuration on the right is the one that in (ii) of Lemma \ref{lemma:shape_optimality} strictly reduces the dissipation without increasing the energy.} \label{fig:(ii)_1}
     \end{figure}
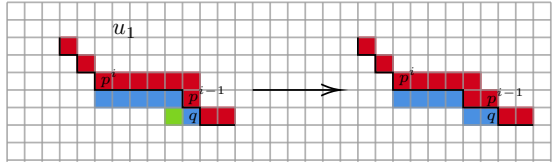
    We are left to prove the final statement of the Lemma. Suppose by contradiction that $i\ge 3$. If $u_1(p^{i-2}-\e e_1)\neq 0$ then by replacing $u_1(p^{i-1}-\e e_1)\mapsto 1$ and $u_1(p^{i-2}-\e e_1)\mapsto 0$ we strictly reduce both the energy and the dissipation, which is a contradiction. We may suppose therefore that $u_1(p^{i-2}-\e e_1) = 0$. Since (ii) holds, we have $u_1(p^{i-1}-\e e_2) = -1$, therefore there exists a point $p'\in\dseg{p^{i-1}-\e e_2}{p^{i-2}-\e e_1}$ such that $u_1(p') = -1$ and $u_1(p'+\e e_1) = 0$. By replacing $u_1(p^{i-1}-\e e_1)\mapsto 1$ and  $u_1(p')\mapsto 0$ we do not increase the energy of $u_1$, and we strictly decrease its dissipation, which is a contradiction .
\end{proof}

We recall that if $u\in\A_\e$ and $p \in \e\ZZ^2$ is such that $u(p)=0$, then the interaction between $p$ and $q\in\mathcal{N}(p)$ costs $1-k$ regardless of the value $u(q).$ This property allows to compute the energy of the set $\Z_u$ in terms of its mass and its perimeter, as shown in the next Lemma. We recall that for $I\subset\e\ZZ^2$ we set $Per(I) = Per(A_I).$
\begin{lemma}\label{lemma:energy_surfactant}
    {\rm[Energy of a set of surfactant]} Let $u\in A_\e$. The following two properties hold.
    \begin{itemize}
        \item[(i)] $\E_\e(u, \Z_u) = 2\e(1-k)\#\Z_u+1/2(1-k)Per(\Z_u).$
        \item[(ii)] Assume further that the set $\Z_u\cup\I_u$ has a strongly connected component $G$ such that $\partial^-G\subset\Z_u$. Let $\I_u\cap G = \cup_{i=1}^nB_i,$ where $B_i$ are the strongly connected components of $I_u\cap G$. Then
        \[\E_\e(u, G)=  2\e(1-k)\#(G\cap \Z_u)+1/2(1-k)\bigl(Per(G)+\sum_{i=1}^nPer(B_i)\bigl).\]
    \end{itemize}
\end{lemma}
\begin{proof}
    For ease of notation, we set
\[f_u(p, q) = \begin{cases}
    0&\text{ if }u(p)u(q)=1,\\
    2&\text{ if }u(p) u(q) = -1,\\
    1-k &\text{ if }u(p) u(q) = 0,
\end{cases}\]
    so that $\E_\e(u) = \sum_{n.n.}\e f_u(p,q).$ We observe that
    \[\E_\e(u, \Z_u) = \sum_{p\in \Z_u}\e f_u(p,p+\e e_1)+\e f_u(p,p+\e e_2)+\sum_{\substack{p\in\partial^-\Z_u\\p-\e e_1\not\in \Z_u}}\e f_u(p,p-\e e_1)+\sum_{\substack{p\in\partial^-\Z_u\\p-\e e_2\not\in \Z_u}}\e f_u(p,p-\e e_2)\]
    and that
    \[\sum_{\substack{p\in\partial^-\Z_u\\p-\e e_i\not\in \Z_u}}\e f_u(p,p-\e e_i) = \sum_{\substack{p\in\partial^-\Z_u\\p+\e e_i\not\in \Z_u}}\e f_u(p,p+\e e_i)\;\;\text{ for }i=1, 2.\]
    Claim (i) follows since 
    \[(1-k)Per(\Z_u) = \Bigl(\sum_{\substack{i=1, 2,\;p\in\partial^-\Z_u\\p-\e e_i\not\in \Z_u}}\e f_u(p,p-\e e_i)+\sum_{\substack{i=1, 2,\;p\in\partial^-\Z_u\\p+\e e_i\not\in \Z_u}}\e f_u(p,p+\e e_i)\Bigl).\]
    The second claim is a straightforward consequence of the first one, since $f_u(p,q) = 0$ if $u(p) = u(q) = 1, $ and since 
    \[Per(G\cap Z_u) = Per(G)+Per(I_u\cap G) = Per(G)+\sum_{i=1}^n Per(B_i).\]
\end{proof}
We conclude this subsection with the following Lemma.
\begin{lemma}\label{lemma:surfactant_placement}
     Let $u_0,\,u_1\in\A_\e$ be such that $u_1$ is a minimizer of $\enF(\cdot, u_0)$, and let $\mu\in(0, 1]$ be independent on $\e$. Suppose that $\I_1\subset\I_0$ and that for every point $a\in\partial^-\I_1$ it holds $d^\e_1(a, \partial\I_0)\le\e^\mu$. Then, there exists $\bar{\e}$ depending only on $\mu$ such that for every $\e\le\bar{\e}$, for every pair of points $p,\,q$ such that $u_1(p) = -1$ and $u_1(q) = 0$ it holds $\#\bigl(\nn(p)\cap\I_1\bigl)\le\#\bigl(\nn(q)\cap\I_1\bigl).$
\end{lemma}
\begin{proof}
    We argue by contradiction. Let $p,q$ be such that $u_1(p)=-1$ and $u_1(q)=0$. Let $s_p:=\#\bigl(\mathcal{N}(p)\cap\I_1\bigl)\ge 1$ and  $s_q:=\#(\mathcal{N}(q)\cap\I_1)<s_p$. Consider the competitor $\tilde{u}_1$ obtained replacing $u_1(p)\mapsto 0$ and $u_1(q)\mapsto -1.$ It holds \begin{equation}\label{eq:3.qualcosa}
        \E_\e(\tilde{u}_1) = \E_\e(u_1) -2\e s_p+\e(1-k)\#(\mathcal{N}(p)\cap\{u_1\neq 0\})+2\e s_q-\e(1-k)\#(\mathcal{N}(q)\cap\{u_1\neq0\}).
    \end{equation}
    Therefore if $s_p\ge 2$ and $s_q=0$ or $s_q=1,$ then 
    \[\E_\e(\tilde{u}_1)-\E_\e(u_1)\le -2\e(s_p-s_q)+4\e(1-k)-s_q\e(1-k)<0\]
     which contradicts the minimality of $u_1$. If instead $s_p=1,\,s_q=0$, and $\#(\mathcal{N}(p)\cap\{u_1\neq 0\})-\#(\mathcal{N}(q)\cap\{u_1\neq0\})<4$ then we find
    \[ -2\e +\e(1-k)\Bigl(\#(\mathcal{N}(p)\cap\{u_1\neq 0\})-\#(\mathcal{N}(q)\cap\{u_1\neq0\})\Bigl)<0.\]
    We are left to consider the case $s_p=1,\,s_q=0,\,\#(\mathcal{N}(p)\cap\{u_1\neq 0\})=4$ and $\#(\mathcal{N}(q)\cap\{u_1\neq 0\})= 0.$ Without loss of generality we can suppose that $u_1(p-\e e_2) = 1$. We observe that 
    \[u_1(p-\e e_2-\e e_1) = u_1(p-\e e_2+\e e_1) = u_1(p-2\e e_2) = 1.\]
    Indeed otherwise, since by assumption it holds $d^\e_1(p-\e e_2, \partial\I_0)\le\e^\mu$, then replacing $u_1(p-\e e_2)\mapsto 0$ and $u_1(q)\mapsto -1$ we would reduce the energy by at least $\e(1-k)$, and increase the dissipation by at most $c\e^{\mu+1}$, for some positive constant $c$ independent on $\e.$ Since $\e(1-k)>c\e^{\mu+1}$ this is a contradiction. In particular $p-\e e_2\pm\e e_1\in\partial^-\I_1$, and therefore by assumption $d^\e_1(p-\e e_2\pm\e e_1, \partial I_0)\le\e^\mu$.  Iterating the argument we find that for every $n\in\ZZ$  it holds $u_1(p+n\e e_1)= -1$, and $u_1(p-\e e_2 +n\e e_1)= 1$, that is $p-\e e_2 +n\e e_1\in\partial^-\I_1$. This contradicts $\I_1\subset\I_0$, since $\I_0$ is bounded.
\end{proof}
\begin{remark}\label{rem:blu_on_the_boundary}
A direct consequence of the previous Lemma, which we use in the following, is that if $\#\Z_1\le \#\partial^+\I_1$ then $\Z_1\subset\partial^+\I_1.$ \end{remark}

\subsection{Connectedness}\label{subsec:connectedness}
In all the statements of this section we assume that $u_0,\,u_1\in\A_\e$, and that $u_1$ is a minimizer of $\enF(\cdot, u_0)$. We require that the family $(I_0^\e)_\e$ satisfies the following assumption:
\begin{equation*}
\tag{H}\label{H}
\parbox{0.9\textwidth}{
there exist constants $\overline{R}$, $\overline{c}$ such that $(I_0^\varepsilon)_\varepsilon$ are staircase sets of width $\overline{c}$, all contained in a ball of radius $\overline{R}$, as in Definition~\ref{def:large_staircase_set}.
}
\end{equation*}
In what follows we say that $u_0$ satisfies \eqref{H} if $(I^\e_0)_\e$ satisfies \eqref{H}. Under this assumption, we prove that $\I^\e_1$ is a strongly connected staircase set contained in $\I^\e_0$ for every $\e$ sufficiently small. We shall omit the dependence on $\e$, whenever it is clear from the context.

We give a quick outline of the rest of the section. First of all, in Proposition \ref{prop:horizontal_convexity} we show that, as long as $\I_j$ is a staircase set, then $\I_{j+1}$ is a staircase set, and $\I_{j+1}\subset \I_j$.
Then in Proposition \ref{prop:hausdorff_distance_from_boundary} we show that for every $\mu<1/4$ we have $\I_0\setminus \{p\in\I_0:d^\e_1(p, \partial\I_0)\le\e^\mu\}\subset\I_1$, in particular the boundaries of $\I_0$ and $\I_1$ are uniformly close to each other. As a consequence, for $\e$ sufficiently small, any minimizer $u_1$ can we written as $\I_1 = \I\cup(\cup_{i=1}^nB_i)$, where $\I$ is the connected set which contains $\I_0\setminus \{p\in\I_0:d^\e_1(p, \partial\I_0)\le\e^\mu\}$ (and therefore contains ``the most of'' $\I_1$), while the $B_i$ are the other connected components, all lying within a distance smaller than $\e^\mu$ from the boundary of $\I_0$. The proof of the connectedness follows showing that $B_i = \emptyset$ for every $i$. To this aim we need some auxiliary results.\vspace{5pt}\\
First we set $\I_1\cup \Z_1 := C\cup(\cup_{i=1}^n G_i)$ to be the decomposition in (strongly) connected components of this set, so that $\I\subset C$ and $B_i\subset G_i$. In Lemma \ref{lemma:optimal_position_connected_components} we prove that if a component $G_i$ is too small, then it can be translated in such a way to reduce the dissipation and/or the energy of $u_1$.\vspace{5pt}\\
Then, in the first step of Proposition \ref{prop:connectedness} we show that if instead there is a component $G_i$ which is too large, then we can replace the value of $u_1$ at every point in the space between $C$ and $G_i$ with the value $1$, and we can rearrange the surfactant which was lying in this region to form a new configuration which has a lower energy. As a result of Lemma \ref{lemma:optimal_position_connected_components} and the first step of Proposition \ref{prop:connectedness}, we have that the set $\I_1\cup\Z_1$ is strongly connected.\vspace{5pt}\\
In the second step of the proof of Proposition \ref{prop:connectedness} we prove that all the components $B_i$ must be contained inside the rectangle $R_{\I}$. We will prove that otherwise there would exist a slice of a component $B_{\bar{i}}$ which is too close to the boundary of $\I_0$ (in view of the previous step), which therefore can be eliminated paying an amount of dissipation which is negligible with respect to the reduction of the energy.\vspace{5pt}\\
In the final step of the proof of Proposition \ref{prop:connectedness} we prove that there cannot exist any component $B_i\subset R_{\I},$ and the connectedness of $I_1$ follows.

\begin{proposition}\label{prop:horizontal_convexity} Let $u_0, u_1 \in\A_\e$ such that $u_1$ is a minimizer for $\enF(\cdot,u_0)$ and  $\I_0$ is a staircase set. Then $\I_1\subset \I_0$ and every weakly connected component of $\I_1$ is a staircase set.
\end{proposition}
\begin{proof}
\emph{Step 1.} We first prove the inclusion. We observe that since $\#\Z_0<\infty$, and since $I_0$ satisfies assumption \eqref{H}, then $\Z_0\cup\I_0$ is bounded. Since $u_1$ is a minimizer of $\enF(\cdot, u_0)$ it follows that also $I_1\cup Z_1$ is bounded. Now, by contradiction, suppose that $\I_1\not\subset \I_0$, i.e. there exists $p \in \e\mathbb{Z}^{2}$ such that $u_0(p)\neq 1$ and $u_{1}(p)=1$. Moreover, up to rotation and reflection, we can find $p$ as above such that for every $t,s \in\mathbb{N}$, it holds $u_{1}(p-t\e e_{1}-s\e e_{2})\neq 1$. If $u_{1}(p-\e e_2)=u_{1}(p-\e e_1)=-1$, then the function $\tilde{u}_{1}\in\A_\e$ obtained replacing $u_1(p)\mapsto-1$ has a strictly smaller dissipation and $\E_\e(\tilde{u}_1)\le \E_\e(u_1)$, which contradicts the minimality of $u_1.$ Therefore, without loss of generality, we can assume that $u_{1}(p-\e e_2)=0$. We define $\overline{s}:=\max\{s \in \mathbb{N}: u_{1}(p-s\e e_{2})=0\}$ and $\overline{t}:=\max\{t \in \mathbb{N}: u_{1}(p-t\e e_{1}-\bar{s}\e e_2)=0\}$, and we define $\tilde{u}_1$ replacing $u_1(p)\mapsto0$ and $u_1(p-\bar{s}\e e_{2}-\overline{t}\e e_{1})\mapsto-1$. As in the previous case, the function $\tilde{u}_{1}$ has a strictly smaller dissipation and $\E_\e(\tilde{u}_1)\le \E_\e(u_1)$, which is not possible.
\vspace{5pt}
\\
\emph{Step 2.} By definition of staircase set we have to prove that every weakly connected component of $\I_1$ is horizontally and vertically convex. We only show the horizontal convexity, since the argument for the vertical convexity is analogous. Without loss of generality we prove the claim assuming that $\I_1$ is weakly connected, and we show that for every pair of points $p, q\in \I_1$ such that $q= p+\e ne_1$ for some $n\in\NN$, then $\dseg{p}{q}\subset\I_1$. We argue by contradiction assuming that there exists $m<n$ such that $u_1(p+\e me_1) \neq 1$. We observe that since $\I_0$ is horizontally convex, and since $\I_1\subset I_0$ is weakly connected, we can find two points $p'\in \I_1$ and $q' = p'+\e n'e_1\in \I_1$, where $n'\in\NN\cup\{0\}$ (i.e. possibly $p' = q'$), such that $\dseg{p'}{q'}\subset \I_1$, and satisfying one of the following:
    \begin{itemize}
        \item[] $p'-\e e_1-\e e_2\in \I_1,\;q'+\e e_1-\e e_2\in \I_1$ and $\dseg{p'-\e e_2}{q'-\e e_2}\subset\{u_1\neq 1\}$;
        \item[] $p'-\e e_1+\e e_2\in \I_1,\;q'+\e e_1+\e e_2\in \I_1$ and $\dseg{p'+\e e_2}{q'+\e e_2}\subset\{u_1\neq 1\}.$
    \end{itemize}
 Without loss of generality we can assume that the second situation occurs. First we observe that $\dseg{p'+\e e_2}{q'+\e e_2}\subset\{u_1 = 0\}.$ In fact otherwise we could find $\dseg{p''}{ q''}\subset\dseg{p'+\e e_2}{q'+\e e_2}\cap\{u_1=-1\}$ (possibly with $p'' = q''$) such that $u_1(p''-\e e_1)\neq -1\neq u_1(q''+\e e_1).$ In this case, we can define a competitor $\tilde{u}_1$ by replacing $u_1(\dseg{p''}{q''})\mapsto 1$, finding that $\E_\e(\tilde{u}_1)\le \E_\e(u_1)$ and $\tilde{u}_1$ has a strictly lower dissipation than $u_1$, which contradicts the minimality of $u_1$.
    
    We are left with the following three cases.
    \begin{itemize}
        \item[(i)] Suppose that $p'+\e e_2$ is a surfactant in the corner unit for $u_1$, according to Definition \ref{def: corner unit}. Then, by Lemma \ref{lemma:shape_optimality}-1.1) the set $\Z_1\cup \I_1$ is a discrete rectangle $R$. The horizontal upper side $\ell :=\dseg{a}{b}$ of $R$ has length at least equal to $n'+3$. We define $\tilde{u}_1\in\A_\e$ replacing $u(\dseg{p'+\e e_2}{q'+\e e_2}) \mapsto 1,\;u(\dseg{a+\e e_2}{ a+n'\e e_1+\e e_2}) \mapsto 0$. We observe that $\tilde{u}_1$ has strictly smaller dissipation than $u_1$, and $\E_\e(\tilde{u}_1)\le \E_\e(u_1),$ which contradicts the minimality of $u_1$.
        \item[(ii)] We suppose that $p'+\e e_2$ is not a surfactant in the corner, and that $n'\ge 1$. If $u_1(p'+2\e e_2) = 1$, then consider any point $\bar{p}\in\partial^+(I_1\cup Z_1)$. Replacing $u_1(\bar{p})\mapsto 0$ and $u_1(p+\e e_2)\mapsto1$ we strictly reduce the dissipation, and we do not increase the energy of $u_1$, which contradicts the minimality of $u_1$. We may therefore suppose that $u_1(p'+2\e e_2) = -1.$ We observe that $u_1(p'-\e e_1+2\e e_2)= u_1(p'+2\e e_2) = u_1(p'+\e e_1+2\e e_2)=-1.$ In fact, if for example $u_1(p'-\e e_1+2\e e_2)=0$  we could replace $u_1(p'+\e e_2)\mapsto 1$ and $u_1(p'+2\e e_2)\mapsto 0$ (see Figure \ref{fig:h_convexification}) obtaining a strictly smaller dissipation, and not increasing the energy.
        \begin{figure}[H]
         \centering
            \resizebox{0.60\textwidth}{!}{\input{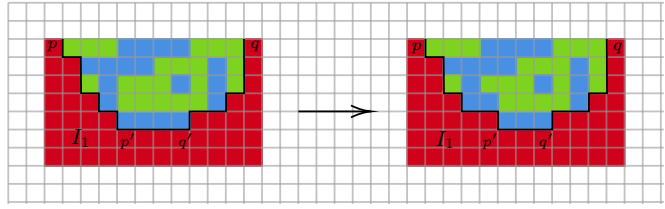}}
            \caption{The configuration on the right is the one that in (ii) of Proposition \ref{prop:horizontal_convexity} reduces the dissipation without increasing the energy when $u_{1}(p'+2\varepsilon e_2)=-1$.} \label{fig:h_convexification}
        \end{figure} 
        Inductively, we conclude that $\dseg{p'-\e e_1+2\e e_2}{q'+\e e_1+2\e e_2}\subset\{u_1 = -1\}$. But again, this is absurd since in this case we could replace $u_1(\dseg{p'+\e e_2}{q'+\e e_2})\mapsto 1$ and $u_1(\dseg{p'+2\e e_2}{q'+2\e e_2})\mapsto 0,$ strictly decreasing the dissipation, and not increasing the energy.
        \item[(iii)] We are left with the case $p' = q'$ (i.e. $n'=0$). As in the previous step we deduce that $u_1(p'+2\e e_2) = -1$. We replace $u_1(p'+\e e_2)\mapsto 1$ and $u_1(p'+2\e e_2)\mapsto 0$. This transformation strictly reduces the dissipation, and does not increase the energy.
    \end{itemize}
\end{proof}

The next Proposition provides a first rough lower bound on the distance between $\partial A_{\I_j}$ and $\partial A_{\I_{j+1}}.$
For $\I\subset\e\ZZ^2,\,A\subset\rr^2$ and $\beta>0$ we set $\mathscr{C}_{\e^\beta}(\I)\coloneqq\{p\in \I:d^\e_1(p, \partial\I)\le\e^\beta\}$ and $\mathscr{C}_{\e^\beta}(A_I)\coloneqq A_{\mathscr{C}_{\e^\beta}(I)}$. To simplify the notation, we write $\mathscr{C}_{\e^\beta}:=\mathscr{C}_{\e^\beta}(A_0).$

\begin{proposition}\label{prop:hausdorff_distance_from_boundary}
Let $u_0,u_1\in\A_\e$ be such that $u_0$ verifies \eqref{H} and $u_1$ is a minimizer of $\enF(\cdot, u_0)$, and let $\mu\in\R$ be such that $0<\mu<1/4$.  Set $\mathscr{C}_{\e^\mu}(\I_0):=\{p\in \I_0:d^\e_1(p, \partial\I_0)\le\e^\mu\}$. Then, there exists $\bar{\e}>0$ which depends only on $\bar{c}$ and $\mu$ such that for every $\e\le\bar{\e}$ it holds that $\I_0\setminus\mathscr{C}_{\e^\mu}(\I_0)\subset \I_1$; in particular $d_{\mathcal{H}}(\partial A_0,\partial A_1)\le \e^{\mu}$.
\end{proposition}
    \begin{proof}   
        Since $\enF(u_1)\le\enF(u_0)$, for every $\beta>0$ we have \begin{equation}\label{eq:3.1}
        \tau(\E_\e(u_0)-\E_\e(u_1))\ge \dis^1_\e(\I_1, \I_0)\ge\int_{(A_0\setminus\mathscr{C}_{\e^\beta})\setminus A_1}d^\e_1(x, \partial I_0)\dx\ge \e^\beta|(A_0\setminus \mathscr{C}_{\e^\beta})\setminus A_1|.
        \end{equation} 
        From (\ref{eq:3.1}), since $\tau=\e\zeta$, it holds
        \begin{equation}\label{eq:3.30}
            |(A_0\setminus\mathscr{C}_{\e^\beta})\setminus A_1|\le c\e^{1-\beta}.
        \end{equation}
     Moreover we observe that $|(A_0\setminus\mathscr{C}_{\e^\beta})\setminus A_1|\ge|A_0\setminus A_1|-c\e^\beta$ for a positive constant depending only on the perimeter of $A_0$. Then, with $\beta=1/2$ we deduce that
        \begin{equation}\label{eq:5}
            |A_0\setminus A_1|\le c(\e^{1-\beta}+\e^\beta)\le c\e^{1/2} \text{ as }\e\to 0.
        \end{equation}
    Now we set $\beta, \nu>0$ so that $\e^{\beta-1}$ and $N:=\e^{-\nu}$ are integers. We set $O_0 := \I_0\setminus \mathscr{C}_{\e^\beta}(\I_0)$, and inductively for every $n = 1, \dots, N$ we define $O_n:=O_{n-1}\setminus\partial^-O_{n-1}.$ By the mean value theorem there exists $\bar{n}$ such that
    \begin{align}\label{eq:3}
        \#(\partial^-O_{\bar{n}}\setminus \I_1)\le \frac{1}{N}\#\bigcup_{n=1}^N(\partial^-O_n\setminus \I_1)\le \frac{1}{N\e^2}|(A_0\setminus\mathscr{C}_{\e^\beta})\setminus A_1)|\le \e^{\nu-\beta-1}
    \end{align}
    and we set $\hat{O}:=O_{\bar{n}}.$ We observe that there exists a constant $c$ such that         \begin{equation}\label{eq:3.distance_from_boundary}
            d_{\mathcal{H}}(\partial A_{0}, \partial A_{\hat{O}})\le c(\e^\beta+\e^{1-\nu}),
            \end{equation}
        and relation \eqref{eq:3_geodesic_distance} still holds for the set $\hat{O}$ (possibly with a bigger constant $\tilde{c}$, which we will not rename).
        
We set $z:=\#\{p\in (\hat{O}\setminus\partial^-\hat{O}):u_1(p)= 0\}$. Let $r>0$ be such that $||p||_1< r\e$ for every $p\in \I_1\cup\Z_1$.
Consider the discrete square $Q$ whose lower basis is the segment $\dseg{(2r\e; 0)}{ (2r\e+\lfloor\sqrt{z}\rfloor\e-\e; 0)},$ and the square $Q'$ whose basis is the segment $\dseg{(2r\e; 0)}{(2r\e+\lfloor\sqrt{z}\rfloor\e; 0)}$. Then let $\hat{Q}\subset\e\ZZ^2$ be a set satisfying 
         \begin{equation}\label{eq:3.25}
             Q\subset\hat{Q}\subset Q', \#\hat{Q} = z,
         \end{equation} 
         and such that $A_{\hat{Q}}$
         has the lowest possible perimeter compatible with the constraint (\ref{eq:3.25}). We define
        \[\tilde{u}_1(p) = \begin{cases}
            1&\text{ if }p\in \I_1\cup(\hat{O}\setminus\partial^-\hat{O}),\\
            0&\text{ if }p\in \bigl(\hat{Q}\cup\Z_1\bigl)\setminus\bigl(\hat{O} \setminus\partial^-\hat{O}\bigl),\\
            -1&\text{ otherwise}
        \end{cases}\]
        (see figure (\ref{fig:Hausdorff_distance_boundary})).

\begin{figure}
    \centering
    \resizebox{0.950\textwidth}{!}{\input{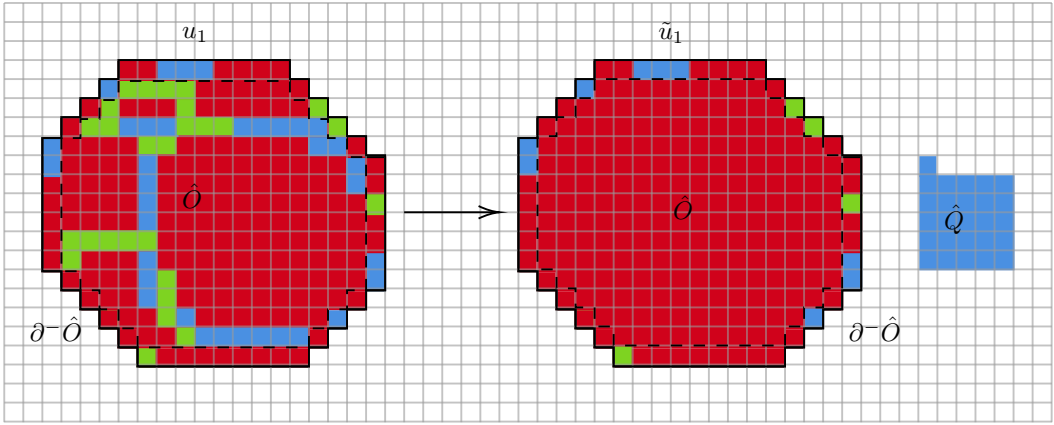}}
    \caption{The configuration on the right is the one that in Proposition \ref{prop:hausdorff_distance_from_boundary} leaves unchanged the dissipation term due to the surfactant, reduces the dissipation term due to the phase $1$ and increases the energy.}
    \label{fig:Hausdorff_distance_boundary}
\end{figure}
Since $\dis^0_\e(u_1, u_0) = \dis^0_\e(\tilde{u}_1, u_0),$ and $\dis^1_\e(u_1, u_0)\ge\dis^1_\e(\tilde{u}_1, u_0),$ from the estimate $\enF(u_1)\le\enF(\tilde{u}_1)$ we deduce that $\E_\e(u_1)\le \E_\e(\tilde{u}_1),$ that is
        \begin{align*}
            &\E_\e\Bigl(u_1, (\hat{O}\setminus\partial^-\hat{O})^C, (\hat{O}\setminus\partial^-\hat{O})^C\Bigl)+\E_\e(u_1, \hat{O}\setminus\partial^-\hat{O}, \hat{O})\le\\
            &\le \E_\e\Bigl(\tilde{u}_1, \bigl((\hat{Q}\cup\hat{O})\setminus\partial^-\hat{O}\bigl)^C, \bigl((\hat{Q}\cup\hat{O})\setminus\partial^-\hat{O}\bigl)^C\Bigl)+\E_\e(\tilde{u}_1, \hat{O}\setminus\partial^-\hat{O}, \hat{O})+\E_\e(\tilde{u}_1, \hat{Q}).
        \end{align*}
        We observe that, since $\hat{Q}\cup\partial^+\hat{Q}\subset\{u_1=-1\}$, we have 
        \[\E_\e\Bigl(u_1, (\hat{O}\setminus\partial^-\hat{O})^C, (\hat{O}\setminus\partial^-\hat{O})^C\Bigl) = \E_\e\Bigl(\tilde{u}_1, \bigl((\hat{Q}\cup\hat{O})\setminus\partial^-\hat{O}\bigl)^C, \bigl((\hat{Q}\cup\hat{O})\setminus\partial^-\hat{O}\bigl)^C\Bigl).\]
        It follows that 
        \begin{equation}\label{eq:0}
            \E_\e(u_1, \hat{O}\setminus\partial^-\hat{O}, \hat{O})\le \E_\e(\tilde{u}_1, \hat{O}\setminus\partial^-\hat{O}, \hat{O})+\E_\e(\tilde{u}_1, \hat{Q}).
        \end{equation}
        By Lemma \ref{lemma:energy_surfactant} we have
        \begin{equation}\label{eq:1}
        \begin{split}
        \E_\e(u_1, \hat{O}\setminus\partial^-\hat{O}, \hat{O})&\ge 2\e(1-k) z+\frac{1}{2}(1-k)P\bigl(\Z_1\cap (\hat{O}\setminus\partial^-\hat{O})\bigl)+\\    &+2\e\#\Bigl(\partial^+\I_1\cap\bigl(\hat{O}\setminus\partial^-\hat{O}\bigl)\cap\{u_1=-1\}\Bigl)\ge\\
        &\ge 2\e(1-k)z+\frac{1}{2}(1-k)\e\#\Bigl(\partial^+\I_1\cap\bigl(\hat{O}\setminus\partial^-\hat{O}\bigl)\cap\Z_1\Bigl)+\\  &+2\e\#\Bigl(\partial^+\I_1\cap\bigl(\hat{O}\setminus\partial^-\hat{O}\bigl)\cap\{u_1=-1\}\Bigl)\ge\\
        &\ge 2\e(1-k)z+c\e\#\Bigl(\partial^+\I_1\cap(\hat{O}\setminus\partial^-\hat{O})\Bigl)
        \end{split}
        \end{equation}
        
        Moreover, observing that $Per(\hat{Q})\le 4(\sqrt{z}+1)\e$, again from Lemma \ref{lemma:energy_surfactant}, it follows \begin{equation}\label{eq:2}
            \E_\e(\tilde{u}_1, \hat{Q})\le 2z(1-k)\e+2(\sqrt{z}+1)(1-k)\e
        \end{equation}
        and from (\ref{eq:3}) we have
        \begin{equation}\label{eq:4}
            \E_\e(\tilde{u}_1, \hat{O}\setminus\partial^-\hat{O}, \hat{O})\le c\e\#\partial^-\hat{O}\le c \e^{\nu-\beta}.
        \end{equation}
        Finally, combining \eqref{eq:0}, \eqref{eq:1},  \eqref{eq:2} and \eqref{eq:4} we get
        \begin{equation*}
            2\e(1-k)z+c\e\#\Bigl(\partial^+\I_1\cap(\hat{O}\setminus\partial^-\hat{O})\Bigl)\le c\e^{\nu-\beta}+2z(1-k)\e+c(\sqrt{z}+1)\e,\end{equation*}
            that is            \begin{equation}\label{eq:3.0} c\e\#\Bigl(\partial^+\I_1\cap(\hat{O}\setminus\partial^-\hat{O})\Bigl)\le c\e^{\nu-\beta}+c(\sqrt{z}+1)\e.
        \end{equation}
        Since $z\le\frac{1}{\e^2}|A_0\setminus A_1|$, from (\ref{eq:5}), \eqref{eq:3.0}, if we set $\nu = 3/4, \beta = 1/4$ then we find
        \begin{equation}\label{eq:3.10}     \e\#\Bigl(\partial^+\I_1\cap(\hat{O}\setminus\partial^-\hat{O})\Bigl)\le c(\sqrt{z}\e + \e^{\nu-\beta}+\e)\le c(\e^{1/4}+\e^{\nu-\beta})\le c\e^{1/4},
        \end{equation}
        as $\e\to 0$. Finally, from (\ref{eq:3.distance_from_boundary}), we have $d_{\mathcal{H}}(\partial A_{\I_0}, \partial A_{\hat{O}})\le c\e^{1/4}$ for some positive constant $c>0.$ 
        Fix $\mu, \mu'$ such that $0<\mu<\mu'<1/4$. We claim that for $\e$ sufficiently small the set $\I_0\setminus\mathscr{C}_{\e^\mu}(\I_0)=\{p\in\I_0:\:d^\e_1(p, \partial^+\I_0) >\e^{\mu}\}$ is contained in $\I_1.$ If we show this, the proof is complete.
        
        Let $I_1\cap\hat{O} =\cup_{i=1}^n B_i$, the $B_i$ being the strongly connected components of $I_1\cap\hat{O},$ and let $\ell_i:=\#(B_i\cap\partial^-\hat{O}).$ We observe that from Proposition \ref{prop:horizontal_convexity} each set $B_i$ is a staircase set, and that as $\e\to 0$ either $\e^{\mu'-1}\ge\ell_i$ or $\ell_i\ge\#(\partial^-\hat{O})-\e^{\mu'-1}.$ Indeed if $\e^{\mu'-1}<\ell_i<\#(\partial^-\hat{O})-\e^{\mu'-1}$ then, since $\hat{O}$ satisfies (\ref{eq:3_geodesic_distance}), we have
        \begin{equation}\label{eq:3.14}
            \#(\partial^+B_i\cap\bigl(\hat{O}\setminus\partial^-\hat{O}\bigl)\Bigl)\ge c\min\{\ell_i, \#\bigl(\partial^-\hat{O}\setminus B_i\bigl)\}\ge\e^{\mu'-1},
        \end{equation}
        which is impossible in view of (\ref{eq:3.10}). 
        
         Finally, since from (\ref{eq:3}) we deduce $\sum_i\ell_i\ge\#(\partial^-\hat{O})-\e^{-3/4}$, equation \eqref{eq:3.14} together with (\ref{eq:3.10}) and \eqref{eq:3_geodesic_distance} implies that there has to be $\bar{i}$ such that $\ell_{\bar{i}}\ge\#(\partial^-\hat{O})-\e^{\mu'-1}.$ We conclude observing that $\I_0\setminus\mathscr{C}_{\e^\mu}(\I_0)\subset B_{\bar{i}}\subset A_1$ as $\e\to 0.$
    \end{proof}

\begin{lemma}\label{lemma:optimal_position_connected_components}
Let $u_0,u_1\in\A_\e$ be such that $u_0$ verifies \eqref{H} and $u_1$ is a minimizer of $\enF(\cdot, u_0)$, and let $\nu\in(0, 1]$ be a constant independent on $\e$. Then there exists $\bar{\e}>0$ (depending on $\nu,\,\bar{c},\,\bar{R}$) such that for every $\e\le\bar{\e}$ the following properties hold true for $u_1$.
\begin{itemize}
    \item[(i)]  The set $\I_1\cup\Z_1$ does not have any weakly connected component $G$ with $\diam(G)\le\e^\nu$;
    \item[(ii)] If the set $\I_1\cup\Z_1$ is a rectangle then $\I_1$ does not have any weakly connected component $G$ with $\diam(G)\le\e^\nu$.
\end{itemize}
\end{lemma}
\begin{proof}
    \emph{First claim.} We argue by contradiction. From Proposition \ref{prop:hausdorff_distance_from_boundary} there exists $\mu\in(0, 1/4)$ such that $d_{\mathcal{H}}(\partial A_1, \partial A_0)\le \e^\mu.$ We prove that we can perform ``translations'' of $G$ which do not increase the dissipation of $u_1.$
    We denote by $I$ the component of $\I_1$ which contains $\I_0\setminus\mathscr{C}_{\e^\mu}(I_0)$, and we denote by $H^i = \dseg{p^{h, i}}{q^{h, i}}, \,i=1, \dots, n_h$ and $V^i = \dseg{p^{v, i}}{q^{v, i}}:\:i=1, \dots, n_v$ the horizontal and vertical slices of $\I_0$ respectively. We consider the map $\mathcal{T}^+_i[u_1]\in\A_\e$, for $i=1, 2,\,$ which translates the component $G$ in the direction $e_i$, defined as
    \begin{equation}\label{eq:translation_1}
       \mathcal{T}^+_i[u_1](q):=\begin{cases}
           u_1(q-\e e_i)&\text{ if    }q\in G+\e e_i,\\
           -1&\text{ if      }q\in G\setminus (G+\e e_i),\\
           u_1(q)&\text{ otherwise}.
       \end{cases} 
    \end{equation}
    and $\mathcal{T}_i^-[u_1]\in\A_\e$, for $i=1, 2,\,$ which translates the component $G$ in the direction $-e_i$, defined as   
    \begin{equation}\label{eq:translation_2}
       \mathcal{T}^-_i[u_1](q):=\begin{cases}
           u_1(q+\e e_i)&\text{ if    }q\in G-\e e_i,\\
           -1&\text{ if      }q\in G\setminus (G-\e e_i),\\
           u_1(q)&\text{ otherwise}.
       \end{cases} 
    \end{equation}
    We observe that $\E_\e(\mathcal{T}^\pm_i[u_1])<\E_\e(u_1)$ if and only if after the translation the boundary of the set $G$ intersects the boundary of an other component of $\I_1\cup\Z_1$, i.e.     \begin{equation}\label{eq:3.13}    \E_\e(\mathcal{T}^\pm_i[u_1])<\E_\e(u_1)\iff \Bigl(\partial^+(G\pm\e e_i)\Bigl)\cap \Bigl((\I_1\cup\Z_1)\setminus G\Bigl)\neq\emptyset,
    \end{equation}
    and otherwise $\E_\e(\mathcal{T}^\pm_i[u_1])=\E_\e(u_1)$ (see Figure \ref{fig:example1}).
    \begin{figure}[H]
    \centering
    \resizebox{0.700\textwidth}{!}{\tikzset{every picture/.style={line width=0.75pt}} 

\begin{tikzpicture}[x=0.75pt,y=0.75pt,yscale=-1,xscale=1]

\draw  [color={rgb, 255:red, 74; green, 144; blue, 226 }  ,draw opacity=1 ][fill={rgb, 255:red, 74; green, 144; blue, 226 }  ,fill opacity=1 ] (350.33,60) -- (410.33,60) -- (410.33,100) -- (350.33,100) -- cycle ;
\draw  [color={rgb, 255:red, 74; green, 144; blue, 226 }  ,draw opacity=1 ][fill={rgb, 255:red, 74; green, 144; blue, 226 }  ,fill opacity=1 ] (150.33,50) -- (210.33,50) -- (210.33,90) -- (150.33,90) -- cycle ;
\draw [color={rgb, 255:red, 208; green, 2; blue, 27 }  ,draw opacity=1 ][fill={rgb, 255:red, 208; green, 2; blue, 27 }  ,fill opacity=1 ]   (80.33,180) -- (80.33,130) -- (140.33,130) -- (140.33,110) -- (161.33,110) -- (230.33,110) -- (230.33,180) ;
\draw  [color={rgb, 255:red, 208; green, 2; blue, 27 }  ,draw opacity=1 ][fill={rgb, 255:red, 208; green, 2; blue, 27 }  ,fill opacity=1 ] (160.33,60) -- (200.33,60) -- (200.33,80) -- (160.33,80) -- cycle ;
\draw  [color={rgb, 255:red, 74; green, 144; blue, 226 }  ,draw opacity=1 ][fill={rgb, 255:red, 74; green, 144; blue, 226 }  ,fill opacity=1 ] (140.33,100) -- (230.33,100) -- (230.33,110) -- (140.33,110) -- cycle ;
\draw  [color={rgb, 255:red, 74; green, 144; blue, 226 }  ,draw opacity=1 ][fill={rgb, 255:red, 74; green, 144; blue, 226 }  ,fill opacity=1 ] (130.33,120) -- (80.33,120) -- (80.33,130) -- (130.33,130) -- cycle ;
\draw  [color={rgb, 255:red, 74; green, 144; blue, 226 }  ,draw opacity=1 ][fill={rgb, 255:red, 74; green, 144; blue, 226 }  ,fill opacity=1 ] (130.33,110) -- (130.33,130) -- (140.33,130) -- (140.33,110) -- cycle ;
\draw  [color={rgb, 255:red, 208; green, 2; blue, 27 }  ,draw opacity=1 ][fill={rgb, 255:red, 208; green, 2; blue, 27 }  ,fill opacity=1 ] (170.33,80) -- (180.33,80) -- (180.33,90) -- (170.33,90) -- cycle ;
\draw [color={rgb, 255:red, 208; green, 2; blue, 27 }  ,draw opacity=1 ][fill={rgb, 255:red, 208; green, 2; blue, 27 }  ,fill opacity=1 ]   (280.33,180) -- (280.33,130) -- (340.33,130) -- (340.33,110) -- (361.33,110) -- (430.33,110) -- (430.33,180) ;
\draw  [color={rgb, 255:red, 208; green, 2; blue, 27 }  ,draw opacity=1 ][fill={rgb, 255:red, 208; green, 2; blue, 27 }  ,fill opacity=1 ] (360.33,70) -- (400.33,70) -- (400.33,90) -- (360.33,90) -- cycle ;
\draw  [color={rgb, 255:red, 74; green, 144; blue, 226 }  ,draw opacity=1 ][fill={rgb, 255:red, 74; green, 144; blue, 226 }  ,fill opacity=1 ] (340.33,100) -- (430.33,100) -- (430.33,110) -- (340.33,110) -- cycle ;
\draw  [color={rgb, 255:red, 74; green, 144; blue, 226 }  ,draw opacity=1 ][fill={rgb, 255:red, 74; green, 144; blue, 226 }  ,fill opacity=1 ] (330.33,120) -- (280.33,120) -- (280.33,130) -- (330.33,130) -- cycle ;
\draw  [color={rgb, 255:red, 74; green, 144; blue, 226 }  ,draw opacity=1 ][fill={rgb, 255:red, 74; green, 144; blue, 226 }  ,fill opacity=1 ] (330.33,110) -- (330.33,130) -- (340.33,130) -- (340.33,110) -- cycle ;
\draw  [color={rgb, 255:red, 208; green, 2; blue, 27 }  ,draw opacity=1 ][fill={rgb, 255:red, 208; green, 2; blue, 27 }  ,fill opacity=1 ] (370.33,90) -- (380.33,90) -- (380.33,100) -- (370.33,100) -- cycle ;
\draw  [draw opacity=0] (30.33,10) -- (450.44,10) -- (450.44,200.33) -- (30.33,200.33) -- cycle ; \draw  [color={rgb, 255:red, 155; green, 155; blue, 155 }  ,draw opacity=0.6 ] (30.33,10) -- (30.33,200.33)(40.33,10) -- (40.33,200.33)(50.33,10) -- (50.33,200.33)(60.33,10) -- (60.33,200.33)(70.33,10) -- (70.33,200.33)(80.33,10) -- (80.33,200.33)(90.33,10) -- (90.33,200.33)(100.33,10) -- (100.33,200.33)(110.33,10) -- (110.33,200.33)(120.33,10) -- (120.33,200.33)(130.33,10) -- (130.33,200.33)(140.33,10) -- (140.33,200.33)(150.33,10) -- (150.33,200.33)(160.33,10) -- (160.33,200.33)(170.33,10) -- (170.33,200.33)(180.33,10) -- (180.33,200.33)(190.33,10) -- (190.33,200.33)(200.33,10) -- (200.33,200.33)(210.33,10) -- (210.33,200.33)(220.33,10) -- (220.33,200.33)(230.33,10) -- (230.33,200.33)(240.33,10) -- (240.33,200.33)(250.33,10) -- (250.33,200.33)(260.33,10) -- (260.33,200.33)(270.33,10) -- (270.33,200.33)(280.33,10) -- (280.33,200.33)(290.33,10) -- (290.33,200.33)(300.33,10) -- (300.33,200.33)(310.33,10) -- (310.33,200.33)(320.33,10) -- (320.33,200.33)(330.33,10) -- (330.33,200.33)(340.33,10) -- (340.33,200.33)(350.33,10) -- (350.33,200.33)(360.33,10) -- (360.33,200.33)(370.33,10) -- (370.33,200.33)(380.33,10) -- (380.33,200.33)(390.33,10) -- (390.33,200.33)(400.33,10) -- (400.33,200.33)(410.33,10) -- (410.33,200.33)(420.33,10) -- (420.33,200.33)(430.33,10) -- (430.33,200.33)(440.33,10) -- (440.33,200.33)(450.33,10) -- (450.33,200.33) ; \draw  [color={rgb, 255:red, 155; green, 155; blue, 155 }  ,draw opacity=0.6 ] (30.33,10) -- (450.44,10)(30.33,20) -- (450.44,20)(30.33,30) -- (450.44,30)(30.33,40) -- (450.44,40)(30.33,50) -- (450.44,50)(30.33,60) -- (450.44,60)(30.33,70) -- (450.44,70)(30.33,80) -- (450.44,80)(30.33,90) -- (450.44,90)(30.33,100) -- (450.44,100)(30.33,110) -- (450.44,110)(30.33,120) -- (450.44,120)(30.33,130) -- (450.44,130)(30.33,140) -- (450.44,140)(30.33,150) -- (450.44,150)(30.33,160) -- (450.44,160)(30.33,170) -- (450.44,170)(30.33,180) -- (450.44,180)(30.33,190) -- (450.44,190)(30.33,200) -- (450.44,200) ; \draw  [color={rgb, 255:red, 155; green, 155; blue, 155 }  ,draw opacity=0.6 ]  ;
\draw    (360.33,90) -- (360.33,70) -- (400.33,70) -- (400.33,90) -- (380.33,90) -- (380.33,100) -- (370.33,100) -- (370.33,90) -- cycle ;
\draw    (430.33,110) -- (340.33,110) -- (340.33,130) -- (280.33,130) -- (280.33,180) ;
\draw [color={rgb, 255:red, 208; green, 2; blue, 27 }  ,draw opacity=1 ]   (430.33,180) -- (250.33,180) -- (250.33,110) -- (300.33,110) -- (300.33,80) -- (320.33,80) -- (320.33,40) -- (430.33,40) -- cycle ;
\draw    (160.33,80) -- (160.33,60) -- (200.33,60) -- (200.33,80) -- (180.33,80) -- (180.33,90) -- (170.33,90) -- (170.33,80) -- cycle ;
\draw [color={rgb, 255:red, 208; green, 2; blue, 27 }  ,draw opacity=1 ]   (230.33,180) -- (50.33,180) -- (50.33,110) -- (100.33,110) -- (100.33,80) -- (120.33,80) -- (120.33,40) -- (230.33,40) -- cycle ;
\draw    (230.33,110) -- (140.33,110) -- (140.33,130) -- (80.33,130) -- (80.33,180) ;
\draw    (240.33,70) -- (268.33,70) ;
\draw [shift={(270.33,70)}, rotate = 180] [color={rgb, 255:red, 0; green, 0; blue, 0 }  ][line width=0.75]    (10.93,-3.29) .. controls (6.95,-1.4) and (3.31,-0.3) .. (0,0) .. controls (3.31,0.3) and (6.95,1.4) .. (10.93,3.29)   ;

\draw (162.33,13.4) node [anchor=north west][inner sep=0.75pt]  [font=\small]  {$u_{1}$};
\draw (172.33,143.4) node [anchor=north west][inner sep=0.75pt]    {$I$};
\draw (52.33,133.4) node [anchor=north west][inner sep=0.75pt]    {$I_{0}$};
\draw (177,65.4) node [anchor=north west][inner sep=0.75pt]  [font=\scriptsize]  {$G$};
\draw (352.33,13.4) node [anchor=north west][inner sep=0.75pt]  [font=\small]  {$\mathcal{T}_{2}^{-}[ u_{1}]$};
\draw (372.33,143.4) node [anchor=north west][inner sep=0.75pt]    {$I$};
\draw (256.33,133.4) node [anchor=north west][inner sep=0.75pt]    {$I_{0}$};
\draw (372.33,73.4) node [anchor=north west][inner sep=0.75pt]  [font=\scriptsize]  {$G$};

\end{tikzpicture}}
    \caption{The configuration on the right is obtained by doing a vertical translation in direction $-e_2$ of the component $G$. Since $\Bigl(\partial^+(G-\varepsilon e_2)\Bigl)\cap \Bigl((\I_1\cup\Z_1)\setminus G\Bigl)\neq\emptyset$, this configuration reduces the energy.}
    \label{fig:example1}
    \end{figure}
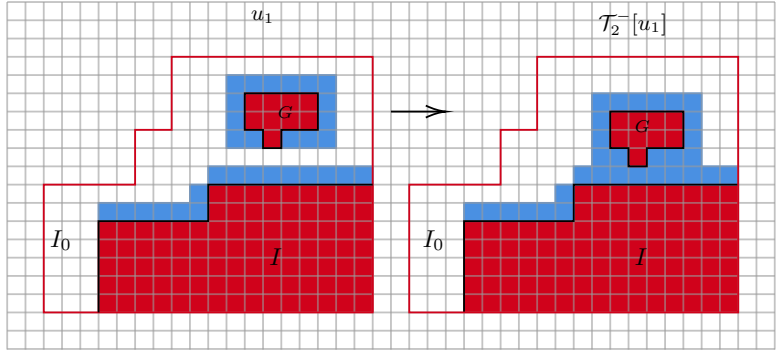
Since the general case follows essentially the same line of proof, we assume for simplicity that for every $p\in G$ it holds $p^{h, n_h}_1+2\e^\mu\le p_1\le q^{h, n_h}_1-2\e^\mu$ and $p_2\ge p^{h, 1}_2-\e^\mu$. As in Figure \eqref{fig:example1} we may translate the component $G$ downwards until it intersects the ``large'' component $I$. In particular the fact that $\diam(G)\le\e^\nu$ and that $u_0$ verifies assumption \eqref{H} ensures that the component $G$ does not surround $\I$, and that the dissipation does not increase while moving $G$ downwards. More precisely, we define inductively a sequence of competitors $(u_{1, j})_j$ such that $u_{1, 0} = u_1$, and for every $j$ we have $u_{1, j+1} = \mathcal{T}^-_2[u_{1, j}]$. We observe that in view of Proposition \ref{prop:hausdorff_distance_from_boundary}, there exists the minimum index $\bar{j}$ such that \eqref{eq:3.13} is satisfied with $u_{1, \bar{j}}$ in place of $u_1$, and $u_{1, \bar{j}+1} = \mathcal{T}^-_2[u_{1, \bar{j}}]$, and moreover $\bar{j}\le\e^{\mu-1}$. 
Since $\I_0$ is a staircase set with width $\bar{c}$, and since $\diam(G)\le\e^\nu$, we have that $\dis^1_\e(u_{1, j+1}, u_0)<\dis^1_\e(u_{1, j}, u_0)$ for every $j = 0, \dots, \bar{j}-1$.
We conclude that $\enF(\tilde{u}_{\bar{j}})<\enF(\tilde{u}_1).$
\vspace{5pt}\\
\emph{Second claim.} Suppose now that $\I_1\cup\Z_1$ is a rectangle. In this case we define the translations $\mathcal{T}^\pm_{i},$ for $i=1, 2$, as \eqref{eq:translation_1} and \eqref{eq:translation_2}, with the only difference that for every $q\in G\setminus (G\mp\e e_1)$ we substitute $0$ instead of $-1$. 
We proceed as in the proof of the first claim, constructing a sequence of competitors \( u_{1,j} \). This sequence can be chosen so that the number of zero-valued cells remains constant, i.e.,
\[
\#\{u_{1, j+1} = 0\} = \#\{u_{1, j} = 0\} \quad \text{for all } j,
\]
and such that
\[
\#\big(\{u_{1, j+1} = 1\} \cap \partial^-(\{u_{1, j+1}= 1\} \cup \Z_1)\big) \le \#\big(\{u_{1, j} = 1\} \cap \partial^-(\{u_{1, j}= 1\} \cup \Z_1)\big).
\]
As a consequence, we obtain \( \E_\varepsilon(\tilde{u}^1_{j+1}) \le \E_\varepsilon(\tilde{u}^1_j) \) for every $j$, and in particular we can conclude with the same argument of the first part of the proof.
\end{proof}

In the next auxiliary lemma we show that the length of the upper horizontal slices of $\I_1$ and $\I_0$ are of the same order, and the euclidean distance between them is less than $c\e$, for a constant $c>0,$ depending only on the constant $\bar{c}$ of our regularity assumption \eqref{H}. We observe that this is a refinement of the estimate that we get in Proposition \ref{prop:hausdorff_distance_from_boundary}.

\begin{lemma}\label{lemma:distance_horizontal_sides}
    Let $u_0,u_1\in\A_\e$ be such that $u_0$ verifies \eqref{H} and $u_1$ is a minimizer of $\enF(\cdot, u_0)$, and let $\mu\in\R$ be such that $0<\mu<1/4$. Let $\I$ be the only connected component of $\I_1$ which contains $\I_0\setminus\mathscr{C}_{\e^\mu}(\I_0)$. Suppose that not all the connected components of $\I_1\cup\Z_1$ are rectangles. We denote by $H^{0, 1}= \dseg{p'}{q'}$ the lower horizontal slice of $\I_0$ of length $P_{0, 1}:=\e(\#H^{0, 1})\ge\bar{c}$ (the constant of \ref{H}),
    $H^{1, i} = \dseg{p^{i}}{q^{i}},\,i=1, \dots, n_h$ the horizontal slices of $\I$ and by $S$ the set 
    \begin{equation}\label{eq:S}
    	S:=\{p\in\partial^-\I:p-\e e_2\in\partial^+\I\;\text{ and }\;p'_1+2\e^\mu\le p_1\le q'_1-2\e^\mu\}.
    \end{equation}
     Then it holds that $S\subset H^{1,1}\cup H^{1,2}$ and there exists a constant $c>0$, depending only on $\zeta$, $k$ and $\bar{c}$ such that     \begin{equation}\label{eq:3_dist_hor_sides}
    |p_2-p'_2|\le  c\e\;\;\text{ for every }\;\;p = (p_1, p_2)\in S.
    \end{equation}
In particular for every $p\in I_0\setminus R_{I}$ it holds that $d^\e_1(p, \partial I_0)\le c\e$. 
    \end{lemma}
\begin{proof}
    \emph{First part.} In order to prove that $S \subset H^{1,1}\cup H^{1,2}$, we argue by contradiction, supposing that there exists $\bar{i}\ge 3$ and $\bar{p}\in H^{1,\bar{i}}\cap S.$ Since $\I$ is a staircase set, without loss of generality we can assume that $p'_1+2\e^\mu\le\bar{p}_1<p^{1}_1.$
    Consider the set 
    \[S':=\{p\in\partial^-\I:p-\e e_2\in\partial^+\I\;\text{ and }\;p'_1+\e^\mu\le p_1\le p^{1}_1\}.\]
    We claim that for every $p\in S'\setminus (H^{1,1}\cup H^{1, 2})$ it holds $u_1(p+\e e_1-\e e_2) = 1.$ Otherwise, there exists $\bar{p}\in S'\cap H^{1,j}$, for some $j\ge 3$ such that the segment $L:=\dseg{\bar{p}-\e e_2}{p^{j-1}-\e e_1}$ has length at least $2$.
    \begin{figure}[H]
    \centering
    \resizebox{0.40\textwidth}{!}{\tikzset{every picture/.style={line width=0.75pt}} 

\begin{tikzpicture}[x=0.75pt,y=0.75pt,yscale=-1,xscale=1]

\draw [color={rgb, 255:red, 208; green, 2; blue, 27 }  ,draw opacity=1 ][fill={rgb, 255:red, 208; green, 2; blue, 27 }  ,fill opacity=1 ]   (170.02,90.12) -- (170.02,110.12) -- (200.02,110.12) -- (200.02,120.12) -- (220.02,120.12) -- (220.02,130.12) -- (260.02,130.12) -- (260.02,120.12) -- (260.02,110.12) -- (230.02,110.12) -- (210.02,110.12) -- (210.02,100.12) -- (180.02,100) -- (180.02,90) -- cycle ;
\draw  [color={rgb, 255:red, 74; green, 144; blue, 226 }  ,draw opacity=1 ][fill={rgb, 255:red, 74; green, 144; blue, 226 }  ,fill opacity=1 ] (170.02,110.12) -- (200.02,110.12) -- (200.02,120.12) -- (170.02,120.12) -- cycle ;
\draw  [color={rgb, 255:red, 74; green, 144; blue, 226 }  ,draw opacity=1 ][fill={rgb, 255:red, 74; green, 144; blue, 226 }  ,fill opacity=1 ] (210.02,120.12) -- (220.02,120.12) -- (220.02,130.12) -- (210.02,130.12) -- cycle ;
\draw  [color={rgb, 255:red, 126; green, 211; blue, 33 }  ,draw opacity=1 ][fill={rgb, 255:red, 126; green, 211; blue, 33 }  ,fill opacity=1 ] (200.02,120.12) -- (210.02,120.12) -- (210.02,130.12) -- (200.02,130.12) -- cycle ;
\draw  [draw opacity=0] (100.02,70) -- (320.33,70) -- (320.33,171) -- (100.02,171) -- cycle ; \draw  [color={rgb, 255:red, 155; green, 155; blue, 155 }  ,draw opacity=0.77 ] (100.02,70) -- (100.02,171)(110.02,70) -- (110.02,171)(120.02,70) -- (120.02,171)(130.02,70) -- (130.02,171)(140.02,70) -- (140.02,171)(150.02,70) -- (150.02,171)(160.02,70) -- (160.02,171)(170.02,70) -- (170.02,171)(180.02,70) -- (180.02,171)(190.02,70) -- (190.02,171)(200.02,70) -- (200.02,171)(210.02,70) -- (210.02,171)(220.02,70) -- (220.02,171)(230.02,70) -- (230.02,171)(240.02,70) -- (240.02,171)(250.02,70) -- (250.02,171)(260.02,70) -- (260.02,171)(270.02,70) -- (270.02,171)(280.02,70) -- (280.02,171)(290.02,70) -- (290.02,171)(300.02,70) -- (300.02,171)(310.02,70) -- (310.02,171)(320.02,70) -- (320.02,171) ; \draw  [color={rgb, 255:red, 155; green, 155; blue, 155 }  ,draw opacity=0.77 ] (100.02,70) -- (320.33,70)(100.02,80) -- (320.33,80)(100.02,90) -- (320.33,90)(100.02,100) -- (320.33,100)(100.02,110) -- (320.33,110)(100.02,120) -- (320.33,120)(100.02,130) -- (320.33,130)(100.02,140) -- (320.33,140)(100.02,150) -- (320.33,150)(100.02,160) -- (320.33,160)(100.02,170) -- (320.33,170) ; \draw  [color={rgb, 255:red, 155; green, 155; blue, 155 }  ,draw opacity=0.77 ]  ;
\draw    (290.02,104) -- (273.78,104.3) ;
\draw [shift={(271.78,104.33)}, rotate = 358.95] [color={rgb, 255:red, 0; green, 0; blue, 0 }  ][line width=0.75]    (10.93,-3.29) .. controls (6.95,-1.4) and (3.31,-0.3) .. (0,0) .. controls (3.31,0.3) and (6.95,1.4) .. (10.93,3.29)   ;
\draw    (290,145) -- (274.33,145) ;
\draw [shift={(272.33,145)}, rotate = 360] [color={rgb, 255:red, 0; green, 0; blue, 0 }  ][line width=0.75]    (10.93,-3.29) .. controls (6.95,-1.4) and (3.31,-0.3) .. (0,0) .. controls (3.31,0.3) and (6.95,1.4) .. (10.93,3.29)   ;
\draw [color={rgb, 255:red, 0; green, 0; blue, 0 }  ,draw opacity=1 ]   (170.02,89.88) -- (170.02,110) -- (200.02,110) -- (200.02,120) -- (220.02,120.12) -- (220.02,130.12) -- (260.02,130) ;
\draw [color={rgb, 255:red, 208; green, 2; blue, 27 }  ,draw opacity=1 ]   (140.02,90) -- (140.02,140) -- (160.02,140.12) -- (160.02,150.12) -- (260.02,150.12) ;
\draw   (180,120) .. controls (179.95,122.67) and (181.26,124.02) .. (183.93,124.07) -- (183.93,124.07) .. controls (187.74,124.14) and (189.63,125.5) .. (189.59,128.17) .. controls (189.63,125.5) and (191.56,124.2) .. (195.37,124.27)(193.66,124.24) -- (195.37,124.27) .. controls (198.04,124.32) and (199.4,123) .. (199.44,120.33) ;

\draw (295.02,139.52) node [anchor=north west][inner sep=0.75pt]  [font=\scriptsize]  {$H^{0,1}$};
\draw (202.02,107.52) node [anchor=north west][inner sep=0.75pt]  [font=\tiny]  {$p^{2}$};
\draw (182.02,100.52) node [anchor=north west][inner sep=0.75pt]  [font=\tiny]  {$\overline{p}$};
\draw (162.02,139.4) node [anchor=north west][inner sep=0.75pt]  [font=\tiny]  {$p'$};
\draw (295.02,99.52) node [anchor=north west][inner sep=0.75pt]  [font=\scriptsize]  {$H^{1,3}$};
\draw (142.02,125.4) node [anchor=north west][inner sep=0.75pt]  [font=\footnotesize]  {$I_{0}$};
\draw (252.02,116.4) node [anchor=north west][inner sep=0.75pt]  [font=\footnotesize]  {$I$};
\draw (184,129.4) node [anchor=north west][inner sep=0.75pt]  [font=\footnotesize]  {$L$};
\draw (222.02,117.4) node [anchor=north west][inner sep=0.75pt]  [font=\tiny]  {$p^{1}$};

\end{tikzpicture}}
    \caption{An example of $\overline{p}\in S'\cap H^{1,3}$.}
    \label{fig:distance_horizontal_sides}
    \end{figure}
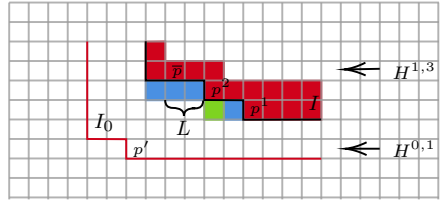
    Since $L\not\subset\{u_0\neq 1\}$, and since by our assumption not every weakly connected component of $\I_1\cup\Z_1$ is a rectangle, we deduce by Lemma \ref{lemma:shape_optimality} that $L\subset\Z_1$ and that $u_1(p^{j-1}-\e e_2) = -1$  (see Figure \ref{fig:distance_horizontal_sides}) . Since $j-1\geq 2$ we can apply Lemma \ref{lemma:shape_optimality} to the segment $L':=\dseg{p^{j-1}-\e e_2}{p^{j-1}+\e e_1-\e e_2}$. More precisely, since $u_0(p^{j-1}-\e e_2)=1$, if on one hand $\#L'=1$, by Lemma \ref{lemma:shape_optimality}-1.2), $u_1(p^{j-1}-\e e_2)\neq -1$, which is a contradiction. On the other hand, if $\#L'\geq 2$, we get $L'\not\subset \{u_0\neq 1\}$ and since by our assumption not every weakly connected components of $I_1 \cup Z_1$ is a rectangle then,  by Lemma  \ref{lemma:shape_optimality}-(ii), we get $L'\subset Z_1$ which is again a contradiction. Therefore the claim is proved.
    
    Now, consider $\bar{p}(t):=\bar{p}-t\e e_1-t\e e_2$ for $t = 0, \dots, \e^{\mu-1}.$  In view of the claim, we have that for every $t$ there exists $q(t)\in S'$ such that $q_2(t)\le\bar{p}_2(t)$ (where we denote by $\bar{p}_2(t)$ the second component of $\bar{p}(t)$).
    Now observe that for every $p\in S'$ it holds
    \[d^\e_1(p, \partial\I_0) = |p_2-p'_2|.\]
    In particular it follows 
    \[d^\e_1\bigl(q(\mu-1), \partial\I_0\bigl) \ge d^\e_1\bigl(\bar{p}(\mu-1), \partial\I_0\bigl) = \e^\mu+d^\e_1(\bar{p},\partial\I_0)>\e^\mu,\]
    which contradicts Proposition \ref{prop:hausdorff_distance_from_boundary}.\vspace{5pt}\\
    \emph{Second part.} We now prove \eqref{eq:3_dist_hor_sides}. Suppose for simplicity that $H^{1, 1}\subset H^{1,2}-\e e_2$ (in the other cases the proof is similar). It is therefore sufficient to prove (\ref{eq:3_dist_hor_sides}) for every point $p\in H^{1,2}$. To do so, we define a competitor $\tilde{u}_1$ obtained by replacing $u_1(H^{1, 1}-\e e_2)\mapsto 1.$ Since we want that $\#Z_1= \#\{\tilde{u}_1=0\},$ in the following construction we explain how to define $\{\tilde{u}_1=0\}$ in such a way that the energy of $\tilde{u}_{1}$ is not ``too much bigger'' than the energy of $u_1$. 

    Set $p^{1, 0}=p^{1},\;q^{1, 0}=q^1$ and 
    \[h_0:=\min\{h\in\NN:\dseg{p^{1,0}-h\e e_2}{q^{1,0}-h\e e_2}\}\not\subset\Z_1.\]Now, if $\dseg{p^{1,0}-h_0\e e_2}{q^{1,0}-h_0\e e_2}\subset\{u_1=-1\}$ we interrupt the process. Otherwise, we observe that there exist $p^{1, 1}, q^{1, 1}$, possibly with $p^{1, 1}=q^{1, 1},$ such that $(H^{1,1}-h_0\e e_2)\cap\Z_1 = \dseg{p^{1, 1}}{q^{1, 1}}.$ Indeed if this is not the case, then there exist two points $\tilde{p}, \tilde{q}\in\dseg{p^{1,0}-h_0\e e_2}{q^{1, 0}-h_0\e e_2}$ such that $\dseg{\tilde{p}}{\tilde{q}}\subset\{u_1=-1\}$ and $u_1(\tilde{p}-\e e_1) = u_1(\tilde{q}+\e e_1) = 0$. But this is not possible because replacing $u_1(\dseg{\tilde{p}}{\tilde{q}})\mapsto1$ we would strictly reduce the dissipation, while not increasing the energy. 

    Inductively we define 
    \[h_{i+1}:=\min\{h\in\NN:\dseg{p^{1,i}-h\e e_2}{q^{1,i}-h\e e_2}\}\not\subset\Z_1,\]
    and sequences 
    $p^{1,i}, q^{1,i}$ such that 
        \begin{equation*}
        \begin{aligned}
            &p^{1,i+1}_1\ge p^{1,i}_1,\;q^{1,i+1}_1\le q^{1,i}_1,\;q^{1,i+1}_1-p^{1,i+1}_1<q^{1,i}_1-p^{1,i}_1,\\
            &\dseg{p^{1,i+1}}{q^{1,i+1}} = \dseg{p^{1,i}-h_i\e e_2}{q^{1,i}-h_i\e e_2}\cap\Z_1.
        \end{aligned}
        \end{equation*}
        We stop this process at index $i=n$ such that $\dseg{p^{1,n}-h_n\e e_2}{q^{1,n}-h_n\e e_2}\subset\{u_1=-1\}.$
    
We now build a competitor $\tilde{u}_1$ distinguishing between two cases. If $\dseg{p^{1, 0}-\e e_2}{q^{1, 0}-\e e_2}\subset\{u_1=-1\}$ then we replace $u_1(\dseg{p^{1, 0}-\e e_2}{q^{1, 0}-\e e_2})\mapsto 1$. Otherwise, we define $\tilde{u}_1$ as follows. We denote by 
\[\tilde{Z}:=\bigcup_{i=0}^{n}\bigcup_{j=1}^{h_i-1}\dseg{p^{1, i}-j\e e_2}{q^{1, i}-j\e e_2},\]
which is the set of the surfactant cells lying ``below'' the segment $H^{1, 1}$. Then we replace
\begin{equation}\label{eq:replace_2}
            \begin{cases}
                u_1(\dseg{p^{1, 0}-\e e_2}{q^{1,0}-\e e_2})\mapsto 1,\\
                u_1(\tilde{Z}-\e e_2) = 0.
            \end{cases}
        \end{equation}
    see Figure \eqref{fig:Figure_Lemma3_11}. 
    \begin{figure}[H]
    \centering
    \resizebox{0.70\textwidth}{!}{\input{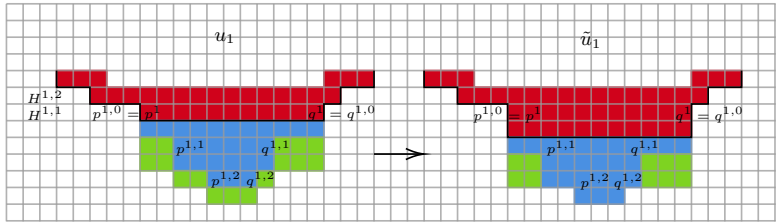}}
    \caption{On the right an example of competitor $\tilde{u}_1$ in Lemma \ref{lemma:distance_horizontal_sides} when $n=2$.}
    \label{fig:Figure_Lemma3_11}
    \end{figure}
        We can repeat the same procedure on the sides $\dseg{p^2-\e e_2}{p^1-\e e_1}$ and $\dseg{q^1+\e e_1}{q^2-\e e_2}$: i.e. we replace
        \[u_1(\dseg{p^2-\e e_2}{p^1-\e e_1}\cup\dseg{q^1+\e e_1}{q^2-\e e_2})\mapsto 1,\]
        and we move downwards the cells of surfactant which were adjacent to these two segments following the same strategy as in \eqref{eq:replace_2}. 

        If we call $\tilde{\I}_1:=\{\tilde{u}_1=1\}$, and $d:=|p'_2-p^1_2|$, we have \begin{equation}\label{eq:energy_1}
            \E_\e(\tilde{u}_1)\le \E_\e(u_1)+4\e,\;\; \#\{\tilde{u}_1=0\} = \#\{u_1=0\},
        \end{equation}
        and 
        \begin{equation}\label{eq:dis}
            \frac{1}{\tau}\dis^1_\e(\tilde{\I}_1, \I_0)\le\frac{1}{\tau}\dis^1_\e(\I_1, \I_0)-\frac{P_{0, 1}-8\e^\mu}{\zeta}d.
        \end{equation}
        By minimality of $u_1$ we deduce that
        \begin{equation}\label{eq:3.estimate}
        \begin{split}
            0\le \E_\e(\tilde{u}_1)+\frac{1}{\tau}&\dis^1_\e(\tilde{I}_1, I_0) - \E_\e(u_1)-\frac{1}{\tau}\dis^1_\e(I_1, I_0)\le 4\e-\frac{P_{0,1}-8\e^\mu}{\zeta}d\\
            &\implies \frac{d}{\e}\le\frac{4\zeta}{P_{0, 1}-8\e^\mu}\le\Bigl\lfloor\frac{4\zeta}{P_{0, 1}}\Bigl\rfloor+1
        \end{split}
        \end{equation}
        where the last inequality holds since $d/\e$ is integer. Hence, by triangular inquality we get \eqref{eq:3_dist_hor_sides}.

        The fact that for every point $p\in I_0\setminus R_I$ it holds $d^\e_1(p, \partial I_0)\le c\e$ is a geometric consequence of \eqref{eq:3_dist_hor_sides}.
\end{proof}

We are finally ready to prove that $\I_1$ is connected.

\begin{proposition}\label{prop:connectedness}
    Let $\gamma>0$. Let $u_0,u_1\in\A_\e$ be such that $u_0$ verifies \eqref{H} and $u_1$ is a minimizer of $\enF(\cdot, u_0)$, and let $\mu\in\R$ be such that $0<\mu<1/4$. Then there exists $\bar{\e}$ such that the set $\I_1$ is a strongly connected staircase set contained in $\I_0$ for every $\e\le\bar{\e}$.
\end{proposition}
\begin{proof}
In view of Proposition \ref{prop:horizontal_convexity} we only need to prove that $\I_1$ is strongly connected.\\
\emph{Step 1.} We prove that the set $\I_1\cup\Z_1$ has only one strongly connected component.\\
From Proposition \ref{prop:hausdorff_distance_from_boundary} we know that there exists $\bar{\e}$ such that for every $\e\le\bar
\e$ it holds $d_{\mathcal{H}}(\partial A_0, \partial A_1)\le \e^\mu.$ Now set $\I_1=\I\bigcup(\cup_{i=1}^n B_i)$, where $\I$ is the strongly connected component of $\I_1$ which contains $\I_0\setminus\mathscr{C}_{\e^\mu}(\I_0)$, and the $B_i$ are the other strongly connected components of $\I_1.$ Then let $\I_1\cup\Z_1 = C\bigcup(\cup_{i=1}^n G_i)$ the decomposition of $\I_1\cup\Z_1$ in strongly connected components such that $\I\subset C$, and $B_i\subset G_i$. We want to prove that $G_i = \emptyset $ for every $i$, that is $I_1\cup Z_1$ is strongly connected. 

Assume by contradiction that $G_i\neq \emptyset$ for some $i$, and rename $G:=G_i$ for simplicity. We fix $\nu\in(0, \mu)$. In view of Lemma \ref{lemma:optimal_position_connected_components} we deduce that, up to taking a smaller $\bar{\e},$ for $\e\le\bar{\e}$ it holds $\diam(G_i)>\e^\nu$. Consider two points $p^G, q^G\in \partial^-G_i$ such that $d^\e_\infty(p^G, q^G)=\max\{d^\e_\infty(p, q):p, q\in G\}$, and let
    \[p^\I\in\argmin_{q\in \partial^-\I} d^\e_\infty(p^G, q),\quad q^\I\in\argmin_{q\in \partial^-\I} d^\e_\infty(q^G, q).\]
    Now we set $\pi_\I$ as the (weakly continuous) path of minimal length belonging to $\Gamma^w_{p^\I, q^\I}(\partial^-\I)$, and we set $\pi_p$ and $\pi_q$ as two (weakly continuous) paths of minimal length belonging to $\Gamma^w_{p^\I, p^G}(\I_0)$ and $\Gamma^w_{q^\I, q^G}(\I_0)$ respectively (see Figure \ref{fig:cammini}) . 
    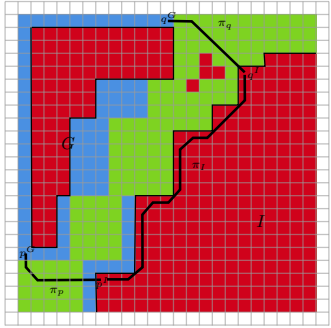
\begin{figure}[H]
          \centering
    \resizebox{0.30\textwidth}{!}{\tikzset{every picture/.style={line width=0.75pt}} 

\begin{tikzpicture}[x=0.75pt,y=0.75pt,yscale=-1,xscale=1]

\draw  [draw opacity=0][fill={rgb, 255:red, 126; green, 211; blue, 33 }  ,fill opacity=1 ] (360.33,50) -- (591,50) -- (591,280) -- (360.33,280) -- cycle ;
\draw [draw opacity=0][fill={rgb, 255:red, 208; green, 2; blue, 27 }  ,fill opacity=1 ]   (371,60) -- (371,230) -- (390.33,230) -- (390.33,190) -- (400.33,190) -- (400.33,130) -- (420.33,130) -- (420.33,100) -- (480.33,100) -- (480.33,60) -- cycle ;
\draw  [draw opacity=0][fill={rgb, 255:red, 74; green, 144; blue, 226 }  ,fill opacity=1 ] (421,240) -- (451,240) -- (451,250) -- (421,250) -- cycle ;
\draw  [draw opacity=0][fill={rgb, 255:red, 74; green, 144; blue, 226 }  ,fill opacity=1 ] (451,190) -- (451,240) -- (440.33,240) -- (440.33,190) -- cycle ;
\draw  [draw opacity=0][fill={rgb, 255:red, 74; green, 144; blue, 226 }  ,fill opacity=1 ] (421,240) -- (421,280) -- (410.33,280) -- (410.33,240) -- cycle ;
\draw  [draw opacity=0][fill={rgb, 255:red, 74; green, 144; blue, 226 }  ,fill opacity=1 ] (371,50) -- (371,240) -- (360.33,240) -- (360.33,50) -- cycle ;
\draw  [draw opacity=0][fill={rgb, 255:red, 74; green, 144; blue, 226 }  ,fill opacity=1 ] (371,230) -- (400.33,230) -- (400.33,240) -- (371,240) -- cycle ;
\draw  [draw opacity=0][fill={rgb, 255:red, 74; green, 144; blue, 226 }  ,fill opacity=1 ] (400.33,190) -- (400.33,230) -- (390.33,230) -- (390.33,190) -- cycle ;
\draw  [draw opacity=0][fill={rgb, 255:red, 74; green, 144; blue, 226 }  ,fill opacity=1 ] (400.33,130) -- (430.33,130) -- (430.33,190) -- (400.33,190) -- cycle ;
\draw  [draw opacity=0][fill={rgb, 255:red, 74; green, 144; blue, 226 }  ,fill opacity=1 ] (431,100) -- (431,130) -- (420.33,130) -- (420.33,100) -- cycle ;
\draw  [draw opacity=0][fill={rgb, 255:red, 74; green, 144; blue, 226 }  ,fill opacity=1 ] (431,100) -- (460.33,100) -- (460.33,130) -- (431,130) -- cycle ;
\draw  [draw opacity=0][fill={rgb, 255:red, 208; green, 2; blue, 27 }  ,fill opacity=1 ] (510,80) -- (510,100) -- (500,100) -- (500,80) -- cycle ;
\draw  [draw opacity=0][fill={rgb, 255:red, 208; green, 2; blue, 27 }  ,fill opacity=1 ] (510,90) -- (520,90) -- (520,100) -- (510,100) -- cycle ;
\draw  [draw opacity=0][fill={rgb, 255:red, 208; green, 2; blue, 27 }  ,fill opacity=1 ] (490,100) -- (500,100) -- (500,110) -- (490,110) -- cycle ;
\draw  [draw opacity=0][fill={rgb, 255:red, 74; green, 144; blue, 226 }  ,fill opacity=1 ] (371,50) -- (480.33,50) -- (480.33,60) -- (371,60) -- cycle ;
\draw [draw opacity=0][fill={rgb, 255:red, 208; green, 2; blue, 27 }  ,fill opacity=1 ]   (591,80) -- (551,80) -- (551,90) -- (531,90) -- (531,120) -- (510.33,120) -- (510.33,140) -- (480.33,140) -- (480.33,190) -- (451,190) -- (451,250) -- (421,250) -- (421,280) -- (591,280) -- cycle ;
\draw [line width=1.5]    (366.33,234.67) -- (366.44,245.33) -- (375.44,255.33) -- (424.33,255) ;
\draw  [draw opacity=0] (360.33,50) -- (591,50) -- (591,280) -- (360.33,280) -- cycle ; \draw  [draw opacity=0] (360.33,50) -- (360.33,280)(370.33,50) -- (370.33,280)(380.33,50) -- (380.33,280)(390.33,50) -- (390.33,280)(400.33,50) -- (400.33,280)(410.33,50) -- (410.33,280)(420.33,50) -- (420.33,280)(430.33,50) -- (430.33,280)(440.33,50) -- (440.33,280)(450.33,50) -- (450.33,280)(460.33,50) -- (460.33,280)(470.33,50) -- (470.33,280)(480.33,50) -- (480.33,280)(490.33,50) -- (490.33,280)(500.33,50) -- (500.33,280)(510.33,50) -- (510.33,280)(520.33,50) -- (520.33,280)(530.33,50) -- (530.33,280)(540.33,50) -- (540.33,280)(550.33,50) -- (550.33,280)(560.33,50) -- (560.33,280)(570.33,50) -- (570.33,280)(580.33,50) -- (580.33,280)(590.33,50) -- (590.33,280) ; \draw  [draw opacity=0] (360.33,50) -- (591,50)(360.33,60) -- (591,60)(360.33,70) -- (591,70)(360.33,80) -- (591,80)(360.33,90) -- (591,90)(360.33,100) -- (591,100)(360.33,110) -- (591,110)(360.33,120) -- (591,120)(360.33,130) -- (591,130)(360.33,140) -- (591,140)(360.33,150) -- (591,150)(360.33,160) -- (591,160)(360.33,170) -- (591,170)(360.33,180) -- (591,180)(360.33,190) -- (591,190)(360.33,200) -- (591,200)(360.33,210) -- (591,210)(360.33,220) -- (591,220)(360.33,230) -- (591,230)(360.33,240) -- (591,240)(360.33,250) -- (591,250)(360.33,260) -- (591,260)(360.33,270) -- (591,270)(360.33,280) -- (591,280) ; \draw  [draw opacity=0]  ;
\draw  [draw opacity=0] (350.33,40) -- (600.44,40) -- (600.44,290.33) -- (350.33,290.33) -- cycle ; \draw  [color={rgb, 255:red, 155; green, 155; blue, 155 }  ,draw opacity=0.6 ] (350.33,40) -- (350.33,290.33)(360.33,40) -- (360.33,290.33)(370.33,40) -- (370.33,290.33)(380.33,40) -- (380.33,290.33)(390.33,40) -- (390.33,290.33)(400.33,40) -- (400.33,290.33)(410.33,40) -- (410.33,290.33)(420.33,40) -- (420.33,290.33)(430.33,40) -- (430.33,290.33)(440.33,40) -- (440.33,290.33)(450.33,40) -- (450.33,290.33)(460.33,40) -- (460.33,290.33)(470.33,40) -- (470.33,290.33)(480.33,40) -- (480.33,290.33)(490.33,40) -- (490.33,290.33)(500.33,40) -- (500.33,290.33)(510.33,40) -- (510.33,290.33)(520.33,40) -- (520.33,290.33)(530.33,40) -- (530.33,290.33)(540.33,40) -- (540.33,290.33)(550.33,40) -- (550.33,290.33)(560.33,40) -- (560.33,290.33)(570.33,40) -- (570.33,290.33)(580.33,40) -- (580.33,290.33)(590.33,40) -- (590.33,290.33)(600.33,40) -- (600.33,290.33) ; \draw  [color={rgb, 255:red, 155; green, 155; blue, 155 }  ,draw opacity=0.6 ] (350.33,40) -- (600.44,40)(350.33,50) -- (600.44,50)(350.33,60) -- (600.44,60)(350.33,70) -- (600.44,70)(350.33,80) -- (600.44,80)(350.33,90) -- (600.44,90)(350.33,100) -- (600.44,100)(350.33,110) -- (600.44,110)(350.33,120) -- (600.44,120)(350.33,130) -- (600.44,130)(350.33,140) -- (600.44,140)(350.33,150) -- (600.44,150)(350.33,160) -- (600.44,160)(350.33,170) -- (600.44,170)(350.33,180) -- (600.44,180)(350.33,190) -- (600.44,190)(350.33,200) -- (600.44,200)(350.33,210) -- (600.44,210)(350.33,220) -- (600.44,220)(350.33,230) -- (600.44,230)(350.33,240) -- (600.44,240)(350.33,250) -- (600.44,250)(350.33,260) -- (600.44,260)(350.33,270) -- (600.44,270)(350.33,280) -- (600.44,280)(350.33,290) -- (600.44,290) ; \draw  [color={rgb, 255:red, 155; green, 155; blue, 155 }  ,draw opacity=0.6 ]  ;
\draw [line width=1.5]    (476,55) -- (494.44,55.33) -- (535.33,94.67) ;
\draw [line width=1.5]    (429,254.67) -- (445.33,254.67) -- (456.44,245.33) -- (456.33,204.67) -- (464.44,195.33) -- (476.44,195.33) -- (485.44,185.33) -- (485.44,154.33) -- (495.44,145.33) -- (506.44,145.33) -- (535.44,116.33) -- (535.33,96.67) ;
\draw    (370.33,60) -- (371,230) -- (390.33,230) -- (390.33,190) -- (400.33,190) -- (400.33,130) -- (420.33,130) -- (420.33,100) -- (480.33,100) -- (480.33,60) -- (371,60) ;
\draw    (421,280) -- (420.33,250) -- (450.33,250) -- (451,190) -- (480.33,190) -- (480.33,140) -- (510.33,140) -- (510.33,120) -- (530.33,120) -- (530.33,90) -- (551,90) -- (550.33,80) -- (590.33,80) ;

\draw (360,227.4) node [anchor=north west][inner sep=0.75pt]  [font=\tiny]  {$p^{G}$};
\draw (469,46.4) node [anchor=north west][inner sep=0.75pt]  [font=\tiny]  {$q^{G}$};
\draw (419,251.4) node [anchor=north west][inner sep=0.75pt]  [font=\tiny]  {$p^{I}$};
\draw (536,89.4) node [anchor=north west][inner sep=0.75pt]  [font=\tiny]  {$q^{I}$};
\draw (383,259.4) node [anchor=north west][inner sep=0.75pt]  [font=\scriptsize]  {$\pi _{p}$};
\draw (513,53.4) node [anchor=north west][inner sep=0.75pt]  [font=\scriptsize]  {$\pi _{q}$};
\draw (493,163.4) node [anchor=north west][inner sep=0.75pt]  [font=\scriptsize]  {$\pi _{I}$};
\draw (543,203.4) node [anchor=north west][inner sep=0.75pt]    {$I$};
\draw (392.33,143.4) node [anchor=north west][inner sep=0.75pt]    {$G$};

\end{tikzpicture}}
    \caption{The thick black line represents the weakly continuous paths $\pi_I$, $\pi_p$ and $\pi_q$ connecting $G$ and $I$.}
    \label{fig:cammini}
   \end{figure}

        Now let $\Gamma:=\pi_\I\cup\pi_p\cup\pi_q\cup G;$ we observe that $\R^2\setminus A_\Gamma$ has at least two connected components, and we denote by $S$ the bounded component bordering $A_{\pi_\I}$. 
Finally let $S_\e:=S\cap\e\ZZ^2$ and set $z:=\#(\Z_1\cap S_\e)$ (see the left hand side of Figure \ref{fig:connectedness}).

Let $r$ be such that $||p||_1\le r\e$ for all $p\in \I_1\cup\Z_1$, and consider the discrete square $Q$ whose lower basis is the segment $\dseg{(2r\e; 0)}{(2r\e+\lfloor\sqrt{z}\rfloor\e-\e; 0)},$ and the square $Q'$ whose lower basis is the segment $[(2r\e; 0), (2r\e+\lfloor\sqrt{z}\rfloor\e; 0)]$. Then let $\hat{Q}\subset\e\ZZ^2$ be any set such that 
    \begin{equation}\label{eq:3.square}
        Q\subset\hat{Q}\subset Q',\;\; \#\hat{Q} = z,
    \end{equation}
    and $A_{\hat{Q}}$ has the lowest possible perimeter consistent with (\ref{eq:3.square}).
    
    We define a new competitor of $u_1$ as 
\begin{equation}\label{eq:3.11}
        \tilde{u}_1(p) = \begin{cases}
            1&\text{ if } p\in \I_1\cup S_\e,\\
            0&\text{ if }p\in\Bigl(\Z_1\setminus S_\e\Bigl)\cup\hat{Q},\\
            -1&\text{ otherwise},
        \end{cases}
    \end{equation}
    see also Figure (\ref{fig:connectedness}).
    \begin{figure}[H]
          \centering
    \resizebox{0.75\textwidth}{!}{\input{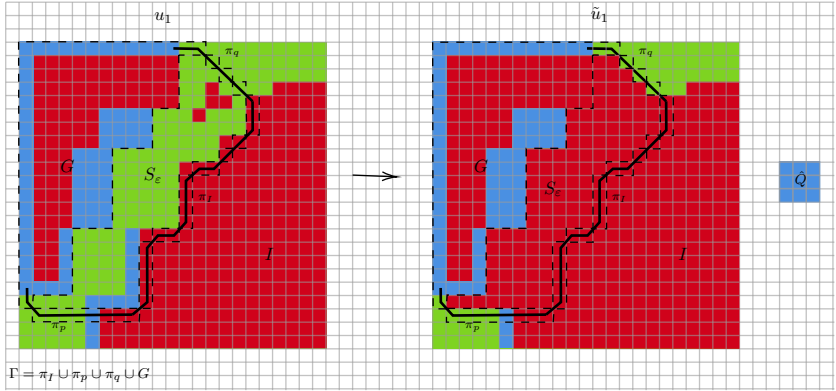}}
    \caption{On the right an example of competitor $\tilde{u}_1$. We represent $\Gamma$ as the region enclosed between the two dashed closed lines which is not $S_\varepsilon$. }
    \label{fig:connectedness}
   \end{figure}
    Now since $G\subset\mathscr{C}_{\e^\mu}(I_0)$ we have $\text{len}(\pi_p)\le\e^\mu$ and $\text{len}(\pi_q)\le\e^\mu$. Remembering that $\diam(G)\ge\e^\nu$, and that $\nu\in(0, \mu)$, 
    by triangular inequality we get (taking also into account the links between the end points of the path $\partial^+I\cap S_\e$ with $p^{I}$ and $q^{I}$ respectively)   
    \begin{equation}\label{eq:mah_0}
        \#(\partial^+\I\cap S_\e)\ge \frac{1}{\e}\Bigl(d^\e_\infty(p^B, q^B)-\len(\pi_p)-\len(\pi_q)-2\e\Bigl)\ge\frac{1}{\e}\Bigl(d^\e_\infty(p^B, q^B)-c\e^\mu-2\e\Bigl)\ge c\frac{\e^\nu}{\e},
    \end{equation}
    for some positive constant $c$ depending only on the choice of $\nu;$ in particular either 
    \begin{equation}\label{eq:3.32}
    \#(\partial^+\I\cap S_\e\cap\Z_1)\ge\frac{c\e^\nu}{2\e}\quad\text{ or }\quad\#(\partial^+\I\cap S_\e\cap\{u_1=-1\})\ge\frac{c\e^\nu}{2\e}.
    \end{equation}
	Now we claim that if the second case of \eqref{eq:3.32} occurs, then $\E_\e(\tilde{u}_1)<\E_\e(u_1)$, which contradicts the minimality of $u_1$ because  $\dis^1(\tilde{u}_1, u_0)\le\dis^1(u_1, u_0)$. The proof of this fact follows the one given in Proposition \ref{prop:hausdorff_distance_from_boundary}. With the same strategy used to prove \eqref{eq:0} we find 
    \begin{equation}\label{eq:energy_estimate_100}
        \E_\e(u_1)-\E_\e(\tilde{u}_1) = \E_\e(u_1, S_\e, S_\e\cup\partial^+S_\e)) -  \E_\e(\tilde{u}_1, S_\e, S_\e\cup\partial^+S_\e)) - \E_\e(\tilde{u}_1, \hat{Q}).
    \end{equation}  
	It holds
	\begin{equation}\label{eq:energy_tilde}
		\begin{aligned}
		&\E_\e(\tilde{u}_1, S_\e, S_\e\cup\partial^+S_\e) = \E_\e(\tilde{u}_1, \partial^-S_\e, \partial^+S_\e) \le\\ &\le\E_\e(\tilde{u}_1, \partial^-S_\e, \partial^-G) + \E_\e(\tilde{u}_1, \partial^-S_\e, \pi_p) + \E_\e(\tilde{u}_1, \partial^-S_\e, \pi_q)\le \E_\e(\tilde{u}_1, \partial^-S_\e, \partial^-G) + c\e^\mu,
		\end{aligned}
	\end{equation}
	where in the first equality we used that $S_\e\subset\{\tilde{u}_1=1\}$, in the last inequality we used that $\max\{\len(\pi_p), \len(\pi_q)\}\le\e^\mu$. Moreover we have
	\begin{equation}\label{eq:energy_u_1}
		\begin{aligned}
		&\E_\e(u_1, S_\e, S_\e\cup\partial^+ S_\e)\ge \E_\e(u_1, S_\e\cap\{u_1 = -1\}, \partial^-G) +\\ & + \E_\e(u_1, S_\e\cap\partial^+I\cap \{u_1 = -1\}, \partial^-I) + \E_\e(u_1, S_\e\cap Z_1, S_\e\cup\partial^+S_\e).
		\end{aligned}
	\end{equation}
	We observe that since $G$ is a connected component of $I_1\cup\Z_1$ then $\partial^+G\subset\{u_1 = -1\}$. It follows that 
	\begin{equation}\label{eq:boundary_G}
	\E_\e(u_1, S_\e\cap\{u_1 = -1\}, \partial^-G) - \E_\e(\tilde{u}_1, \partial^-S_\e, \partial^-G)\ge 0.	\end{equation}
	Putting all together, from \eqref{eq:energy_estimate_100}, \eqref{eq:energy_tilde}, \eqref{eq:energy_u_1} and \eqref{eq:boundary_G} we deduce that 
	\begin{equation}\label{eq:intermediate_estimate}
		\begin{aligned}
		\E_\e(u_1) - \E_\e(\tilde{u}_1) &\ge \E_\e(u_1, S_\e\cap\partial^+I\cap\{u_1 = -1\}, \partial^-I) + 
		\\
		& + \E_\e(u_1, S_\e\cap Z_1, S_\e\cup\partial^+S_\e) - c\e^\mu - \E_\e(\tilde{u}_1, \hat{Q}).
		\end{aligned}
	\end{equation}
	Now, if the second case of \eqref{eq:3.32} occurs, then 
	\begin{equation}\label{eq:intermedia}\E_\e(u_1, S_\e\cap\partial^+I\cap\{u_1 = -1\}, \partial^-I)\ge c\e^\nu.
	\end{equation} Since $\hat{Q}$ is a set of minimal perimeter among all the subsets of $\e\ZZ^2$ having exactly $z$ cells, from Lemma \ref{lemma:energy_surfactant} we deduce 
	\begin{equation}\label{eq:zero_estimate}
		\E_\e(u_1, S_\e\cap Z_1, S_\e\cup\partial^+S_\e) - \E_\e(\tilde{u}_1, \hat{Q})\ge 0.
	\end{equation}
	In particular from \eqref{eq:intermediate_estimate}, \eqref{eq:intermedia} and \eqref{eq:zero_estimate} we deduce 
	\[\E_\e(u_1) - \E_\e(\tilde{u}_1) \ge c\e^\nu - c\e^\mu>0\]
	for $\e$ sufficiently small, which is a contradiction with the minimality of $u_1$.
	
	Suppose instead that the first situation of \eqref{eq:3.32} occurs. It follows from Lemma \ref{lemma:energy_surfactant} that 
	\begin{equation}\label{eq:boundary_I}
		\E_\e(u_1, S_\e\cap Z_1, S_\e\cup\partial^+ S_\e)\ge 2z\e(1-k) + c\e\#(\partial^+I\cap S_\e\cap Z_1) \ge 2z\e(1-k) +c\e^\nu. 
	\end{equation}
	Again from Lemma \ref{lemma:energy_surfactant} we have 
    \begin{equation}\label{eq:en_estim}\E_\e(\tilde{u}_1,\hat{Q})\le 2z(1-k)\e+c\e\sqrt{z}.
    \end{equation}
    By the minimality of $u_1$, we have $\E_\e(u_1)\le\E_\e(\tilde{u}_1)$, and then we conclude from \eqref{eq:intermediate_estimate},  \eqref{eq:boundary_I} and \eqref{eq:en_estim} that 
    \begin{equation}\label{eq:conclusion}
    0\geq \E_\e(u_1)-\E_\e(\tilde{u}_1) \ge c\e^\nu-c\e^\mu - c\e\sqrt{z}\implies z\ge c\e^{2\nu-2},
    \end{equation}
    where we used that $c\e^\nu-c\e^\mu\ge c\e^\nu$ for $\e$ sufficiently small, since $\nu<\mu$. But the conclusion of \eqref{eq:conclusion} is absurd since, recalling that $z=\#(Z_1\cap S_\e)$ ,
    \[z\le\#S_\e\le c \#\Bigl((\partial^+I)\cap S_\e\Bigl)\frac{d_\hh(\partial A_1, \partial A_I)}{\e}\le c\e^{\nu+\mu-2}<c\e^{2\nu-2}.\]
    \vspace{5pt}\\
    \emph{Step 2.} Suppose that $C$ is a rectangle. Then, in view of Lemma \ref{lemma:optimal_position_connected_components}-(ii), the set $C$ cannot contain any weakly connected component $B_i$ with diameter smaller than $\e^\nu,$ and arguing similarly to the previous step we deduce immediately that $C$ cannot contain components $B_i$ with diameter larger than $\e^\nu$ either. Therefore $I_1$ is connected, and in this case the proof is complete.
    
    From now on we may suppose that $C$ is not a rectangle. We set $T := R_I\setminus\I = \cup_{j=1}^4T_j$ as in Definition \ref{def:staircase_set}. We replace $T_j$ with $T_j\cap\I_0$ without renaming it. From Lemma \ref{lemma:distance_horizontal_sides} for every point $p\in I_0\setminus R_I$ it holds that \begin{equation}\label{eq:3.31}
        d^\e_1(p, \partial I_0)\le d_{\mathcal{H}}(R_{I_0}, R_{I})\le c\e.
    \end{equation}
    In this step we claim that every component $B_i$ is contained in one of the $T_j$.
    We argue by contradiction. Suppose that, for instance, $B_1\not\subset R_I$. Let $H^{B,1}:=\dseg{p^{B, h, 1}}{q^{B, h, 1}}$ be the lowest horizontal slice of $B_1$ which we can suppose that lies below $R_I$. Up to replacing $B_1$ with a different component, we can assume that 
    \begin{equation}\label{eq:3.40}
        \{q^{B,h, 1}-n\e e_1-m\e e_2:\:(n, m)\in\NN\times\NN\}\cap\I_1 = \emptyset,\text{ and }\{p^{B,h, 1}-n\e e_1:n\in\NN\}\cap\I_1 = \emptyset.
    \end{equation}
    Let $N:=\#H^{B,1}$, and suppose first that $\dseg{p^{B,h, 1}-\e e_2}{q^{B, h, 1}-\e e_2}\subset\{u_1=0\}.$ We prove that we can replace every point belonging to the slice $H^{B,1}$ with either the value $0$ or $-1$, in such a way that we lose enough energy to compensate the increment of dissipation.
    We construct a sequence $u_{1, i}$ for $i=0, \dots, N-1$ which at step $i$ replaces $u_1(p^{B, h, 1}+i\e e_1)$ with either the value $0$ or $-1$ as follows. \\
    There exists $z_0\in \Z_1$ which is below the slice $H^{B, 1},$ and such that $u_1(z_0-\e e_1) = u_1(z_0-\e e_2) = -1$. We define a new function $u_{1,0}$ by replacing $u_1(z_0)\mapsto -1$ and $u_1(p^{B,h, 1})\mapsto 0$ (see Figure \ref{fig:2Figure_Lemma3_11}).

    \begin{figure}[H]
          \centering
    \resizebox{0.60\textwidth}{!}{\tikzset{every picture/.style={line width=0.75pt}} 

\begin{tikzpicture}[x=0.75pt,y=0.75pt,yscale=-1,xscale=1]

\draw [color={rgb, 255:red, 208; green, 2; blue, 27 }  ,draw opacity=1 ][fill={rgb, 255:red, 208; green, 2; blue, 27 }  ,fill opacity=1 ]   (140.02,80) -- (140.02,110) -- (150.02,110) -- (150.02,120) -- (160.02,120) -- (160.02,150) -- (170.02,150) -- (170.02,160) -- (230.02,160) -- (230.02,170) -- (240.02,170) -- (240.02,200) -- (250.02,200) -- (250.02,150) -- (180.02,150) -- (180.02,110) -- (160.02,110) -- (160.02,80) -- cycle ;
\draw [color={rgb, 255:red, 208; green, 2; blue, 27 }  ,draw opacity=1 ][fill={rgb, 255:red, 208; green, 2; blue, 27 }  ,fill opacity=1 ]   (200.02,230) -- (160.02,230) -- (160.02,220) -- (150.02,220) -- (150.02,180) -- (190.02,180) -- (190.02,190) -- (210.02,190) -- (210.02,220) -- (200.02,220) -- cycle ;
\draw  [color={rgb, 255:red, 74; green, 144; blue, 226 }  ,draw opacity=1 ][fill={rgb, 255:red, 74; green, 144; blue, 226 }  ,fill opacity=1 ] (160.02,230) -- (200.02,230) -- (200.02,240) -- (160.02,240) -- cycle ;
\draw  [color={rgb, 255:red, 74; green, 144; blue, 226 }  ,draw opacity=1 ][fill={rgb, 255:red, 74; green, 144; blue, 226 }  ,fill opacity=1 ] (170.02,240) -- (200.02,240) -- (200.02,250) -- (170.02,250) -- cycle ;
\draw  [color={rgb, 255:red, 74; green, 144; blue, 226 }  ,draw opacity=1 ][fill={rgb, 255:red, 74; green, 144; blue, 226 }  ,fill opacity=1 ] (170.02,250) -- (190.02,250) -- (190.02,260) -- (170.02,260) -- cycle ;
\draw [color={rgb, 255:red, 208; green, 2; blue, 27 }  ,draw opacity=1 ][fill={rgb, 255:red, 208; green, 2; blue, 27 }  ,fill opacity=1 ]   (360.02,80) -- (360.02,110) -- (370.02,110) -- (370.02,120) -- (380.02,120) -- (380.02,150) -- (390.02,150) -- (390.02,160) -- (450.02,160) -- (450.02,170) -- (460.02,170) -- (460.02,200) -- (470.02,200) -- (470.02,150) -- (400.02,150) -- (400.02,110) -- (380.02,110) -- (380.02,80) -- cycle ;
\draw [color={rgb, 255:red, 208; green, 2; blue, 27 }  ,draw opacity=1 ][fill={rgb, 255:red, 208; green, 2; blue, 27 }  ,fill opacity=1 ]   (420.02,230) -- (390.02,230) -- (390.02,220) -- (380.02,220) -- (370.02,220) -- (370.02,180) -- (410.02,180) -- (410.02,190) -- (430.02,190) -- (430.02,220) -- (420.02,220) -- cycle ;
\draw  [color={rgb, 255:red, 74; green, 144; blue, 226 }  ,draw opacity=1 ][fill={rgb, 255:red, 74; green, 144; blue, 226 }  ,fill opacity=1 ] (380.02,230) -- (420.02,230) -- (420.02,240) -- (380.02,240) -- cycle ;
\draw  [color={rgb, 255:red, 74; green, 144; blue, 226 }  ,draw opacity=1 ][fill={rgb, 255:red, 74; green, 144; blue, 226 }  ,fill opacity=1 ] (390.02,240) -- (420.02,240) -- (420.02,250) -- (390.02,250) -- cycle ;
\draw  [color={rgb, 255:red, 74; green, 144; blue, 226 }  ,draw opacity=1 ][fill={rgb, 255:red, 74; green, 144; blue, 226 }  ,fill opacity=1 ] (400.02,250) -- (410.02,250) -- (410.02,260) -- (400.02,260) -- cycle ;
\draw  [color={rgb, 255:red, 74; green, 144; blue, 226 }  ,draw opacity=1 ][fill={rgb, 255:red, 74; green, 144; blue, 226 }  ,fill opacity=1 ] (380.02,220) -- (390.02,220) -- (390.02,230) -- (380.02,230) -- cycle ;
\draw  [color={rgb, 255:red, 126; green, 211; blue, 33 }  ,draw opacity=1 ][fill={rgb, 255:red, 126; green, 211; blue, 33 }  ,fill opacity=1 ] (390.02,250) -- (400.02,250) -- (400.02,260) -- (390.02,260) -- cycle ;
\draw  [draw opacity=0] (90.02,60) -- (500.44,60) -- (500.44,270.33) -- (90.02,270.33) -- cycle ; \draw  [color={rgb, 255:red, 155; green, 155; blue, 155 }  ,draw opacity=0.77 ] (90.02,60) -- (90.02,270.33)(100.02,60) -- (100.02,270.33)(110.02,60) -- (110.02,270.33)(120.02,60) -- (120.02,270.33)(130.02,60) -- (130.02,270.33)(140.02,60) -- (140.02,270.33)(150.02,60) -- (150.02,270.33)(160.02,60) -- (160.02,270.33)(170.02,60) -- (170.02,270.33)(180.02,60) -- (180.02,270.33)(190.02,60) -- (190.02,270.33)(200.02,60) -- (200.02,270.33)(210.02,60) -- (210.02,270.33)(220.02,60) -- (220.02,270.33)(230.02,60) -- (230.02,270.33)(240.02,60) -- (240.02,270.33)(250.02,60) -- (250.02,270.33)(260.02,60) -- (260.02,270.33)(270.02,60) -- (270.02,270.33)(280.02,60) -- (280.02,270.33)(290.02,60) -- (290.02,270.33)(300.02,60) -- (300.02,270.33)(310.02,60) -- (310.02,270.33)(320.02,60) -- (320.02,270.33)(330.02,60) -- (330.02,270.33)(340.02,60) -- (340.02,270.33)(350.02,60) -- (350.02,270.33)(360.02,60) -- (360.02,270.33)(370.02,60) -- (370.02,270.33)(380.02,60) -- (380.02,270.33)(390.02,60) -- (390.02,270.33)(400.02,60) -- (400.02,270.33)(410.02,60) -- (410.02,270.33)(420.02,60) -- (420.02,270.33)(430.02,60) -- (430.02,270.33)(440.02,60) -- (440.02,270.33)(450.02,60) -- (450.02,270.33)(460.02,60) -- (460.02,270.33)(470.02,60) -- (470.02,270.33)(480.02,60) -- (480.02,270.33)(490.02,60) -- (490.02,270.33)(500.02,60) -- (500.02,270.33) ; \draw  [color={rgb, 255:red, 155; green, 155; blue, 155 }  ,draw opacity=0.77 ] (90.02,60) -- (500.44,60)(90.02,70) -- (500.44,70)(90.02,80) -- (500.44,80)(90.02,90) -- (500.44,90)(90.02,100) -- (500.44,100)(90.02,110) -- (500.44,110)(90.02,120) -- (500.44,120)(90.02,130) -- (500.44,130)(90.02,140) -- (500.44,140)(90.02,150) -- (500.44,150)(90.02,160) -- (500.44,160)(90.02,170) -- (500.44,170)(90.02,180) -- (500.44,180)(90.02,190) -- (500.44,190)(90.02,200) -- (500.44,200)(90.02,210) -- (500.44,210)(90.02,220) -- (500.44,220)(90.02,230) -- (500.44,230)(90.02,240) -- (500.44,240)(90.02,250) -- (500.44,250)(90.02,260) -- (500.44,260)(90.02,270) -- (500.44,270) ; \draw  [color={rgb, 255:red, 155; green, 155; blue, 155 }  ,draw opacity=0.77 ]  ;
\draw    (140.02,80) -- (140.02,200) -- (250.02,200) ;
\draw    (360.02,80) -- (360.02,200) -- (470.02,200) ;
\draw    (140.02,110) -- (150.02,110) -- (150.02,120) -- (160.02,120) -- (160.02,150) -- (170.02,150) -- (170.02,160) -- (230.02,160) -- (230.02,170) -- (240.02,170) -- (240.02,200) ;
\draw    (360.02,110) -- (370.02,110) -- (370.02,120) -- (380.02,120) -- (380.02,150) -- (390.02,150) -- (390.02,160) -- (450.02,160) -- (450.02,170) -- (460.02,170) -- (460.02,200) ;
\draw    (260.02,140) -- (328.02,140) ;
\draw [shift={(330.02,140)}, rotate = 180] [color={rgb, 255:red, 0; green, 0; blue, 0 }  ][line width=0.75]    (10.93,-3.29) .. controls (6.95,-1.4) and (3.31,-0.3) .. (0,0) .. controls (3.31,0.3) and (6.95,1.4) .. (10.93,3.29)   ;

\draw (172.02,183.4) node [anchor=north west][inner sep=0.75pt]  [font=\footnotesize]  {$B_{1}$};
\draw (382.02,183.4) node [anchor=north west][inner sep=0.75pt]  [font=\footnotesize]  {$B_{1}$};
\draw (162,218.4) node [anchor=north west][inner sep=0.75pt]  [font=\scriptsize]  {$p^{B,h,1}$};
\draw (192,218.4) node [anchor=north west][inner sep=0.75pt]  [font=\scriptsize]  {$q^{B,h,1}$};
\draw (170,252.4) node [anchor=north west][inner sep=0.75pt]  [font=\scriptsize]  {$z_{0}$};
\draw (390,252.4) node [anchor=north west][inner sep=0.75pt]  [font=\scriptsize]  {$z_{0}$};
\draw (412.02,217.4) node [anchor=north west][inner sep=0.75pt]  [font=\scriptsize]  {$q^{B,h,1}$};
\draw (192.02,73.4) node [anchor=north west][inner sep=0.75pt]    {$u_{1}$};
\draw (422.02,73.4) node [anchor=north west][inner sep=0.75pt]    {$u_{1,0}$};
\draw (382.02,218.4) node [anchor=north west][inner sep=0.75pt]  [font=\scriptsize]  {$p^{B,h,1}$};
\draw (142.02,93.4) node [anchor=north west][inner sep=0.75pt]    {$I$};
\draw (122.02,163.4) node [anchor=north west][inner sep=0.75pt]    {$R_{I}$};
\draw (342.02,163.4) node [anchor=north west][inner sep=0.75pt]    {$R_{I}$};
\draw (362.02,93.4) node [anchor=north west][inner sep=0.75pt]    {$I$};

\end{tikzpicture}}
    \caption{The case $\dseg{p^{B,h, 1}-\e e_2}{q^{B, h, 1}-\e e_2}\subset\{u_1=0\}$. On right side an example of how to get the first element of the sequence $u_{1,i}$.}
    \label{fig:2Figure_Lemma3_11}
   \end{figure}
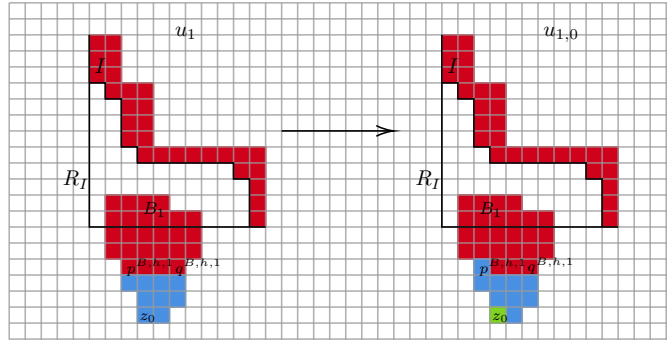
    
    We have that $\E_\e(u_{1,0})\le \E_\e(u_1).$ We proceed inductively for $i=1, \dots, N-1$, constructing a sequence of functions $u_{1,i}$. At step $i$ we find an element $z_i\in \{u_{1, i-1}=0\}$ which is below the slice $H^{B, 1}$ and such that $u_{1, i-1}(z_i-\e e_1) = u_{1, i-1}(z_i-\e e_2) = -1$, and we define $u_{1,i}$ replacing $u_{1,i-1}(p^{B,h, 1}+i\e e_1)\mapsto 0$ and $u_{1,i-1}(z_i)\mapsto -1$. We observe that $\E_\e(u_{1, i})\le \E_\e(u_{1, i-1})$, and that for $i=N-1$ it holds \begin{equation}\label{eq:energy_estim}
        \E_\e(u_{1,N-1})\le\E_\e(u_{1,N-2})-\e(1-k)\implies \E_\e(u_{1,N-1})\le\E_\e(u_1)-\e(1-k).
    \end{equation}
    Moreover, from the first step of this proof we have $|q^{B, h, 1}-p^{B,h, 1}| = (N-1)\e\le\e^\nu,$ therefore we can estimate the dissipation of $u_{1,N-1}$ with   \begin{equation}\label{eq:3.energy_diss_estimate} \dis^1_\e(u_{1,N-1}, u_0)-\dis^1_\e(u_1, u_0)\le \frac{(N-1)\e\cdot c\e}{\zeta}\le \frac{c}{\zeta}\e^{\nu+1},
    \end{equation}
    where constant $c$ is the one in (\ref{eq:3.31}). From \eqref{eq:energy_estim} and \eqref{eq:3.energy_diss_estimate} we deduce that $\mathcal{F}^{\tau, \gamma}_\e(u_{1,N-1})<\mathcal{F}^{\tau, \gamma}_\e(u_1)$, which contradicts the minimality of $u_1.$ 

    Suppose now that there exists $r\in\dseg{p^{B,h, 1}-\e e_2}{q^{B,h, 1}-\e e_2}\cap\{u_1=-1\}.$ From Lemma \ref{lemma:surfactant_placement} we have that every $p\in\Z_1$ satisfies 
    \begin{equation}\label{eq:3.50}
        \#(\mathcal{N}(p)\cap\I_1) \ge 1.
    \end{equation}
    Arguing as in the proof of Lemma \ref{lemma:distance_horizontal_sides} there exist two points $p', q'$ such that 
    \begin{equation}\label{eq:dseg}
    	\dseg{p^{B,h, 1}-\e e_2}{q^{B,h, 1}-\e e_2}\cap Z_1 = \dseg{p'}{q'}.
    \end{equation}
    From \eqref{eq:3.50} and \eqref{eq:3.40} it follows \begin{equation}\label{eq:3.34}
    \dseg{p^{B,h, 1}-2\e e_2}{q^{B,h, 1}-2\e e_2}\cup\{p^{B,h, 1}-\e e_1-\e e_2\}\cup\{q^{B, h, 1}+\e e_1-\e e_2\}\subset\{u_1=-1\}.
    \end{equation}
    Therefore, we may translate the segment $\dseg{p'}{q'}$ in such a way that $p' = p^{B, h, 1}$, without changing the energy of $u_1$. More precisely, we first define the competitor 
    \begin{equation}\label{eq:new_minimizer}
    	u_1'(p):=\begin{cases}
    		0&\text{ if }p\in \dseg{p^{B, h, 1}-\e e_2}{p^{B, h, 1}-\e e_2+(q'_1-p'_1)e_1},\\
    		-1&\text{ if }p\in \dseg{p^{B, h, 1}-\e e_2+(q'_1-p'_1+\e)e_1}{q^{B, h, 1}-\e e_2},\\
    		u_1(p)&\text{ otherwise}.	\end{cases}
    \end{equation}
    From \eqref{eq:dseg} and \eqref{eq:3.34} we deduce that $\enF(u_1', u_0) = \enF(u_1, u_0)$, therefore $u_1'$ is a minimizer. 
    Then, starting from $u'_1$ we define a new competitor $\tilde{u}_1$ setting 
        \begin{equation}\label{eq:tilde_competitor}
        	\tilde{u}_1(p) = \begin{cases}
            u_1'(p-\e e_2)&\text{ if }p\in\dseg{p^{B,h, 1}}{q^{B,h, 1}},\\
            -1&\text{ if }p\in\dseg{p^{B,h, 1}-\e e_2}{q^{B,h, 1}-\e e_2},\\
            u_1'(p)&\text{ otherwise}.
        \end{cases}
    \end{equation}
        Since $u_1'(p^{B, h, 1} -\e e_2) = 0$, and in view of \eqref{eq:3.34}, we have $\E_\e(\tilde{u}_1)-\E_\e(u_1)\le -\e(1-k)$ (see Figure \ref{fig:3Figure_Lemma3_11}).
        \begin{figure}[H]
          \centering
    \resizebox{0.70\textwidth}{!}{\input{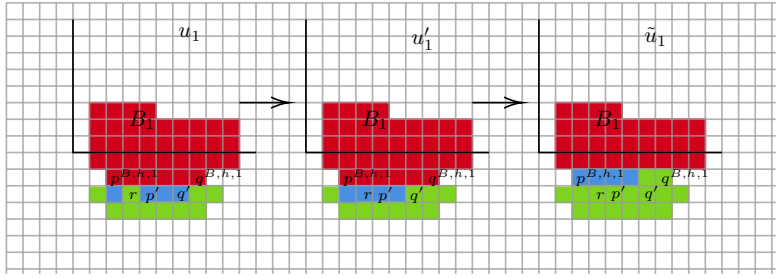}}
    \caption{The case $\dseg{p^{B,h, 1}-\e e_2}{q^{B,h, 1}-\e e_2}\cap\{u_1=-1\}\neq \emptyset$. An example of the construction of the competitor $\tilde{u}_1$ which reduces the total energy.}
    \label{fig:3Figure_Lemma3_11}
   \end{figure}
        
        Estimate (\ref{eq:3.energy_diss_estimate}) holds with $\tilde{u}_1$ in place of $u_{1, N-1}$, and therefore $\enF(\tilde{u}_1, u_0)<\enF(u_1, u_0)$, which contradicts the minimality of $u_1$.
        
        Finally, we may now suppose that $\dseg{p^{B, h, 1}-\e e_2}{q^{B, h, 1}-\e e_2}\cap Z_1 = \emptyset$. If $u_1(p^{B, h, 1}-\e e_1) = -1$ we define a competitor $\tilde{u}_1$ as in \eqref{eq:tilde_competitor} with $u_1$ in place of $u_1'$. It holds $\E_\e(\tilde{u}_1)\le\E_\e(u_1)-2\e$, and since estimate \eqref{eq:3.energy_diss_estimate} holds with $\tilde{u}_1$ in place of $u_{1, N-1}$ we contradict the minimality of $u_1$.
        
        Therefore, now we may suppose that $u_1(p^{B, h, 1}-\e e_1) = 0$. 
        Denote by $H^{B,2}:=\dseg{p^{B, h, 2}}{q^{B, h, 2}}$ the second horizontal slice of $B_1$. For simplicity we assume that $p^{B, h, 2}_1 < p^{B, h, 1}_1$, since the other situations are similar. We observe that \eqref{eq:3.40} and \eqref{eq:3.50} still hold. Therefore, we have that $\bar{p}:=p^{B, h, 2}-\e e_1-\e e_2$ satisfies  $u_1(\bar{p}) = -1$. Indeed $u_1(\bar{p})\neq 1$ because of (\ref{eq:3.40}), and $u_1(\bar{p})\neq 0$ again because of (\ref{eq:3.40}) and (\ref{eq:3.50}). Consider the segment $S:=\dseg{p^{B, h, 2}-\e e_2}{p^{B, h, 1}-\e e_1}\cap Z_1$. Let $n:=\#S$, and define the competitor $\tilde{u}_1$ by
        \[\tilde{u}_1(p) = \begin{cases}
        	-1&\text{ if }p\in\dseg{p^{B, h, 2}-\e e_2}{q^{B, h, 1}-n\e e_1}\\
            0&\text{ if }p\in \dseg{q^{B, h, 1}-(n-1)\e e_1}{q^{B, h, 1}},\\
            u_1(p)&\text{ otherwise},
        \end{cases}\]
        see Figure \ref{fig:4Figure_Lemma3_11}.

        \begin{figure}[H]
          \centering
          \resizebox{0.70\textwidth}{!}{\tikzset{every picture/.style={line width=0.75pt}} 

\begin{tikzpicture}[x=0.75pt,y=0.75pt,yscale=-1,xscale=1]

\draw [color={rgb, 255:red, 208; green, 2; blue, 27 }  ,draw opacity=1 ][fill={rgb, 255:red, 208; green, 2; blue, 27 }  ,fill opacity=1 ]   (150.02,170) -- (150.02,160) -- (140.02,160) -- (110.44,160.33) -- (110.44,120.33) -- (180.44,120.33) -- (180.02,130) -- (240.02,130) -- (240.02,160) -- (210.02,160) -- (210.02,170) ;
\draw  [color={rgb, 255:red, 126; green, 211; blue, 33 }  ,draw opacity=1 ][fill={rgb, 255:red, 126; green, 211; blue, 33 }  ,fill opacity=1 ] (150.44,170.33) -- (210.02,170.33) -- (210.02,180) -- (150.44,180) -- cycle ;
\draw  [color={rgb, 255:red, 126; green, 211; blue, 33 }  ,draw opacity=1 ][fill={rgb, 255:red, 126; green, 211; blue, 33 }  ,fill opacity=1 ] (200.02,170.33) -- (209.68,170.33) -- (209.68,180) -- (200.02,180) -- cycle ;
\draw  [color={rgb, 255:red, 74; green, 144; blue, 226 }  ,draw opacity=1 ][fill={rgb, 255:red, 74; green, 144; blue, 226 }  ,fill opacity=1 ] (120.02,160) -- (150.02,160) -- (150.02,170) -- (120.02,170) -- cycle ;
\draw  [color={rgb, 255:red, 126; green, 211; blue, 33 }  ,draw opacity=1 ][fill={rgb, 255:red, 126; green, 211; blue, 33 }  ,fill opacity=1 ] (100.44,160.33) -- (120.44,160.33) -- (120.44,170.33) -- (100.44,170.33) -- cycle ;
\draw [color={rgb, 255:red, 208; green, 2; blue, 27 }  ,draw opacity=1 ][fill={rgb, 255:red, 208; green, 2; blue, 27 }  ,fill opacity=1 ]   (290.44,160.33) -- (290.44,120.33) -- (360.02,121) -- (360.02,131) -- (420.02,131) -- (420.44,160.33) ;
\draw  [color={rgb, 255:red, 126; green, 211; blue, 33 }  ,draw opacity=1 ][fill={rgb, 255:red, 126; green, 211; blue, 33 }  ,fill opacity=1 ] (330.02,171) -- (389.02,171) -- (389.02,180.33) -- (330.02,180.33) -- cycle ;
\draw  [color={rgb, 255:red, 126; green, 211; blue, 33 }  ,draw opacity=1 ][fill={rgb, 255:red, 126; green, 211; blue, 33 }  ,fill opacity=1 ] (380.02,171) -- (390.02,171) -- (390.02,181) -- (380.02,181) -- cycle ;
\draw  [color={rgb, 255:red, 126; green, 211; blue, 33 }  ,draw opacity=1 ][fill={rgb, 255:red, 126; green, 211; blue, 33 }  ,fill opacity=1 ] (300.02,161) -- (360.44,161) -- (360.44,170.33) -- (300.02,170.33) -- cycle ;
\draw  [color={rgb, 255:red, 126; green, 211; blue, 33 }  ,draw opacity=1 ][fill={rgb, 255:red, 126; green, 211; blue, 33 }  ,fill opacity=1 ] (280.44,161) -- (300.02,161) -- (300.02,170.33) -- (280.44,170.33) -- cycle ;
\draw  [color={rgb, 255:red, 74; green, 144; blue, 226 }  ,draw opacity=1 ][fill={rgb, 255:red, 74; green, 144; blue, 226 }  ,fill opacity=1 ] (360.44,160.33) -- (390.44,160.33) -- (390.44,170.33) -- (360.44,170.33) -- cycle ;
\draw  [draw opacity=0] (70.44,50.33) -- (461.44,50.33) -- (461.44,211.33) -- (70.44,211.33) -- cycle ; \draw  [color={rgb, 255:red, 155; green, 155; blue, 155 }  ,draw opacity=0.77 ] (70.44,50.33) -- (70.44,211.33)(80.44,50.33) -- (80.44,211.33)(90.44,50.33) -- (90.44,211.33)(100.44,50.33) -- (100.44,211.33)(110.44,50.33) -- (110.44,211.33)(120.44,50.33) -- (120.44,211.33)(130.44,50.33) -- (130.44,211.33)(140.44,50.33) -- (140.44,211.33)(150.44,50.33) -- (150.44,211.33)(160.44,50.33) -- (160.44,211.33)(170.44,50.33) -- (170.44,211.33)(180.44,50.33) -- (180.44,211.33)(190.44,50.33) -- (190.44,211.33)(200.44,50.33) -- (200.44,211.33)(210.44,50.33) -- (210.44,211.33)(220.44,50.33) -- (220.44,211.33)(230.44,50.33) -- (230.44,211.33)(240.44,50.33) -- (240.44,211.33)(250.44,50.33) -- (250.44,211.33)(260.44,50.33) -- (260.44,211.33)(270.44,50.33) -- (270.44,211.33)(280.44,50.33) -- (280.44,211.33)(290.44,50.33) -- (290.44,211.33)(300.44,50.33) -- (300.44,211.33)(310.44,50.33) -- (310.44,211.33)(320.44,50.33) -- (320.44,211.33)(330.44,50.33) -- (330.44,211.33)(340.44,50.33) -- (340.44,211.33)(350.44,50.33) -- (350.44,211.33)(360.44,50.33) -- (360.44,211.33)(370.44,50.33) -- (370.44,211.33)(380.44,50.33) -- (380.44,211.33)(390.44,50.33) -- (390.44,211.33)(400.44,50.33) -- (400.44,211.33)(410.44,50.33) -- (410.44,211.33)(420.44,50.33) -- (420.44,211.33)(430.44,50.33) -- (430.44,211.33)(440.44,50.33) -- (440.44,211.33)(450.44,50.33) -- (450.44,211.33)(460.44,50.33) -- (460.44,211.33) ; \draw  [color={rgb, 255:red, 155; green, 155; blue, 155 }  ,draw opacity=0.77 ] (70.44,50.33) -- (461.44,50.33)(70.44,60.33) -- (461.44,60.33)(70.44,70.33) -- (461.44,70.33)(70.44,80.33) -- (461.44,80.33)(70.44,90.33) -- (461.44,90.33)(70.44,100.33) -- (461.44,100.33)(70.44,110.33) -- (461.44,110.33)(70.44,120.33) -- (461.44,120.33)(70.44,130.33) -- (461.44,130.33)(70.44,140.33) -- (461.44,140.33)(70.44,150.33) -- (461.44,150.33)(70.44,160.33) -- (461.44,160.33)(70.44,170.33) -- (461.44,170.33)(70.44,180.33) -- (461.44,180.33)(70.44,190.33) -- (461.44,190.33)(70.44,200.33) -- (461.44,200.33)(70.44,210.33) -- (461.44,210.33) ; \draw  [color={rgb, 255:red, 155; green, 155; blue, 155 }  ,draw opacity=0.77 ]  ;
\draw    (100.02,60) -- (100.02,140) -- (250.02,140) ;
\draw    (280.02,61) -- (280.44,140.33) -- (430.44,140.33) ;
\draw    (240.44,100.33) -- (268.44,100.33) ;
\draw [shift={(270.44,100.33)}, rotate = 180] [color={rgb, 255:red, 0; green, 0; blue, 0 }  ][line width=0.75]    (10.93,-3.29) .. controls (6.95,-1.4) and (3.31,-0.3) .. (0,0) .. controls (3.31,0.3) and (6.95,1.4) .. (10.93,3.29)   ;
\draw   (120.44,170.33) .. controls (120.44,174.45) and (122.5,176.51) .. (126.62,176.51) -- (126.62,176.51) .. controls (132.5,176.51) and (135.44,178.57) .. (135.44,182.69) .. controls (135.44,178.57) and (138.38,176.51) .. (144.27,176.51)(141.62,176.51) -- (144.27,176.51) .. controls (148.38,176.51) and (150.44,174.45) .. (150.44,170.33) ;
\draw   (300.44,170.33) .. controls (300.44,174.45) and (302.5,176.51) .. (306.62,176.51) -- (306.62,176.51) .. controls (312.5,176.51) and (315.44,178.57) .. (315.44,182.69) .. controls (315.44,178.57) and (318.38,176.51) .. (324.27,176.51)(321.62,176.51) -- (324.27,176.51) .. controls (328.38,176.51) and (330.44,174.45) .. (330.44,170.33) ;

\draw (346.02,70.4) node [anchor=north west][inner sep=0.75pt]    {$\tilde{u}_{1}$};
\draw (331,160.4) node [anchor=north west][inner sep=0.75pt]  [font=\scriptsize]  {$p^{B,h,1}$};
\draw (382,160.4) node [anchor=north west][inner sep=0.75pt]  [font=\scriptsize]  {$q^{B,h,1}$};
\draw (336.02,121.4) node [anchor=north west][inner sep=0.75pt]  [font=\footnotesize]  {$B_{1}$};
\draw (292.02,149.4) node [anchor=north west][inner sep=0.75pt]  [font=\scriptsize]  {$p^{B,h,2}$};
\draw (412.02,148.4) node [anchor=north west][inner sep=0.75pt]  [font=\scriptsize]  {$q^{B,h,2}$};
\draw (310.02,185.4) node [anchor=north west][inner sep=0.75pt]  [font=\footnotesize]  {$S$};
\draw (151,159.4) node [anchor=north west][inner sep=0.75pt]  [font=\scriptsize]  {$p^{B,h,1}$};
\draw (202,159.4) node [anchor=north west][inner sep=0.75pt]  [font=\scriptsize]  {$q^{B,h,1}$};
\draw (192.02,73.4) node [anchor=north west][inner sep=0.75pt]    {$u_{1}$};
\draw (156.02,122.4) node [anchor=north west][inner sep=0.75pt]  [font=\footnotesize]  {$B_{1}$};
\draw (112.02,148.4) node [anchor=north west][inner sep=0.75pt]  [font=\scriptsize]  {$p^{B,h,2}$};
\draw (232.02,147.4) node [anchor=north west][inner sep=0.75pt]  [font=\scriptsize]  {$q^{B,h,2}$};
\draw (130.02,184.4) node [anchor=north west][inner sep=0.75pt]  [font=\footnotesize]  {$S$};

\end{tikzpicture}}
          \caption{The case $\dseg{p^{B, h, 1}-\e e_2}{q^{B, h, 1}-\e e_2}\cap Z_1 = \emptyset$. The configuration of $B_1$ on the right reduces the total energy.}
          \label{fig:4Figure_Lemma3_11}
        \end{figure}
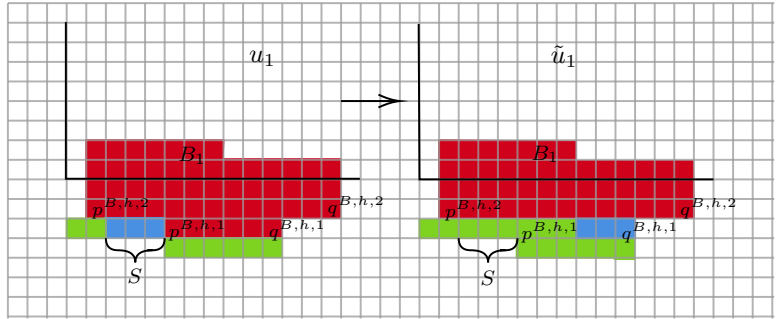

        In this case we have $\E_\e(\tilde{u}_1)\le\E_\e(u_1)-\e(1-k)$, and again estimate \eqref{eq:3.energy_diss_estimate} holds for $\tilde{u}_1$ in place of $u_{1,N-1}$, which contradicts the minimality of $u_1$.
    \vspace{5pt}\\
    \emph{Step 3.} Up to renaming the components $B_i$, we can suppose that $B_1, \dots, B_{n}$ are all contained in $T_1$ (which was defined at the beginning of Step 2). Since $B_i$ is strongly connected for every $i$, up to replacing $T_1$ with one of its strongly connected components, we can suppose that $T_1$ is strongly connected too. Since furthermore we assumed that $T_1\subset I_0$, in particular we have also that  \begin{equation}\label{eq:3.90}\partial^+\I\cap T_1\subset\I_0.\end{equation}
    We complete the proof under the assumption that the set $T_1$ coincides with the lower left component of $R_I\setminus I$, since the proof of the general case follows the same argument.
        Denote by $H^{\I,i} = \dseg{p^{\I, h, i}}{q^{\I, h, i}}:\:i=1, \dots, n_h,\,V^{\I, i} = \dseg{p^{\I, v, i}}{q^{\I, v, i}}:\:i=1, \dots, n_v$ the horizontal and vertical slices of $\I$. Suppose that there exist two points $\{p, q\}\subset(\partial^-I\cap\partial^+T_1)$ such that $q = p+\e e_1$ and $\{p-\e e_2, q-\e e_2\}\subset (\partial^+I\cap\partial^-T_1)$. In view of \eqref{eq:3.90}, and since we are assuming that $I_1\cup Z_1$ is not a rectangle, we deduce that case (ii) of Lemma \ref{lemma:shape_optimality} occurs. From the final claim of that Lemma we deduce that $\{p, q\}\subset H^{I, 2}$. Similarly, if there exist two points $\{p, q\}\subset(\partial^-I\cap\partial^+T_1)$ such that $q = p-\e e_2$ and $\{p-\e e_1, q-\e e_1\}\subset (\partial^+I\cap\partial^-T_1)$ then again from Lemma \ref{lemma:shape_optimality} we deduce that $\{p, q\}\subset V^{I, 2}$. As a geometrical consequence, it follows that there exists an index $l$ such that $p^{I, v, 2} = p^{I, h, l}$ and that 
    \begin{equation}\label{eq:boundary_regularity_horizontal}
    	p^{I, h, j} = p^{I, h, 2}+(j-2)(-\e e_1+\e e_2),\;\;\text{ for every }\;\;j\in 2, \dots, l,
    \end{equation}
    and similarly there exists an index $m$ such that $p^{I, h, 2} = p^{I, v, m}$ and that 
    \begin{equation}\label{eq:boundary_regularity_vertical}
    	p^{I, v, j} = p^{I, v, 2}+(j-2)(\e e_1-\e e_2),\;\;\text{ for every }\;\;j\in 2, \dots, m.
    \end{equation}
    Let $B:=\cup_{i=1}^{n}B_i$. Since by assumption $B$ is not empty, and since $I\cup B$ is not weakly connected, as a direct geometrical consequence we deduce that $l>2$. 
    For simplicity, we complete the proof under the assumption that $\#\dseg{p^{I, h, 2}-\e e_2}{p^{I, h, 1}-\e e_1} = 1 = \#\dseg{p^{I, v, 2}-\e e_1}{p^{I, v, 1}-\e e_2}$.
    
    Since $I_1\cup Z_1$ is strongly connected, there exists a path $\bar{\pi}$ of minimal length connecting $B$ and $I$ entirely contained in $I_1\cup Z_1$, i.e. there exist points $p^B\in B, p^I\in I$ and $\bar{\pi}\in\Gamma^s_{p^B,p^\I}(\Z_1\cup\{p^B, p^\I\})$. Suppose first that there exists an index $j\in \{1, \dots, l-1\}$ such that $p^{I, h, j}-\e e_1\in\bar{\pi}.$ Since $\bar{\pi}$ is strongly connected, without loss of generality we can suppose that $p^{I, h, j}-2\e e_1\in Z_1.$ We have to distinguish between the case $j=1$ and $j\neq 1$. In this latter case, if also $p^{I, h, j}-\e e_1-\e e_2\in Z_1$ then the cell $p^{I, h, j}-\e e_1$ is a surfactant in the corner (Definition \ref{def: corner unit}) and therefore, in view of Lemma \ref{lemma:shape_optimality}-1.1), we deduce that every weakly connected component of $Z_1\cup I_1$ is a rectangle, which is a contradiction with our initial assumption. We may suppose therefore that $u_1(p^{I, h, j}-\e e_1-\e e_2) = -1.$ Now, in view of \eqref{eq:boundary_regularity_horizontal} together with the assumption $j\ge 2$, we have that $p^{I, h, j}-\e e_2\in\zero_1\cap I_0$, and therefore, from Lemma \ref{lemma:shape_optimality}-1.2) we have that $p^{I, h, j}-\e e_2\in Z_1$. Now, by replacing $u_1(p^{I, h, j}-\e e_1-\e e_2)\mapsto 0$ and $u_1(p^{I, h, j}-\e e_1)\mapsto 1$ (see Figure \ref{fig:5Figure_Lemma3_11}) we strictly reduce the dissipation without increasing the energy, which contradicts the minimality of $u_1$.

       \begin{figure}[H]
          \centering
          \resizebox{0.70\textwidth}{!}{\input{small_components_final_part}}
          \caption{An example of substitution in the case $l=11$ and $j=8$, with $u_1(p^{I,h,j}-\e e_1-\e e_2)=-1$.}
          \label{fig:5Figure_Lemma3_11}
        \end{figure}
    Suppose now that $j=1.$ Since $l\ge 3$, we deduce that $p^{I, h, 2}-\e e_1\in\zero_1\cap I_0$ and therefore, again by Lemma \ref{lemma:shape_optimality}-1.2), it holds $p^{I, h, 2}-\e e_1\in Z_1.$ We conclude with the same reasoning as above.

    We are left to consider the case in which for every index $j\in\{1, \dots, l-1\}$ it holds $p^{I, h, j}-\e e_1\not\in\bar{\pi}.$ In particular $\bar{\pi}\setminus R_I\neq\emptyset$. We observe that for every $c>0$ there exists $\e(c)$ such that for every $\e\le\e(c)$ and for every $q\in B$ it holds \[
        d^\e_1(q, \partial I_0)\ge c \e,\]
    since otherwise we could argue as in step two, starting from \eqref{eq:3.31}, ``removing'' either the lowest or the leftmost slice of $B$, contradicting the minimality of $u_1$. This last consideration together with the fact that $\bar{\pi}\setminus R_I\neq\emptyset$ implies that there exists a point $\bar{p}\in\bar{\pi}\setminus R_I$ such that $\#\nn(\bar{p})\cap\{u_1 = -1\}\ge 2$ and $\nn(\bar{p})\cap I_1 = \emptyset.$ Consider now the point $\bar{q}:=p^{I, h, 1}-2\e e_1$. If $\bar{q}\in Z_1$ then we can conclude with the same argument above. If instead $u_1(\bar{q}) = -1$ then we can define a new competitor $\tilde{u}_1$ by replacing $u_1(\bar{p})\mapsto -1$ and $u_1(\bar{q})\mapsto 0$. We have that $\tilde{u}_1$ is a minimizer of $\enF(\cdot, u_0)$. Even in this case we find a contradiction with the same argument above.
\end{proof}

\section{The case \texorpdfstring{$\gamma >2$}{gamma>2}}\label{sec:gamma>2}
In this section we study the minimizing movements as in Definition \ref{def:minimizing_movement} assuming that $\gamma>2$. We will always assume that the starting set $I^\e_0$ associated to $u^\e_0$ is a discrete octagon, in the sense of the following definitions.
\begin{definition}\label{def:wulff_shape}
	[Wulff-type shape] We denote by $\mathcal{W}$ the set of all the convex octagons contained in $\rr^2$ whose sides are parallel to $e_1, e_2, e_1+e_2$ or $e_1-e_2$. We say that an element of $\mathcal{W}$ is \emph{non degenerate} if it has exactly eight sides of positive length. For simplicity, we refer to these Wulff-type shapes as \emph{octagons}. For $i=1, \dots, 4,$ we denote by $P_i$ (resp. $D_i$) the length of the sides of the octagon which are parallel either to $e_1$ or $e_2$ (resp. either to $e_1 + e_2$ or $e_1-e_2$). We set $P_1$ (resp. $D_1$) to be the length of the lower side parallel to $e_1$ (resp. to $e_1 - e_2$) and we label all lengths $P_i$ and $D_i$ in clockwise order.
\end{definition}

\begin{definition}\label{def:discrete_octagon}[Discrete Wulff-type shape] We say that $I\subset \e\ZZ^2$ is a discrete octagon if there exists $A\in\mathcal{W}$ such that $I = A\cap\e\ZZ^2$ (see Figure \ref{fig:disc_oct}).
    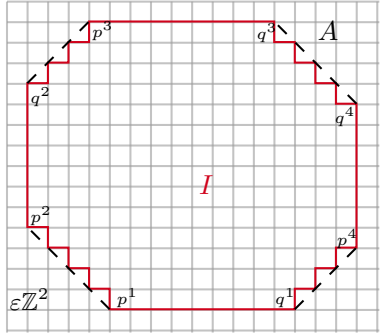
\begin{figure}[H]
       \centering
       \resizebox{0.35\textwidth}{!}{\tikzset{every picture/.style={line width=0.75pt}} 

\begin{tikzpicture}[x=0.75pt,y=0.75pt,yscale=-1,xscale=1]

\draw  [draw opacity=0] (260.44,90.33) -- (441.44,90.33) -- (441.44,251.33) -- (260.44,251.33) -- cycle ; \draw  [color={rgb, 255:red, 155; green, 155; blue, 155 }  ,draw opacity=0.6 ] (260.44,90.33) -- (260.44,251.33)(270.44,90.33) -- (270.44,251.33)(280.44,90.33) -- (280.44,251.33)(290.44,90.33) -- (290.44,251.33)(300.44,90.33) -- (300.44,251.33)(310.44,90.33) -- (310.44,251.33)(320.44,90.33) -- (320.44,251.33)(330.44,90.33) -- (330.44,251.33)(340.44,90.33) -- (340.44,251.33)(350.44,90.33) -- (350.44,251.33)(360.44,90.33) -- (360.44,251.33)(370.44,90.33) -- (370.44,251.33)(380.44,90.33) -- (380.44,251.33)(390.44,90.33) -- (390.44,251.33)(400.44,90.33) -- (400.44,251.33)(410.44,90.33) -- (410.44,251.33)(420.44,90.33) -- (420.44,251.33)(430.44,90.33) -- (430.44,251.33)(440.44,90.33) -- (440.44,251.33) ; \draw  [color={rgb, 255:red, 155; green, 155; blue, 155 }  ,draw opacity=0.6 ] (260.44,90.33) -- (441.44,90.33)(260.44,100.33) -- (441.44,100.33)(260.44,110.33) -- (441.44,110.33)(260.44,120.33) -- (441.44,120.33)(260.44,130.33) -- (441.44,130.33)(260.44,140.33) -- (441.44,140.33)(260.44,150.33) -- (441.44,150.33)(260.44,160.33) -- (441.44,160.33)(260.44,170.33) -- (441.44,170.33)(260.44,180.33) -- (441.44,180.33)(260.44,190.33) -- (441.44,190.33)(260.44,200.33) -- (441.44,200.33)(260.44,210.33) -- (441.44,210.33)(260.44,220.33) -- (441.44,220.33)(260.44,230.33) -- (441.44,230.33)(260.44,240.33) -- (441.44,240.33)(260.44,250.33) -- (441.44,250.33) ; \draw  [color={rgb, 255:red, 155; green, 155; blue, 155 }  ,draw opacity=0.6 ]  ;
\draw [color={rgb, 255:red, 208; green, 2; blue, 27 }  ,draw opacity=1 ][line width=0.75]    (270.33,130) -- (270.33,200) -- (280.33,200) -- (280.33,210) -- (290.33,210) -- (290.33,220) -- (300.33,220) -- (300.33,230) -- (310.33,230) -- (310.33,240) -- (400.33,240) -- (400.33,230) -- (410.33,230) -- (410.33,220) -- (420.33,220) -- (420.33,210) -- (430.33,210) -- (430.33,140) -- (420.33,140) -- (420.33,130) -- (410.33,130) -- (410.33,120) -- (400.33,120) -- (400.33,110) -- (390.33,110) -- (390.33,100) -- (300.33,100) -- (300.33,110) -- (290.33,110) -- (290.33,120) -- (280.33,120) -- (280.33,130) -- cycle ;
\draw  [dash pattern={on 4.5pt off 4.5pt}]  (270.33,130) -- (300.33,100) ;
\draw  [dash pattern={on 4.5pt off 4.5pt}]  (400.33,240) -- (430.33,210) ;
\draw  [dash pattern={on 4.5pt off 4.5pt}]  (310.33,240) -- (270.33,200) ;
\draw  [dash pattern={on 4.5pt off 4.5pt}]  (430.33,140) -- (390.33,100) ;

\draw (352.33,173.4) node [anchor=north west][inner sep=0.75pt]  [font=\normalsize,color={rgb, 255:red, 208; green, 2; blue, 27 }  ,opacity=1 ]  {$I$};
\draw (410,98.4) node [anchor=north west][inner sep=0.75pt]    {$A$};
\draw (389.33,228.4) node [anchor=north west][inner sep=0.75pt]  [font=\tiny]  {$q^{1}$};
\draw (270.33,188.4) node [anchor=north west][inner sep=0.75pt]  [font=\tiny]  {$p^{2}$};
\draw (300.33,98.4) node [anchor=north west][inner sep=0.75pt]  [font=\tiny]  {$p^{3}$};
\draw (418.33,140.4) node [anchor=north west][inner sep=0.75pt]  [font=\tiny]  {$q^{4}$};
\draw (260.33,228.4) node [anchor=north west][inner sep=0.75pt]  [font=\small]  {$\varepsilon\mathbb{Z}^2$};
\draw (312.33,228.4) node [anchor=north west][inner sep=0.75pt]  [font=\tiny]  {$p^{1}$};
\draw (270.33,130.4) node [anchor=north west][inner sep=0.75pt]  [font=\tiny]  {$q^{2}$};
\draw (380.33,99.4) node [anchor=north west][inner sep=0.75pt]  [font=\tiny]  {$q^{3}$};
\draw (419.33,199.4) node [anchor=north west][inner sep=0.75pt]  [font=\tiny]  {$p^{4}$};

\end{tikzpicture}}
       \caption{In red an example of discrete octagon $I$.}
    \label{fig:disc_oct}
    \end{figure}
Since a discrete octagon is a staircase set, we adopt notation \eqref{eq:parallel_sides_and_slices} to refer to the sides of $I$ which are parallel to $e_1$ and $e_2$. If $\PP_i = \dseg{p^i}{q^i}$ for $i=1, \dots, 4$, then we set $\DD_1:=\dseg{p^2}{p^1}$ to be the ``diagonal side'' of $I$ which is parallel to $e_1-e_2$, and we define $\DD_i:i=2, 3, 4$ analogously to be the other ``diagonal sides'' of $I,$ labeled clockwise. If $P_i$ and $D_i$ are the side lengths of the smallest octagon $A\in\mathcal{W}$ which contains $A_I$, in the following we say that $P_i$ (resp. $D_i$) is the length of $\PP_i$ (resp. $\DD_i$). 
\end{definition}

\begin{remark}\label{rem:geometry_octagon}
	Let $\I\subset\e\ZZ^2$ be a staircase set and define $\PP_i$ as in \eqref{eq:parallel_sides_and_slices}.
	Then, $\I$ is an octagon if and only if for every pair $p,\,q\in\partial^-\I$ such that $p=q+\e e_i$ for $i=1$ or $i=2$, there exists $j\in\{1, \dots, 4\}$ such that $\{p, q\}\subset \PP_j.$ 
\end{remark}

Next Lemma shows that the octagonal shape is preserved during the minimizing movements.

\begin{lemma}\label{lemma:octagonal minimizer} 
     Let $\gamma>2$. Let $u_0,u_1\in\A_\e$ be such that $u_0$ verifies assumption \eqref{H} of Subsection \ref{subsec:connectedness} and $u_1$ is a minimizer of $\enF(\cdot, u_0)$. Suppose furthermore that $\I_0$ is a discrete octagon in the sense of Definition \ref{def:discrete_octagon}. Then there exists $\bar{\e}$ such that for every $\e\le\bar{\e}$ the set $\I_1$ is a discrete octagon contained in $I_0$, and it holds $\partial^{+}\I_1=\Z_1$.
\end{lemma}
     \begin{proof}
         \emph{Step 1}. From Proposition  \ref{prop:connectedness} we have that $\I_1$ is a connected staircase set contained in $\I_0$. We prove that $\Z_1\subset\partial^+\I_1.$ In fact, if $\Z_1\setminus \partial^+\I_1\neq\emptyset$, then there exists $z\in\Z_1\setminus \partial^+\I_1$ such that $\#(\mathcal{N}(z)\cap\{u_1 = -1\})\ge 2.$ If we define a new competitor $\tilde{u}_1$ replacing $u_1(z)\mapsto -1$ then we have $\dis^1_\e(\tilde{u}_1, u_0) = \dis^1_\e(u_1, u_0) $ and, since $\tau = \zeta \e$ and since $\gamma>2$,
         \[\E_\e(\tilde{u}_1)-\E_\e(u_1)+\frac{\e^\gamma}{\tau}\left(\dis^0_\e(\tilde{u}_1, u_0)-\dis^0_\e(u_1, u_0)\right)\le -2\e(1-k)+\frac{\e^{\gamma-1}}{\zeta}<0\]
         for every $\e$ sufficiently small, which contradicts the minimality of $u_1$. \vspace{5pt}\\
         \emph{Step 2}. We prove that $\I_1$ is an octagon. From Proposition \ref{prop:connectedness} we have that $\I_1$ is a staircase set. Denote by $H^i = \dseg{p^{h, i}}{q^{h, i}}:\:i=1, \dots, n_h$ its horizontal slices.
         In view of Remark \ref{rem:geometry_octagon}, we need to show that for every $p, q\in\partial^-\I_1$ such that $q=p+\e e_i$ for $i=1, 2$ there exists $j\in\{1, \dots, 4\}$ such that $\{p, q\}\subset\PP_j$. We argue by contradiction. Without loss of generality, we may suppose that there exist $p, q\in H^j$ with $1< j< n_h,\,q = p+\e e_1$, and $\{p, q\}\subset \dseg{p^{h, j}}{p^{h, j-1}+\e e_2-\e e_1}$ so that, in particular, it holds $S:=\dseg{p^{h, j}-\e e_2}{p^{h, j-1}-\e e_1}\subset\{u_1\neq 1\}.$ Since $I_0$ is an octagon, since $\#S\ge 2$, and since $I_1\subset I_0$ we deduce that $S\not\subset\{u_0\neq 1\}$ (see the shape on the left of Figure \eqref{fig:octagonal minimizer}). Therefore, from Lemma \ref{lemma:shape_optimality}, we deduce that $S\subset\Z_1$. Moreover, from the first step we deduce that if $I_1\cup\Z_1$ is a rectangle, then $u_1(p^{h, j-1}-\e e_2) = -1$. If instead $I_1\cup\Z_1$ is not a rectangle, then by Lemma \ref{lemma:shape_optimality}-(ii) we deduce again that $u_1(p^{h, j-1}-\e e_2) = -1$. Summing up, we have 
         \begin{equation}\label{eq:fig_10}
         	S\subset\Z_1\;\; \text{ and }\;\;u_1(p^{h, j-1}-\e e_2) = -1.
         \end{equation}
      We consider the competitor $\tilde{u}_1$ obtained replacing $u_1(p^{h, j-1})\mapsto 0$ (see Figure \ref{fig:octagonal minimizer}).
         
    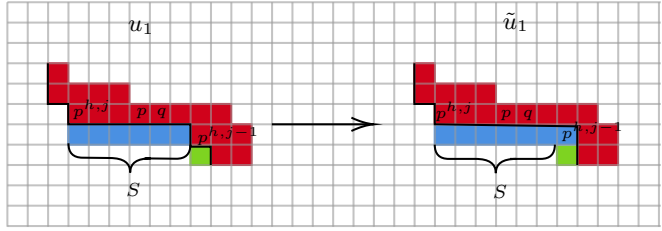
\begin{figure}[H]
    \centering    \resizebox{0.60\textwidth}{!}{\tikzset{every picture/.style={line width=0.75pt}} 

\begin{tikzpicture}[x=0.75pt,y=0.75pt,yscale=-1,xscale=1]

\draw [color={rgb, 255:red, 208; green, 2; blue, 27 }  ,draw opacity=1 ][fill={rgb, 255:red, 208; green, 2; blue, 27 }  ,fill opacity=1 ]   (100.33,120) -- (100.33,110) -- (90.33,110) -- (90.33,90) -- (100.33,90) -- (100.33,100) -- (130.33,100) -- (130.33,110) -- (180.33,110) -- (180.33,120) -- (190.33,120) -- (190.33,140) -- (170.33,140) -- (170.33,130) -- (160.33,130) -- (160.33,120) ;
\draw  [color={rgb, 255:red, 74; green, 144; blue, 226 }  ,draw opacity=1 ][fill={rgb, 255:red, 74; green, 144; blue, 226 }  ,fill opacity=1 ] (100.33,120) -- (160.33,120) -- (160.33,130) -- (100.33,130) -- cycle ;
\draw  [color={rgb, 255:red, 126; green, 211; blue, 33 }  ,draw opacity=1 ][fill={rgb, 255:red, 126; green, 211; blue, 33 }  ,fill opacity=1 ] (160.33,131) -- (170.33,131) -- (170.33,140) -- (160.33,140) -- cycle ;
\draw [color={rgb, 255:red, 208; green, 2; blue, 27 }  ,draw opacity=1 ][fill={rgb, 255:red, 208; green, 2; blue, 27 }  ,fill opacity=1 ]   (280.33,120) -- (280.33,109) -- (270.33,109) -- (270.33,90) -- (280.33,90) -- (280.33,100) -- (310.33,100) -- (310.33,110) -- (360.33,110) -- (360.33,120) -- (370.33,120) -- (370.33,140) -- (350.33,140) -- (350.33,129) -- (350.33,121) ;
\draw  [color={rgb, 255:red, 74; green, 144; blue, 226 }  ,draw opacity=1 ][fill={rgb, 255:red, 74; green, 144; blue, 226 }  ,fill opacity=1 ] (280.33,121) -- (350.33,121) -- (350.33,130) -- (280.33,130) -- cycle ;
\draw  [color={rgb, 255:red, 126; green, 211; blue, 33 }  ,draw opacity=1 ][fill={rgb, 255:red, 126; green, 211; blue, 33 }  ,fill opacity=1 ] (340.33,131) -- (350.33,131) -- (350.33,140) -- (340.33,140) -- cycle ;
\draw  [draw opacity=0] (70.44,60) -- (391.33,60) -- (391.33,170.33) -- (70.44,170.33) -- cycle ; \draw  [color={rgb, 255:red, 155; green, 155; blue, 155 }  ,draw opacity=0.6 ] (70.44,60) -- (70.44,170.33)(80.44,60) -- (80.44,170.33)(90.44,60) -- (90.44,170.33)(100.44,60) -- (100.44,170.33)(110.44,60) -- (110.44,170.33)(120.44,60) -- (120.44,170.33)(130.44,60) -- (130.44,170.33)(140.44,60) -- (140.44,170.33)(150.44,60) -- (150.44,170.33)(160.44,60) -- (160.44,170.33)(170.44,60) -- (170.44,170.33)(180.44,60) -- (180.44,170.33)(190.44,60) -- (190.44,170.33)(200.44,60) -- (200.44,170.33)(210.44,60) -- (210.44,170.33)(220.44,60) -- (220.44,170.33)(230.44,60) -- (230.44,170.33)(240.44,60) -- (240.44,170.33)(250.44,60) -- (250.44,170.33)(260.44,60) -- (260.44,170.33)(270.44,60) -- (270.44,170.33)(280.44,60) -- (280.44,170.33)(290.44,60) -- (290.44,170.33)(300.44,60) -- (300.44,170.33)(310.44,60) -- (310.44,170.33)(320.44,60) -- (320.44,170.33)(330.44,60) -- (330.44,170.33)(340.44,60) -- (340.44,170.33)(350.44,60) -- (350.44,170.33)(360.44,60) -- (360.44,170.33)(370.44,60) -- (370.44,170.33)(380.44,60) -- (380.44,170.33)(390.44,60) -- (390.44,170.33) ; \draw  [color={rgb, 255:red, 155; green, 155; blue, 155 }  ,draw opacity=0.6 ] (70.44,60) -- (391.33,60)(70.44,70) -- (391.33,70)(70.44,80) -- (391.33,80)(70.44,90) -- (391.33,90)(70.44,100) -- (391.33,100)(70.44,110) -- (391.33,110)(70.44,120) -- (391.33,120)(70.44,130) -- (391.33,130)(70.44,140) -- (391.33,140)(70.44,150) -- (391.33,150)(70.44,160) -- (391.33,160)(70.44,170) -- (391.33,170) ; \draw  [color={rgb, 255:red, 155; green, 155; blue, 155 }  ,draw opacity=0.6 ]  ;
\draw    (90.33,90) -- (90.33,110) -- (100.33,110) -- (100.33,120) -- (160.33,120) -- (160.33,131) -- (170.33,131) -- (170.33,140) ;
\draw    (270.33,90) -- (270.33,110) -- (280.33,110) -- (280.33,120) -- (350.33,121) -- (350.33,131) -- (350.33,140) ;
\draw    (200.33,120) -- (248.33,120) ;
\draw [shift={(250.33,120)}, rotate = 180] [color={rgb, 255:red, 0; green, 0; blue, 0 }  ][line width=0.75]    (10.93,-3.29) .. controls (6.95,-1.4) and (3.31,-0.3) .. (0,0) .. controls (3.31,0.3) and (6.95,1.4) .. (10.93,3.29)   ;
\draw  [line width=0.75]  (100.44,129.33) .. controls (100.39,134) and (102.7,136.36) .. (107.37,136.41) -- (120.83,136.56) .. controls (127.5,136.63) and (130.8,139) .. (130.75,143.67) .. controls (130.8,139) and (134.16,136.71) .. (140.83,136.78)(137.83,136.75) -- (153.26,136.92) .. controls (157.93,136.97) and (160.28,134.67) .. (160.33,130) ;
\draw  [line width=0.75]  (280.33,130) .. controls (280.33,134.67) and (282.66,137) .. (287.33,137) -- (300.33,137) .. controls (307,137) and (310.33,139.33) .. (310.33,144) .. controls (310.33,139.33) and (313.66,137) .. (320.33,137)(317.33,137) -- (332.33,137) .. controls (337,137) and (339.33,134.67) .. (339.33,130) ;

\draw (127.33,147.4) node [anchor=north west][inner sep=0.75pt]  [font=\scriptsize]  {$S$};
\draw (128.33,67.4) node [anchor=north west][inner sep=0.75pt]  [font=\small]  {$u_{1}$};
\draw (312.33,63.4) node [anchor=north west][inner sep=0.75pt]  [font=\small]  {$\tilde{u}_{1}$};
\draw (101.33,106.4) node [anchor=north west][inner sep=0.75pt]  [font=\tiny]  {$p^{h,j}$};
\draw (162.33,119.4) node [anchor=north west][inner sep=0.75pt]  [font=\tiny]  {$p^{h,j-1}$};
\draw (132.33,110.4) node [anchor=north west][inner sep=0.75pt]  [font=\tiny]  {$p$};
\draw (142.33,110.4) node [anchor=north west][inner sep=0.75pt]  [font=\tiny]  {$q$};
\draw (341.33,117.4) node [anchor=north west][inner sep=0.75pt]  [font=\tiny]  {$p^{h,j-1}$};
\draw (311.33,111.4) node [anchor=north west][inner sep=0.75pt]  [font=\tiny]  {$p$};
\draw (322.33,111.4) node [anchor=north west][inner sep=0.75pt]  [font=\tiny]  {$q$};
\draw (279.33,107.4) node [anchor=north west][inner sep=0.75pt]  [font=\tiny]  {$p^{h,j}$};
\draw (307.33,148.4) node [anchor=north west][inner sep=0.75pt]  [font=\scriptsize]  {$S$};

\end{tikzpicture}}
    \caption{The configuration on the right reduces the total energy.}
    \label{fig:octagonal minimizer}
\end{figure} From Proposition \ref{prop:hausdorff_distance_from_boundary}, $d_{\mathcal{H}}(\partial A_0,\partial A_1)\leq \e^\mu$, with $\mu \in (0,1/4)$, and by construction of $\tilde{u}_1$, we get \begin{align*}
     \E_\e(\tilde{u}_1)-&\E_\e(u_1) \le -2\e+3\e(1-k) = \e(1-3k)<0,\\         &\frac{1}{\tau}\Bigl(\dis^1_\e(\tilde{u}_1, u_0)-\dis^1_\e(u_1, u_0)\Bigl)\le \frac{\e^{\mu+1}}{\zeta},\\       &\frac{\e^\gamma}{\tau}\Bigl(\dis^0_\e(\tilde{u}_1, u_0)-\dis^0_\e(u_1, u_0)\Bigl)\le\frac{\e^{\gamma-1}}{\zeta},
     \end{align*}
     therefore $\enF(\tilde{u}_1)<\enF(u_1)$ as $\e\to 0$, which contradicts the minimality of $u_1$.\vspace{5pt}\\
     \emph{Step 3.} We are left to prove that $\partial^+I_1 = Z_1.$ Denote by $\PP_{1, i}:=\dseg{p^{1, i}}{q^{1, i}}$ the sides of $I_1$ parallel to the axes and by $\DD_{1,i}$ the sides of $I_1$ parallel to $e_1\pm e_2$. Consider the upper diagonal of $\DD _{1,1}$ defined as the discrete segment $\DD_{1, 1}^+=\dseg{p^{1, 2}-\e e_2}{p^{1, 1}-\e e_1}$, according to definition \eqref{def:discrete_segment}. If there exists $p\in\DD_{1, 1}^+\cap\{u_1 = -1\}$ then the competitor $\tilde{u}_1$ obtained by replacing $u_1(p)\mapsto 0$ satisfies $\E_\e(\tilde{u}_1) = \E_\e(u_1)-4\e+4\e(1-k)$ and $\e^\gamma/\tau\Bigl|\dis^0_\e(\tilde
     {u}_1, u_0) -\dis^0_\e(u_1, u_0)\Bigl| \le \e^{\gamma-1}/\zeta.$ Therefore we deduce $\enF(\tilde{u}_1)<\enF(u_1)$, which contradicts the minimality of $u_1$. Similarly we deduce that $\cup_{i=1}^4\DD_{1, i}^+\subset\Z_1.$ Finally, in order to prove that $(\PP_{1, 1}-\e e_2)\subset\Z_1,$ we can argue similarly as in the proof of the second step. 
\end{proof}

\begin{remark}\label{rmk: octagonal initial data}
    From Lemma \ref{lemma:octagonal minimizer} we deduce that, as long as $u_j$ satisfies the assumptions of Proposition \ref{prop:connectedness}, then $\Z_{j+1} = \partial^+\I_{j+1}$ for every $j\ge 1$. Therefore, for simplicity, from now on we suppose that $\#\Z_0= \#\partial^+\I_0$, i.e. our initial datum is an octagon completely surrounded by surfactant.
\end{remark}

\begin{definition}\label{def:sides_movement}
Let $\I_j,\,\I_{j+1}$ be two discrete octagons such that $\I_{j+1}\subset\I_j.$ Let $A_j$ and $A_{j+1}$ be the smallest octagons (in the sense of Definition \ref{def:wulff_shape}) containing $A_{\I_j}$ and $A_{\I_{j+1}}$ respectively. For $s=j,\,j+1$ denote by $\{\PP_{s, i}: i=1,\dots,4\}$ and $\{\DD_{s, i}:\:i=1, \dots, 4\}$ the parallel and the diagonal sides of $A_s$. 
We denote by $a_{j,i}$ the Euclidean distance between the (parallel) straight lines containing the sides $\PP_{j,i}$ and $\PP_{j+1,i}$, and by $b_{j,i}$ the Euclidean distance between the (parallel) straight lines containing the sides $\DD_{j,i}$ and $\DD_{j+1,i}$. We define $\alpha_{j,i}, \beta_{j,i} \in \NN\cup\{0\}$ as 
       \begin{equation}
         \begin{split}
           \label{eq: spostamenti}
           &\alpha_{j,i}=\frac{a_{j,i}}{\e},\ \quad i=1,\dots,4\\
           &\beta_{j,i}=\frac{\sqrt{2}b_{j,i}}{\e},\ \quad i=1,\dots,4.
         \end{split}
      \end{equation}
Denoting by $P_{s, i}$ and $D_{s, i}$ for $i=1, \dots, 4$ and $s= j, j+1$ the lengths of $\PP_{s,i}$ and $\DD_{s,i}$ respectively we have that
     \begin{equation}
     \label{eq: Lung. lati}
       \begin{split}
           &P_{j+1,i}=P_{j,i}+2a_{j,i}-\sqrt{2}(b_{j,i-1}+b_{j,i}) = P_{j, i}+2\e\alpha_{j, i}-\e(\beta_{j, i-1}+\beta_{j, i}),\\
           &D_{j+1,i}=D_{j,i}+2 b_{j,i}-\sqrt{2}(a_{j,i}+a_{j,i+1}) = D_{j,i}+\sqrt{2}\e \beta_{j,i}-\sqrt{2}\e(\alpha_{j,i}+\alpha_{j,i+1}),
       \end{split}
    \end{equation}
    where index $i\in\{1, 2, 3, 4\}$ is understood modulo $4$.
\end{definition}

\begin{theorem}\label{teo:free_surfactant_movement}
    Let $A\subset\rr^2$ be a non degenerate octagon, and
    for all $\e>0$, let $A_\e$ be a non degenerate octagon such that $\lim_{\e\rightarrow 0^+}d_\mathcal{H}(A_\e,A)=0$. 
    Let $(u_0^\e)_\e$ be such that $u_0^\e\in\A_\e$ for every $\e$, and $\I^\e_0= A_\e\cap\e\ZZ^2.$ Let $u^\e_j$ be a minimizing movement with initial datum $u^\e_0$, and parameter $\gamma>2$. For every $t\ge 0$ we denote by $A_\e(t)$ the set     
    \begin{equation*}
    A_\e(t) = \bigcup_{i\in\I^\e_{\lfloor t/\tau\rfloor}}Q_\e(i).
    \end{equation*}
    Then, it holds that $A_\e(t)$ converges as $\e\to 0$ locally uniformly in time to $A(t)$
    which is an octagon for every $t\ge 0$ satisfying $A(0) = A$. \\
    Denoting by $\PP_i(t),\,\DD_i(t)$ its parallel and diagonal sides, and by $P_i(t)$ and $D_i(t)$ their lengths, then every side of $A(t)$ moves inwards remaining parallel to the side itself. We denote by $v_{\PP_i}(t)$  and $v_{\DD_i}(t)$ the velocity of its parallel and diagonal sides respectively. The velocity of each side satisfies the following differential inclusion
    \begin{equation}\label{eq:free_surfactant_sides_velocity}
    \begin{aligned}
    &v_{\PP_i}(t)\begin{cases}
        =\frac{1}{\zeta}\bigg\lfloor\frac{2\zeta(1-k)}{P_i(t)}\bigg\rfloor&\text{ if }\frac{2\zeta(1-k)}{P_i(t)}\not\in\NN,\\\rule{0pt}{1.4em}
        \in \Bigl[\frac{2(1-k)}{P_i(t)}-\frac{1}{\zeta},\,\frac{2(1-k)}{P_i(t)}\Bigl]&\text{ otherwise.}
    \end{cases}\\
    &v_{\DD_i}(t)\begin{cases}
        =\frac{\sqrt{2}}{2\zeta}\bigg\lfloor\frac{2\sqrt{2}\zeta(1-k)}{D_i(t)}\bigg\rfloor&\text{ if }\frac{2\sqrt{2}\zeta(1-k)}{D_i(t)}\not\in\NN,\\\rule{0pt}{1.4em}
        \in\Bigl[\frac{2(1-k)}{D_i(t)}-\frac{\sqrt{2}}{2\zeta},\,\frac{2(1-k)}{D_i(t)}\Bigl]&\text{ otherwise}.
    \end{cases}
    \end{aligned}
    \end{equation}
\end{theorem}

\begin{proof}
Let $\bar{c}>0$ be a positive constant, and let $\bar{\e}$ be given by Proposition \ref{prop:connectedness}. 
Then from Proposition \ref{prop:connectedness} and Lemma \ref{lemma:octagonal minimizer}, we know that for every $\e\le\bar{\e}$ the set $\I^\e_{j+1}$ is an octagon contained into $\I^\e_j$, as long as $P^\e_{j,i}>\bar{c}$ for every $i = 1, \dots, 4.$ Moreover from Lemma \ref{lemma:octagonal minimizer} we know that $\Z_{j}^\e = \partial^+\I^\e_{j}$ for every $j$ and for $\e\le\bar{\e}$. As a result the competitors in the minimum problem will be assumed to be octagon of this type.  Moreover, since $A$ is not degenerate, we can suppose that $D^\e_{0,i}>0$ for every $i=1,\dots,4$.

From now on we omit the dependence on $\e$. For $s=j, j+1$ let us denote by $\PP_{s, i} = \dseg{p^{s, i}}{q^{s, i}}$ the sides of $I_s$ parallel to the axes, and by $P_{s, i}$ their lengths. In order to determine $\alpha_{j,1}$ we compare $u_{j+1}$ with two competitors $u^+_{j+1}$ and $u^-_{j+1}.$ We define 
\begin{equation}\label{eq:free_surfactant_second_competitor}
    u^-_{j+1}(p):=
    \begin{cases}
        0&\text{ if }p\in \PP_{j+1, 1},\\
        -1&\text{ if }p\in \PP_{j+1, 1}-\e e_2,\\
        u_{j+1}(p)&\text{ otherwise}.
    \end{cases}
\end{equation}
and, in case $\alpha_{j, 1} > 0$,
\begin{equation}\label{eq:free_surfactant_first_competitor}
    u^+_{j+1}(p):=
    \begin{cases}
        1&\text{ if }p\in \dseg{p^{j+1, 1}+\e e_1-\e e_2}{q^{j+1,1}-\e e_1-\e e_2},\\
        0&\text{ if }p\in \dseg{p^{j+1,1}+\e e_1-2\e e_2}{q^{j+1,1}-\e e_1-2\e e_2},\\
        u_{j+1}(p)&\text{ otherwise}
    \end{cases}
\end{equation}
(see Figure \ref{fig:gamma maggiore 2}).
\begin{figure}[H]
    \centering
    \resizebox{0.70\textwidth}{!}{\input{Figure1_Theorem_4.7}}
    \caption{On the right, the competitors $u^+_{j+1}$ and $u^-_{j+1}$ in order to determine $\alpha_{j,1}$.}
    \label{fig:gamma maggiore 2}
\end{figure}
Since $\Z_{j+1} = \partial^+\I_{j+1}$ we deduce that $\E_\e(u_{j+1}^+) = \E_\e(u_{j+1})+2\e(1-k)$ and $\E_\e(u_{j+1}^-) = \E_\e(u_{j+1})-2\e(1-k)$. Form Proposition \ref{prop:hausdorff_distance_from_boundary} for every $\mu<1/4$ we have that $d_{\mathcal{H}}(\partial A_j, \partial A_{j+1})\le\e^\mu.$ Since $\I_j$ and $\I_{j+1}$ are octagons, we deduce that
\[p_1^{j, 1}-\e^{\mu} \le p^{j+1, 1}_1\le p^{j, 1}_1+\e^{\mu};\ \;q^{j, 1}_1-\e^\mu\le q_1^{j+1, 1}\leq q_1^{j, 1}+\e^\mu.\]
We deduce that  $q_1^{j+1, 1}-p_1^{j+1, 1}\leq q^{j, 1}_1+\e^\mu-p_1^{j, 1}+\e^\mu=P_{j, 1}+2\e^\mu$ and \\$q_1^{j+1, 1}-p_1^{j+1, 1}\geq q_1^{j, 1}-\e^{\mu}-p_1^{j, 1}-\e^\mu=P_{j, 1}-2\e^{\mu}$, 
therefore we find the following estimates 
\begin{align}\label{eq:free_dissipation_estimate_1}
        &\frac{1}{\tau}\Bigl(\dis^1_\e(u_{j+1}^+, u_j)-\dis^1_\e(u_{j+1}, u_j)\Bigl)\le -\frac{(P_{j, 1}-2\e^\mu)\alpha_{j, 1}\e}{\zeta},\\\label{eq:free_dissipation_estimate_20}
    &\frac{1}{\tau}\Bigl(\dis^1_\e(u_{j+1}^-, u_j)-\dis^1_\e(u_{j+1}, u_j)\Bigl)\le \frac{(P_{j, 1}+2\e^\mu)(\alpha_{j, 1}+1)\e}{\zeta},
\end{align}
where the first inequality \eqref{eq:free_dissipation_estimate_1} holds in case $\alpha_{j, 1}>0.$ Moreover we have $\dis^0_\e(u_{j+1}^-, u_j) = \dis^0_\e(u_{j+1}^+, u_j) = \dis^0_\e(u_{j+1}, u_j)$. By minimality of $u_{j+1}$, from \eqref{eq:free_dissipation_estimate_1} we deduce that
\[2\e(1-k)\ge\frac{(P_{j, 1}-2\e^\mu)\alpha_{j,1}\e}{\zeta}\implies\alpha_{j, 1}\le\frac{2\zeta(1-k)}{P_{j,1}-2\e^\mu} = \frac{2\zeta(1-k)}{P_{j,1}}+\frac{2\e^\mu}{P_{j, 1}(P_{j, 1}-2\e^\mu)}\]
and from \eqref{eq:free_dissipation_estimate_20} we have that
\begin{equation}\label{eq:estimate_100}
    2\e(1-k)\le \frac{(P_{j, 1}+2\e^\mu)(\alpha_{j, 1}+1)\e}{\zeta}\implies \alpha_{j, 1}\ge \frac{2\zeta(1-k)}{P_{j, 1}+2\e^\mu}-1 = \frac{2\zeta(1-k)}{P_{j, 1}}-\frac{4\zeta(1-k)\e^\mu}{P_{j, 1}(P_{j, 1}+2\e^\mu)}.
\end{equation}
In particular, if $\alpha_{j, 1}>0,$ we deduce that 
\begin{equation}\label{eq:free_dissipation_alpha}
    \alpha_{j,1} \begin{cases}
    =\Bigl\lfloor\frac{2\zeta(1-k)}{P_{j, 1}}\Bigl\rfloor&\text{ if }\;\;\text{dist}(2\zeta(1-k)/P_{j, 1}, \NN)<\frac{4\zeta(1-k)\e^\mu}{P_{j, 1}(P_{j, 1}-2\e^\mu)},\\
    \in \Bigl\{\Bigl[\frac{2\zeta(1-k)}{P_{j, 1}}\Bigl],\Bigl[\frac{2\zeta(1-k)}{P_{j, 1}}\Bigl]-1\Bigl\}&\text{ otherwise}.
\end{cases}
\end{equation}
In case $\alpha_{j, 1} = 0,$ we deduce from \eqref{eq:estimate_100} that
\[\frac{2\zeta(1-k)}{P_{j, 1}}\le 1+\frac{4\zeta(1-k)\e^\mu}{P_{j, 1}(P_{j, 1}+2\e^{\mu})},\]
therefore \eqref{eq:free_dissipation_alpha} holds also in this case.
We may argue similarly for every $\alpha_{j,i}$. 

We now estimate $\beta_{j, 1}$ (and similarly every $\beta_{j, i}$). Suppose first that $\beta_{j, 1}>0$. Denoting by $\DD_{j+1, 1}:=\dseg{p^{j+1, 2}}{p^{j+1, 1}}$ and $\DD^+_{j+1, 1}:=\dseg{p^{j+1, 2}-\e e_2}{p^{j+1, 1}-\e e_1}$, we define $u_{j+1}^+$ and $u^-_{j+1}$ as
\begin{equation}
    u_{j+1}^+(p) = \begin{cases}
        1&\text{ if }\;p\in\DD_{j+1, 1}^+,\\
        0&\text{ if }\;p\in(\DD_{j+1, 1}^+-\e e_2)\cup\{p^{j+1, 2}-\e e_2-\e e_1\},\\
        u_{j+1}(p)&\text{ otherwise},
    \end{cases}
\end{equation}
and 
\begin{equation}
    u_{j+1}^-(p) = \begin{cases}
        -1&\text{ if }\;p\in(\DD_{j+1, 1}-\e e_2)\cup\{p^{j+1, 2}-\e e_1\},\\
        0&\text{ if }\;p\in \DD_{j+1, 1},\\
        u_{j+1}(p)&\text{ otherwise}.
    \end{cases}
\end{equation}
(see Figure \ref{fig:gamma_mag2_obliqui}).
\begin{figure}[H]
    \centering
    \resizebox{0.50\textwidth}{!}{\input{Figure2_Theorem_4.7}}
    \caption{On the right, the competitors $u^+_{j+1}$ and $u^-_{j+1}$ in order to determine $\beta_{j,1}$.}
    \label{fig:gamma_mag2_obliqui}
\end{figure}
\noindent A straightforward computation shows that 
\[\E_\e(u^+_{j+1}) = \E_\e(u_{j+1})+2\e(1-k),\;\text{ and }\;\E_\e(u^-_{j+1}) = \E_\e(u_{j+1})-2\e(1-k).\]
Moreover 
\[|\#\{u^+_{j+1}=0\}-\#Z_{j+1}| = |\#\{u^-_{j+1}=0\}-\#Z_{j+1}|= 1,\]
therefore 
\begin{equation}\label{eq:dis_estimate}
    \frac{\e^\gamma}{\tau}\dis^0_\e(u^\pm_{j+1}, u_j) \le \frac{\e^\gamma}{\tau}\dis^0_\e(u_{j+1}, u_j) +\frac{\e^{\gamma-1}}{\zeta}.
\end{equation}
Again, since $\I_j$ and $\I_{j+1}$ are octagons, and since $d_{\mathcal{H}}(\partial A_j, \partial A_{j+1})\le\e^\mu$, we deduce that 
$D_{j+1, 1}\ge D_{j, 1}-2\sqrt{2}\e^\mu$, and
\begin{align}\label{eq:free_dissipation_estimate_2}
        &\frac{1}{\tau}\Bigl(\dis^1_\e(u_{j+1}^+, u_j)-\dis^1_\e(u_{j+1}, u_j)\Bigl)\le -\frac{(D_{j, 1}-2\sqrt{2}\e^\mu)\beta_{j, 1}\e}{\sqrt{2}\zeta},\\\label{eq:free_dissipation_estimate_beta_2}
    &\frac{1}{\tau}\Bigl(\dis^1_\e(u_{j+1}^-, u_j)-\dis^1_\e(u_{j+1}, u_j)\Bigl)\le \frac{(D_{j, 1}+2\sqrt{2}\e^\mu)(\beta_{j, 1}+1)\e}{\sqrt{2}\zeta}.
\end{align}
In particular by minimality of $u_{j+1}$, we find that 
\[\frac{2\sqrt{2}\zeta(1-k)-\sqrt{2}\e^{\gamma-2}}{D_{j, 1}+2\sqrt{2}\e^\mu}\le\beta_1\le\frac{2\sqrt{2}\zeta(1-k)+\sqrt{2}\e^{\gamma-2}}{D_{j, 1}-2\sqrt{2}\e^\mu}.\]
We observe that $\e^{\gamma-2}\to 0$ as $\e\to 0$ since $\gamma>2$. As before we have
\begin{equation}\label{eq:free_dissipation_beta}
    \beta_{j, 1} \begin{cases}
    =\Bigl\lfloor\frac{2\sqrt{2}\zeta(1-k)}{D_{j, 1}}\Bigl\rfloor&\text{ if }\;\;\text{dist}\Bigl(\frac{2\sqrt{2}\zeta(1-k)}{D_{j, 1}}, \NN\Bigl)<\frac{\sqrt{2}\e^{\gamma-2}D_{j, 1}+8\zeta(1-k)\e^\mu}{D_{j, 1}(D_{j, 1}-2\sqrt{2}\e^\mu)},\\
    \in \Bigl\{\Bigl[\frac{2\sqrt{2}\zeta(1-k)}{D_{j, 1}}\Bigr],\Bigl[\frac{2\sqrt{2}\zeta(1-k)}{D_{j, 1}}\Bigl]-1\Bigl\}&\text{ otherwise}
\end{cases}
\end{equation}
and therefore the inwards velocity of the side $\DD_{j+1, 1}$ is given by $\beta_{j, 1}/\sqrt{2}$. In case $\beta_{j, 1} = 0,$ from \eqref{eq:free_dissipation_estimate_beta_2}, arguing as in the case $\alpha_{j, 1} = 0,$ we deduce \eqref{eq:free_dissipation_beta}.

Recursively repeating the argument above, we find sequences $D^\e_{j, i},\,P^\e_{j, i}$ and $\alpha^\e_{j, i},\,\beta^\e_{j, i}$ which satisfy (\ref{eq: Lung. lati}), i.e.  
 \begin{equation}
     \label{eq: Lung. lati_1}
       \begin{split}
           &P^\e_{j+1,i}=P^\e_{j,i}+2\e \alpha_{j,i}-\e(\beta_{j,i}+\beta_{j,i+1})\\           &D^\e_{j+1,i}=D^\e_{j,i}+\sqrt{2}\e\beta_{j,i}-\sqrt{2}\e(\alpha_{j,i}+\alpha_{j,i+1}).
       \end{split}
\end{equation}
For $i=1, \dots, 4$ we define $D_i^\e(t)$ as the linear interpolation in $[j\tau, (j+1)\tau]$ of the values $D^\e_{j,i}$ and $D^\e_{j+1,i}$. We define analogously $P_i^\e(t)$ as the linear interpolation in $[j\tau, (j+1)\tau]$ of the values $P^\e_{j,i}$ and $P^\e_{j+1,i}$. We observe that $P_i^\e(t)$ and $D_i^\e(t)$ are equibounded and uniformly Lipschitz (because of (\ref{eq: Lung. lati_1})) on all intervals $[0, T]$ such that $P_i^\e(t)\ge\bar{c}$ and $D_i^\e(t)\ge\bar{c}$. In particular they converge uniformly as $\e\to 0$ to Lipschitz functions $P_i(t)$ and $D_i(t).$ It follows that the octagons $A^\e(t)$ whose sides are $\PP^\e_i(t)$ and $\DD^\e_i(t)$ converge in the sense of Hausdorff (up to subsequences) as $\e\to 0$ to a limit octagon $A(t)$. 

To conclude, we observe that (\ref{eq:free_surfactant_sides_velocity}) follows by passing to the limit in \eqref{eq:free_dissipation_alpha} and \eqref{eq:free_dissipation_beta}, after dividing the sides displacements for the time step $\tau = \zeta\e$. For this part of the proof we refer to the end of the proof of \cite[Theorem 1]{BGN}.
\end{proof}

\section{The case \texorpdfstring{$\gamma <2$}{gamma<2}}\label{sec:gamma<2}
In this last section we study the minimizing movements in the case $\gamma<2$ and under the assumption that at the first step the surfactant is sufficient to surround the starting set $I_0^\e$. In the first part of this section we show two preliminary results.
In Lemma \ref{lemma: surfactant costante} we prove that if $\gamma<2$ then the amount of surfactant (i.e. $\#\Z^\e_j$) remains constant at each step $j$. As a consequence, in Lemma \ref{lemma:surrounded_shape_at_every_step}, we prove that the minimizer set $I_j$ is surrounded by surfactant at each time step (i.e. $\partial^+I_j\subset Z_j$).

\begin{lemma}
\label{lemma: surfactant costante}
    Let $\gamma<2$. Let $u_0, u_1\in\A_\e$ be such that $u_1$ is a minimizer of $\enF(\cdot, u_0)$. Set $C_\e=\#\Z_0^\e$. Then, there exists $\overline{\e}>0$ such that $\#\Z^\e_1=C_\e$ for every $\e\leq \overline{\e}$.
\end{lemma}
    \begin{proof}
        We omit the dependence on $\e$. By contradiction let us assume that $\#\Z_1\neq\#\Z_0$. Assume first that $\#\Z_1= \#\Z_{0}+ n$ for some $n\in\NN$, and let $p^1,\dots,p^n \in \e\ZZ^2$ such that $u_1(p^i)=0$ for every $i=1,\dots,n$. We define the competitor $\tilde{u}_0$ replacing $u_1(p^i)\mapsto -1$.
        It holds that
        $$\E_\e(\tilde{u}_1)-\E_\e(u_1)\leq 8\e n,$$
        $$\dis^{1}_\e(\tilde{u}_1,u_0)=\dis^{1}_\e(u_1,u_0),$$
        $$\frac{\e^\gamma}{\tau}(\dis^{0}(\tilde{u}_1,u_0)-\dis^{0}(u_1,u_0))=-\frac{\e^{\gamma}}{\tau}n.$$
        Therefore, since $\tau = \zeta\e$,
        \begin{equation*}
            \begin{split}            \mathcal{F}^{\tau,\gamma}_{\e}(\tilde{u}_1,u_0)-\mathcal{F}^{\tau,\gamma}_{\e}(u_1,u_0)&\leq 8\e n-\frac{\e^{\gamma}}{\tau}n\\
                &=\e n\left(8-\frac{\e^{\gamma-2}}{\zeta}\right)
            \end{split}
        \end{equation*}
        Since $\gamma<2$ we have that $\left(8-\frac{\e^{\gamma-2}}{\zeta}\right)<0$ and we get the contradiction. In the case  $\#\Z_1= \#\Z_0-n$, we consider $p^1,\dots,p^n \in \e \ZZ^2$ such that $u_1(p^i)=-1$. We define the competitor $\tilde{u}_1$ replacing $u_1(p^i)\mapsto 0$ for every $i=1, \dots, n$.
        We conclude that $\enF(\tilde{u}_1)<\enF(u_1)$ by repeating the same computation above.
         \end{proof}


We set $C_\e := \#\Z^\e_j$ for every $\e>0$. In view of the last Lemma \ref{lemma: surfactant costante} we have that $C_\e$ is well defined (i.e. it does not depend on $j$). From now on we suppose that $\e^2 C_\e\to 0$ as $\e\to 0$ and, as in Section \ref{sec:gamma>2}, we suppose that the starting set $I_0^\e = \{u^\e_0 = 1\}$ is a discretization of a non degenerate octagon. Furthermore, in order to ensure that the starting set $I^\e_0$ is surrounded by surfactant, we shall require that
\begin{equation}\label{eq:5.1}
	\e C_\e\ge \sum_{i=1}^4 P_{0,i}^{\e}+D_{0,i}^{\e}/\sqrt{2}.
\end{equation}

In the following of this section, we will show that at every time step the set $\I^\e_j$ is either a discrete octagon (Definition \ref{def:discrete_octagon}) or a $\e$-quasi-rectangle, in the sense of the following Definition \ref{def:quasi_rectangle}.

\begin{definition}\label{def:quasi_rectangle}Let $\I$ be a staircase set, and let $C_\e$ be as above. We say that $\I$ is an $\e$-\emph{quasi-rectangle} if there exist sets $\tilde{I},\,\PP_i:i=1, \dots, 4$ and $\Delta_i:i=1, \dots, 4$ satisfying the properties below such that $I$ can be written as
\begin{equation}\label{eq:quasi_rectangle}
	\I = \tilde{\I}\cup(\cup_{i=1}^4\PP_i)\cup(\cup_{i=1}^4\Delta_i).
\end{equation}
In (\ref{eq:quasi_rectangle}), the $\PP_i$ are the external slices of $I$, as in \eqref{eq:parallel_sides_and_slices}. The set $\tilde{\I}$ is an octagon contained into $I\setminus(\cup_{i=1}^4\PP_i)$ such that $\frac{\tilde{D}_i}{\sqrt{2}} = \max\{\e\lceil4\sqrt{C_\e}\rceil, \e\lceil\e^{1/8-1}\rceil\}$. If we denote by $T_i:i=1, \dots, 4$ the four sets such that $R_{\tilde{I}}\setminus\tilde{I} = \cup_{i=1}^4T_i$, then the $\Delta_i$ have the property that $\tilde{\I}\cup(\cup_i\Delta_i),$ is strongly connected, and $\Delta_i\subset T_i$, for every $i=1, \dots, 4$ (see Figure \ref{fig:quasirectangle}). This set is a ``quasi''-rectangle in the sense that 
\[|R_{I}\setminus I| \le \sum_{i=1}^4|T_i|+c\e Per(I)\le4\frac{16\e^{2}C_\e}{2}+c\e\to 0\;\;\text{ as }\;\;\e\to 0.\]
\begin{figure}[H]
	\centering
	\resizebox{0.60\textwidth}{!}{\input{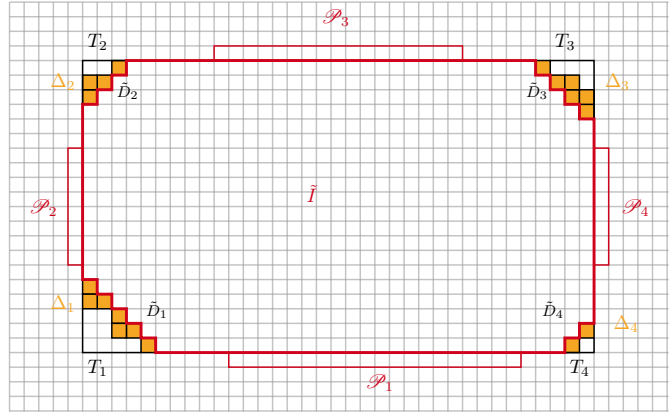}}
	\caption{An example of $\e$-quasi-rectangle; in orange we represent the set $\Delta_i\subset T_i$ for $i=1,\dots,4$.}
	\label{fig:quasirectangle}
\end{figure}
\end{definition}

\begin{remark}
    In the last Definition we required $\frac{\tilde{D}_i}{\sqrt{2}} \ge \e\lceil\e^{1/8-1}\rceil.$ The choice of the value $1/8-1 = -7/8$ is made to ensure that $\frac{\tilde{D}_i}{\sqrt{2}} \gg \e^{1/4}.$ Instead, requiring that $\frac{\tilde{D}_i}{\sqrt{2}} \ge \e\lceil 4\sqrt{C_\e}\rceil$ we ensured that $\#T_i\ge \frac{\lceil 4\sqrt{C_\e}\rceil(\lceil 4\sqrt{C_\e}\rceil+1)}{2}>8C_\e>C_\e$ for every $i$.
\end{remark}

If both $\I^\e_j$ and $\I^\e_{j+1}$ are octagons we adopt the notation of (\ref{eq: spostamenti}) and (\ref{eq: Lung. lati}) to describe the length of the sides of $\I^\e_{j+1}.$ If instead $\I^\e_{j+1} = \tilde{\I}^{\e}_{j+1}\cup_i\{\PP_{j+1,i}^{\e}\}\cup_i\{\Delta_i^\e\}$ is an $\e$-quasi-rectangle, and $\I_j^\e$ is either an octagon or an $\e$-quasi-rectangle with parallel sides $\PP_{j+1,i}^{\e}$, we set $a^\e_{j,i}$  to be the Euclidean distance between the (parallel) straight lines containing $\PP^{\e}_{j,i}$ and $\PP^{\e}_{j+1,i}$, and $\alpha^\e_{j, i}:=a^\e_{j,i}/\e$.

In the following Lemma we prove that the minimizing set $I^\e_j$ is surrounded by surfactant at every time step. We observe that the statement is true for every $\gamma>0$, and not only for $\gamma<2$. 
\begin{lemma}\label{lemma:surrounded_shape_at_every_step}
	Let $\gamma>0$. Let $u_j, u_{j+1}\in\A_\e$ be such that $u_j$ verifies assumption \eqref{H} of Subsection \ref{subsec:connectedness}, and $u_{j+1}$ is a minimizer of $\enF(\cdot, u_j)$. 
    Suppose moreover that $\partial^+I_j\subset Z_j$. Then for every $\e$ small enough it holds $\partial^+I_{j+1}\subset Z_{j+1}$ and $\# Z_{j+1}\le\#Z_j$. 
    
    Moreover, if there exists $p\in\zero_{j+1}\cap I_j$, then $\#\partial^+I_{j+1}<\#\partial^+I_j$.
\end{lemma}
\begin{proof}
From Proposition \ref{prop:connectedness} we have that for every $\e$ sufficiently small the set $I_{j+1}$ is a staircase set contained in $I_j$.\vspace{5pt}\\
\emph{First part of the proof.} We prove that $\partial^+I_{j+1}\subset Z_{j+1}$. We argue by contradiction assuming that $\partial^+I_{j+1}\not\subset Z_{j+1}$.\\

\emph{Step 1.} We prove that there exists an injective map $\varphi:\partial^+I_{j+1}\to\partial^+I_j$ (in particular $\#\partial^+I_{j+1}\le\#\partial^+I_j$). 
We denote by $H^{1, i},i=1, \dots, n_h$ and $V^{1, i},i=1, \dots, n_v$ the horizontal and vertical slices of $I_{j+1}.$

We introduce a partition of the boundary $\partial^+I_{j+1}$ in four disjoint subsets $B^\pm_i:\:i=1, 2$ as follows
\begin{equation}\label{eq:partition_partial}
    \begin{aligned}
        B^\pm_1&:=\{p\in\partial^+I_{j+1}:\:p\pm\e e_1\in\partial^-I_{j+1}\},\\
        B^\pm_2&:=\{p\in\partial^+I_{j+1}:\:p\pm\e e_2\in\partial^-I_{j+1}\}\setminus (B_1^+\cup B_1^-).
    \end{aligned}
\end{equation}
For $i=1, 2$ we define the map $t_i^\pm(p):\:\partial^+I_{j+1}\to\NN\cup\{0\}$ by
\begin{equation}\label{eq:t(p)} t_i^\pm(p):=\min\bigl\{t\in\NN\cup\{0\}:\:p\pm t\e e_i\in \partial^+I_j\bigl\}.
\end{equation}
Finally, we define the map $\varphi:\partial^+I_{j+1}\to\partial^+I_j$ as 
\begin{equation}\label{eq:phi}
    \varphi(p) = \begin{cases}
        p-t_2^-(p)\e e_2&\text{ if }p+\e e_2\in H^{1, 1}\cup H^{1, 2},\\
        p+t_2^+(p)\e e_2&\text{ if }p-\e e_2\in H^{1, n_h}\cup H^{1, n_h-1},\\
        p-t_1^-(p)\e e_1&\text{ if }p+\e e_1\in V^{1, 1}\cup V^{1, 2},\\
        p+t_1^+(p)\e e_1&\text{ if }p-\e e_1\in V^{1, n_v}\cup V^{1, n_v-1},\\
        p\mp t_i^\mp(p)\e e_i&\text{ if }p\in B^\pm_i,\text{ for }i=1, 2, \text{ and does not belong to any of the sets above.}
    \end{cases}
\end{equation}
We have to show that $\varphi$ is injective. We observe that for every pair $\{p, q\}\subset\partial^+I_{j+1}$ such that $\varphi(p)$ is defined in one of the first four cases of \eqref{eq:phi} it holds $\varphi(p) = \varphi(q)\implies p=q$. Same result holds in case $\varphi(p)$ is defined in one of the first four cases of \eqref{eq:phi}, and $\varphi(q)$ is defined in the fifth case. Therefore we have to show that for every pair $\{p, q\}$ such that $\varphi(p)$ and $\varphi(q)$ are defined in the fifth way of \eqref{eq:phi} it holds $\varphi(p) = \varphi(q)\implies p=q$. Let $T:=R_{I_{j+1}}\setminus I_{j+1} = \cup_{i=1}^4T_i$, as in Definition \ref{def:staircase_set}. Without loss of generality we can also assume that $\{p, q\}\subset \partial^+I_{j+1}\cap T_i$ for some $i\in\{1,\cdots, 4\}$. We complete the proof in case $\{p, q\}\subset \partial^+I_{j+1}\cap T_1$, and assuming by contradiction that $p\neq q$. Again without loss of generality we can assume that $p\in B^+_1$ and $q\in B^+_2$. Since $\varphi(p) = \varphi(q)$ it follows that $q_1 < p_1$, and since $q\not\in B^+_1$ we deduce that $q+\e e_1\in \partial^+I_1.$ We apply Lemma \ref{lemma:shape_optimality} to the two adjacent points $\{q+\e e_2,\,q+\e e_1+\e e_2\}$. Since by assumption $\partial^+I_{j+1}\not\subset Z_{j+1}$, we deduce that case (iii) of Lemma \ref{lemma:shape_optimality} cannot occur. Let $\bar{i}$ such that $\{q+\e e_2,\,q+\e e_1+\e e_2\}\subset H^{1, \bar{i}}.$ If case (i) of Lemma \ref{lemma:shape_optimality} occurs, then $\varphi(q) = q$. Since the two components of $\varphi(p)$ are $\varphi(p) = \Bigl(\bigl(\varphi(p)\bigl)_1, p_2\Bigl)$, and since we assumed that $\varphi(p) = \varphi(q)$, we deduce from the fact that $p\in B^+_1$ that $p = p^{1, \bar{i}-1}-\e e_1$. But then $\varphi(p) = p\neq q$, which is a contradiction. The only case left, in view of Lemma \ref{lemma:shape_optimality}, is that $\bar{i} = 2$, but in this situation we have $q+\e e_2\in H^{1, 2}$, therefore $\varphi(q)$ is defined as in the first case of \eqref{eq:phi}, which is a contradiction. The proof of the injectivity of $\varphi$ is complete.\\

\emph{Step 2.} We recall that we are assuming by contradiction that $\partial^+I_{j+1}\not\subset Z_{j+1}.$ 
Then from Lemma  \ref{lemma:surfactant_placement} we have $ Z_{j+1}\subsetneq \partial^+I_{j+1}$, and we deduce from the first step and from $\partial ^+I_j\subset Z_j$ that $\#Z_{j+1}<\#Z_j$. In case $0<\gamma<2$ this is a contradiction with Lemma \ref{lemma: surfactant costante}. For general $\gamma>0$ we argue similarly as in the proof of Lemma \ref{lemma:octagonal minimizer}. Since $ Z_{j+1}\subsetneq \partial^+I_{j+1}$ there exists $p\in\partial^+I_{j+1}\cap\{u_{j+1} = -1\}$. We may suppose for simplicity that $p\in H^{j+1, 1}-\e e_2$ (where $H^{j+1, 1}$ is the first horizontal slice of $I_{j+1}$), since the general case is similar. If there exists a cell $q\in (H^{j+1, 1}-\e e_2)\cap Z_{j+1}$ which is adjacent to $p$ (i.e. $q = p\pm\e e_1$) then we define the competitor $\tilde{u}_{j+1}$ by replacing $u_{j+1}(p)\mapsto 0$. It holds 
	\[\E_\e(\tilde{u}_{j+1})\le \E_\e(u_{j+1})+3\e(1-k)-2\e = \E_\e(u_{j+1})-\e(1-3k)<\E_\e(u_{j+1}).\]
	Moreover $\dis^1_\e(u_{j+1}, u_j) = \dis^1_\e(\tilde{u}_{j+1}, u_j)$ and, since $\#Z_{j+1}<\#Z_j$, it holds
	\[\dis^0_\e(\tilde{u}_{j+1}, u_j)<\dis^0_\e(u_{j+1}, u_j).\]
	It follows that $\enF(\tilde{u}_{j+1})<\enF(u_{j+1})$, which contradicts the minimality of $u_{j+1}$. We are left to consider the case in which for every $p\in (H^{j+1, 1}-\e e_2)\cap\{u_{j+1} = -1\}$ it holds $p\pm\e e_1\in \{u_{j+1} = -1\}$. This implies, in particular, that $(H^{j+1, 1}-\e e_2)\subset\{u_{j+1} = -1\}.$ Let $H^{j+1, 1}= \dseg{p^{j+1, 1}}{q^{j+1, 1}}$, and consider the competitor $\tilde{u}_{j+1}$ obtained by replacing $u_{j+1}(p^{j+1, 1})\mapsto 0$. We may estimate $\E_\e(\tilde{u}_{j+1})$ and $\dis^0_\e(\tilde{u}_{j+1}, u_j)$ as before, and, as a consequence of Proposition \ref{prop:hausdorff_distance_from_boundary}, we have that 
	\[\frac{1}{\tau}\dis^1_\e(\tilde{u}_{j+1}, u_j)\le \frac{1}{\tau} \dis^1_\e(u_{j+1}, u_j)+c\e^{1+\mu},\]
	for some parameter $\mu<1/4$ and for $\e$ sufficiently small. We conclude again that $\enF(\tilde{u}_{j+1})<\enF(u_{j+1})$, which contradicts the minimality of $u_{j+1}$.\vspace{5pt}\\
    \emph{Second part of the proof.} We have to show that $\#Z_{j+1}\le\#Z_j.$ We argue by contradiction assuming that $\# Z_{j+1}>\# Z_{j}$. Suppose first that $Z_{j+1} \subset \partial^+I_{j+1}.$ Then arguing as in the first step we deduce that the map $\varphi$ defined as in \eqref{eq:phi} is injective, and we have a contradiction since 
    \[\# Z_{j+1} \le \#\partial^+I_{j+1}\le \#\partial^+I_{j}\le \# Z_j <\# Z_{j+1}.\]
    If instead $\partial^+I_{j+1}\subsetneq Z_{j+1}$ then there exists a point $p\in(\partial^- Z_{j+1})\setminus(\partial^+I_{j+1})$. The new competitor $\tilde{u}_{j+1}$ obtained by replacing $u_{j+1}(p)\mapsto -1$ satisfies \[\E_\e(\tilde{u}_{j+1})<\E_\e(u_{j+1}),\;\dis^1_\e(\tilde{u}_{j+1}, u_j) = \dis^1_\e(u_{j+1}, u_j),\;\text{ and }\;\dis^0_\e(\tilde{u}_{j+1}, u_j) < \dis^0_\e(u_{j+1}, u_j),\]
    which is in contradiction with the minimality of $u_{j+1}$, and the claim follows. \vspace{5pt}\\
    \emph{Third part of the proof.} Arguing similarly as in the first step of the first part of the proof one can show using Lemma \ref{lemma:shape_optimality} that the map $\varphi$ defined in \eqref{eq:phi} is not surjective if there exists $p\in\zero_{j+1}\cap I_{j}$. We omit the details.
\end{proof}

\paragraph{First stage of the minimizing movements: pinning of the diagonals.}
\begin{remark}\label{rem:gamma<2_surf_arrangement}
In this remark we give an outline of the following of this section. 

The minimizing movements can be divided in two qualitatively different stages. In the first one the minimizer is an octagon at each step, having at least a diagonal side of length bigger than $2\sqrt{2}\e\sqrt{C_\e}$. During this stage the diagonal sides having positive length remain pinned, while the parallel sides move inwards without changing their orientation. We computed the sides displacements in Proposition \ref{prop:surrounded_shape_minimizer_beginning}. In this paragraph, before stating this Proposition, we introduce two Lemmas. 
In the first one \ref{lemma:5_optimal_placement_surfactant} we discuss the optimal placement of the surfactant set $Z_j$ around $I_j,$ and we introduce a peculiarity of the case $\gamma<2,$ i.e. that the surfactant placement is not uniquely determined, not even when $I_j$ is known (in particular it is no longer true that $Z_j = \partial^+I_j$ for every $j$). We deal with this problem as follows. First, we observe that only the cardinality of $Z_j$ (and not its geometrical shape) plays a role in the minimization problem $\argmin\{\enF(\cdot, u_j)\}$, in particular the set $I_{j+1}$ depends only on $I_j$ and on $\#Z_j = C_\e$. Second, since we require that $\e^2C_\e\to 0$ as $\e\to 0$, we have that the surfactant does not appear in the limit flow. Therefore, since our main concern is to describe the limit flow of the phase $1$ (which, as just mentioned, depends only on the sets $I_j$), in the construction of the minimizing movement $(u^\e_j)_j$, we will only describe the shape of $I_j$, and we give only an approximate description of $Z_j$. This description is the content of Lemma \ref{lemma:5_optimal_placement_surfactant}. 

The second preliminary result of this paragraph is the technical Lemma \ref{lemma:5_preliminaries}, which contains some preliminary remarks concerning the geometry of the set $I_{j+1},$ and the distance of its boundary from the boundary of $I_j$.

The second stage of the motion is completely different, and will be analyzed in the next paragraph. During this stage, the minimizer $I_j$ is an $\e$-quasi-rectangle at each time step, and in this case the sides displacements are given in Proposition \ref{prop:surrounded_shape_minimizer_final_part}. However, in this second stage a precise description of the motion is more difficult. Therefore, we studied the minimizing movements of a special class of minimizers (the ``good minimizers'', see Definition \ref{def:good_minimizer}).

In the final Theorem \ref{teo:final teo A_j surrounded} we computed the limit flow, which follows immediately from the results of Propositions \ref{prop:surrounded_shape_minimizer_beginning} and \ref{prop:surrounded_shape_minimizer_final_part}.
\end{remark} 

In the following Lemma we show a possible way to arrange the surfactant set of a minimizer $u_{j+1}$.
\begin{lemma}\label{lemma:5_optimal_placement_surfactant}
    Let $\gamma<2$. Let $u_j,\ u_{j+1} \in \mathcal{A}_\e$ be such that $u_j$ satisfies assumption \eqref{H} of Subsection \ref{subsec:connectedness}, and $u_{j+1}$ is a minimizer for $\enF(\cdot, u_j)$. Define $R_{\I_{j+1}},\,n^{j+1}_h,\,n^{j+1}_v$ as in Definition \ref{def:staircase_set}, where $n^{j+1}_h$ and $n^{j+1}_v$ denote the number of horizontal and vertical slices of $\I_{j+1}$. Then for $\e$ small enough the set $\I_{j+1}$ is a staircase set, and if \begin{equation}\label{eq:5.inclusion}
        \#\partial^+\I_{j+1}\le C_\e\le\#(R_{\I_{j+1}\cup\partial^+\I_{j+1}})-\#\I_{j+1}
    \end{equation}
    then we have $\partial^+\I_{j+1}\subset\Z_{j+1}\subset (R_{\I_{j+1}\cup\partial^+\I_{j+1}})\setminus\I_{j+1}$ and 
    \begin{equation}\label{eq:5.5}
        \E_\e(u_{j+1}) = 2\e(1-k)C_\e+2\e(1-k)(n_h^{j+1}+n^{j+1}_v+2).
    \end{equation}
    
    The set $\Z_{j+1}$ is not uniquely determined. One possible way to arrange its elements is following the steps below.
    \begin{itemize}
        \item[(i)] We set $R:=R_{\I_{j+1}\cup\partial^+\I_{j+1}}$;
        \item[(ii)]We set  $Z_{j+1}^0:=\partial^+\I_{j+1},$ $I_{j+1}^0:=\I_{j+1}\cup Z_{j+1}^0$ and, inductively, 
        \[Z_{j+1}^{i+1}:=\{p\in\partial^+I^{i}_{j+1}:\#\bigl(\mathcal{N}(p)\cap I^{i}_{j+1}\bigl) = 2\},\;\text{ and }\;I^{i+1}_{j+1}:=I^{i}_{j+1}\cup Z_{j+1}^{i+1}.\]
        We interrupt this process at the first index $\bar{i}$ such that $\sum_{i=1}^{\bar{i}}\# Z_{j+1}^{i} > C_\e.$
        \item[(iii)] We set 
        \[u_{j+1}(p) = 0\text{ for every }p\in\bigcup_{i=1}^{\bar{i}-1}Z_{j+1}^i.\] Then, if $n := C_\e-(\#I_{j+1}^{i-1}-\#\I_{1})>0$, we select randomly $p^1, \dots, p^n\in Z_{j+1}^{\bar{i}}$ and we set $u_{j+1}(p^i) = 0$ for every $i$ (see Figure \ref{fig:Placement_of_surfactant}).
    \end{itemize}
    This concludes the construction.
    
    Finally, if instead
    \begin{equation}\label{eq:5.inclusion_bis}
        C_\e>\#(R_{I_{j+1}\cup\partial^+ I_{j+1}})-\#I_{j+1},
    \end{equation}
    then we can arrange the set $Z_{j+1}$ so that $R_{I_{j+1}\cup\partial^+I_{j+1}}\setminus\I_{j+1}\subset\Z_{j+1}.$
    \begin{figure}[H]
	\centering
	\resizebox{0.80\textwidth}{!}{\input{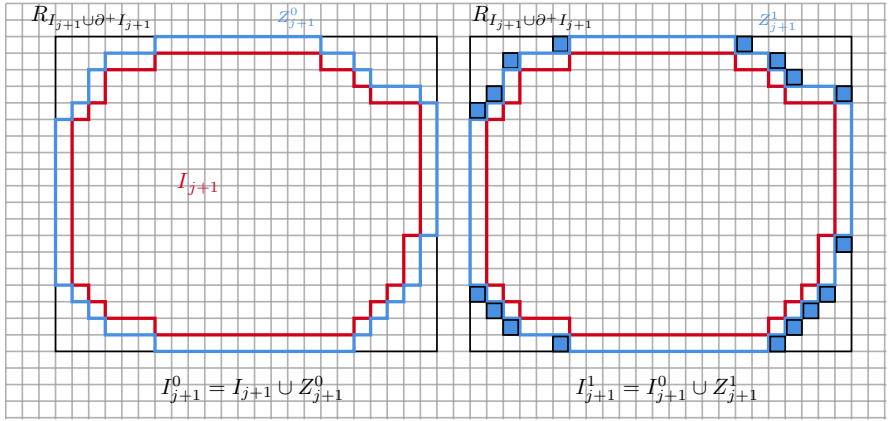}}
	\caption{A possible way to arrange the elements of $Z_{j+1}$ in the case $\gamma<2$. On the left side, the set $I^{0}_{j+1}=I_{j+1}\cup Z^{0}_{j+1}$; on the right the set $I^{1}_{j+1}$ where the colored blue squares are elements of $Z^{1}_{j+1}$ according to the definition given in (ii) of Lemma \ref{lemma:5_optimal_placement_surfactant}.}
	\label{fig:Placement_of_surfactant}
\end{figure}
    
\end{lemma}
\begin{proof}
    From Propositions \ref{prop:horizontal_convexity} and \ref{prop:connectedness}, we know that $\I_{j+1}$ is a staircase set contained into $\I_j$. We observe that $\partial^+\I_{j+1}\subset Z_{j+1}$ because of Lemma \ref{lemma:surfactant_placement}. Therefore, in view of Lemma \ref{lemma:energy_surfactant}, it holds 
    \[\E_\e(u_{j+1}) \ge 2\e(1-k)C_\e+\frac{1}{2}(1-k)[2\e( n_h^{j+1}+n^{j+1}_v)+2\e\bigl((n^{j+1}_h+2)+(n_v^{j+1}+2)\bigl)],\]
    and that the inequality is strict if $Z_{j+1}\not\subset(R_{I_{j+1}\cup\partial^+I_{j+1}})\setminus I_{j+1}.$
    To conclude we observe that for the function $u_{j+1}$ built following the steps (i)-(ii)-(iii) equality holds in (\ref{eq:5.5}).\vspace{5pt}\\
    Now we prove the last statement. Let $u_{j+1}$ be a minimizer such that \begin{equation}\label{eq:5.60}
    	R_{I_{j+1}\cup\partial^+I_{j+1}}\setminus\I_{j+1}\not\subset\Z_{j+1}
    \end{equation}
    (otherwise there is nothing to prove). Then, there exists $p\in R_{I_{j+1}\cup\partial^+I_{j+1}}\setminus\I_{j+1}$ such that $u_{j+1}(p)=-1$ and $\#(\nn(p)\cap (I_{j+1}\cup Z_{j+1}))\ge 2$. Moreover, in view of \eqref{eq:5.inclusion_bis}, there exists $q\in Z_{j+1}\setminus (R_{I_{j+1}\cup\partial^+I_{j+1}})$ such that $\#\nn(q)\cap \{u_{j+1} = -1\}\ge2$. We define a competitor $\tilde{u}_{j+1}$ by replacing $u_{j+1}(p)\mapsto 0$ and $u_{j+1}(q)\mapsto -1$. Since $\E_\e(u_{j+1})\ge \E_\e(\tilde{u}_{j+1})$, and since the dissipation of $u_{j+1}$ is equal to the dissipation of $\tilde{u}_{j+1}$ we deduce that $\tilde{u}_{j+1}$ is an other minimizer. Now if $\tilde{u}_{j+1}$ verifies the claim the proof is complete. Otherwise we can repeat inductively the same construction, until \eqref{eq:5.60} is no longer verified.
\end{proof}

Before stating the main result of this paragraph, we introduce a technical Lemma.

\begin{lemma}\label{lemma:5_preliminaries}
    Let $\gamma<2$. Let $u_j,\ u_{j+1} \in \mathcal{A}_\e$ be such that $u_j$ verifies assumption \eqref{H} of Subsection \ref{subsec:connectedness} and $u_{j+1}$ is a minimizer for $\enF(\cdot, u_j)$. Suppose that $\#\partial^+I_{j}\le\# Z_j$. 
    Let $p\in\zero_{j+1}$ (see Definition \ref{def: corner unit}). Then at least one of the following holds.
    \begin{itemize}
        \item[(i)] $p\in\partial^+\I_j$ (in particular $p\in\zero_j$);
        \item[(ii)] the set $\I_{j+1}\cup\Z_{j+1}$ is a rectangle;
    \end{itemize}
    Moreover it holds that    \begin{equation}\label{eq:5.7}
    d^\e_1(q, \partial\I_j)\ge d^\e_1(p, \partial\I_j)\;\;\text{ for all }\;q\in\one_{j+1},\;p\in\zero_{j+1}\setminus\mathcal{N}(q).
    \end{equation}
    \end{lemma}
\begin{proof}
\emph{First part.} In view of Lemma \ref{lemma:surrounded_shape_at_every_step} we have $\partial^+\I_{j+1}\subset\Z_{j+1}.$
We observe that since $I_{j+1}\subset I_j$ then $p$ always belongs to $I_j\cup\partial^+I_j$. If $p\in\partial^+ I_j$ then case (i) occurs. Therefore we may now suppose that $p\in\zero_{j+1}\cap\I_j$, and we prove (ii). Without loss of generality, we can assume that $u_{j+1}(p-\e e_i)=1$ for $i=1, 2.$ Therefore we have that $u_{j+1}(p-\e e_2+\e e_1)$ and $u_{j+1}(p-\e e_1+\e e_2)$ are either equal to zero or to one. In particular, with the notation of Lemma \ref{lemma:5_optimal_placement_surfactant}, we have that $p+\e e_2\in Z_{j+1}^0 = \partial^+I_{j+1}$ or $p+\e e_2\in Z_{j+1}^1$, and the same holds for $p+\e e_1.$ 
If $p+\e e_2\in \partial^+I_{j+1}$ then $u_{j+1}(p+\e e_2) = 0$. Suppose instead $p+\e e_2\in Z_{j+1}^1$. Then, since $p\in\zero_{j+1}\cap I_j$, by Lemma \ref{lemma:surrounded_shape_at_every_step} we have $\#\partial^+\I_{j+1}<\#\partial^+I_j\le C_\e$. From Lemma \ref{lemma:5_optimal_placement_surfactant} we deduce that either $u_{j+1}(p+\e e_2) = 0$, or we can rearrange the set $\Z_{j+1}$ so that $u_{j+1}(p+\e e_2) = 0.$ If also $u_{j+1}(p+\e e_1)=0$, then $p$ is a surfactant in the corner (Definition \ref{def: corner unit}), and this implies by Lemma \ref{lemma:shape_optimality}-1.1) that $\I_{j+1}\cup\Z_{j+1}$ is a rectangle. If instead $u_{j+1}(p+\e e_1)=-1$ then we can replace $u_{j+1}(p+\e e_1)\mapsto 0$ and $u_{j+1}(p)\mapsto 1$. This transformation strictly decreases the dissipation and does not increase the energy, which contradicts the minimality of $u_{j+1}$.\vspace{5pt}\\
\emph{Second part.} Let $p, q$ be as in (\ref{eq:5.7}). We argue by contradiction supposing $d^\e_1(p, \partial\I_j)>d^\e_1(q, \partial\I_j)\ge\e$, therefore in particular $p\in\zero_{j+1}\cap\I_j$. From the first part we deduce that $\I_{j+1}\cup\Z_{j+1}$ is a rectangle, and in particular the transformation $u_{j+1}(q)\mapsto 0$ and $u_{j+1}(p)\mapsto 1$ strictly reduces the dissipation without affecting the energy, which contradicts the minimality of $u_{j+1}$.
\end{proof}

The following Proposition shows that $I_{j+1}$ is an octagon until $I_j$ is an octagon having sufficiently long diagonal sides. In order to better understand property (i) and (ii) of the Proposition, we observe that if $\DD_{j, i}$ is the $i$-th diagonal side of an octagon $I_j$, it holds $\#\DD_{j, i} = 1$ if and only if $D_{j, i} = 0$.

\begin{proposition}\label{prop:surrounded_shape_minimizer_beginning}
    Let $\gamma<2$ and $\mu<1/4$. Let $u_j,\ u_{j+1} \in \mathcal{A}_\e$ such that $u_{j+1}$ is a minimizer for $\enF(\cdot, u_j)$. Suppose that $\I_j$ is an octagon satisfying assumption \eqref{H} of Subsection \ref{subsec:connectedness}. Suppose that there exists $i \in\{1, \dots, 4\}$ such that $D_{j,i}/\sqrt{2}\ge \max\{2\e\sqrt{C_\e}, \e^\mu\}$. Then there exists $\bar{\e}$ such that for every $\e\le\bar{\e}$ we have that the set $\I_{j+1}$ is a (possibly degenerate) octagon with the following properties:
    \begin{itemize}
        \item[(i)] $\zero_{j+1}\subset\zero_j$, in particular $\DD_{j+1, i}\subset\DD_{j, i}$ for every $i$ such that $D_{j+1, i}>0$.
        \item[(ii)] $\#\DD_{j+1, i}=1$ for every $i$ such that $D_{j, i}=0$.
        \item[(iii)] There exists a constant $c$ depending only on $\bar{c}$ (the constant of \eqref{H}) and $\mu$ such that the displacements of the parallel sides $\alpha_{j, i}$ defined as in \eqref{eq: spostamenti} satisfy \begin{equation}\label{eq:movement_parallel_sides_surrounding_surfactant}
            \alpha_{j, i} \begin{cases}
                =\Bigl\lfloor\frac{2\zeta(1-k)}{P_{j,i}}\Bigl\rfloor&\text{ if }\text{ dist}\bigl(\frac{2\zeta(1-k)}{P_{j,i}}, \NN\bigl)\ge c\e\\
                \in\Bigl\{\Bigl[\frac{2\zeta(1-k)}{P_{j, i}}\Bigl], \Bigl[\frac{2\zeta(1-k)}{P_{j,i}}\Bigl] -1\Bigl\}&\text{ otherwise.}
            \end{cases}
        \end{equation}
    \end{itemize}
\end{proposition}
\begin{proof}
We already know from Propositions \ref{prop:horizontal_convexity} and \ref{prop:connectedness} that $\I_{j+1}$ is a connected staircase set. We set $\delta_\e := \max\{2\e\sqrt{C_\e}, \e^\mu\}$ and $\mu':=(\mu+1/4)/2$, so that $\delta_\e\gg\e^{\mu'}$ as $\e\to 0$.
\vspace{5pt}\\
\emph{First step}. From Lemma \ref{lemma:surrounded_shape_at_every_step} we have that $\partial^+\I_{j+1}\subset \Z_{j+1}$.
Then from Lemma \ref{lemma:shape_optimality}, Remark \ref{rem:geometry_octagon}, and the fact that $\I_j$ is an octagon, it follows that either $\I_{j+1}$ is an octagon or $\I_{j+1}\cup\Z_{j+1}$ is a rectangle. We prove that $\I_{j+1}\cup Z_{j+1}$ cannot be a rectangle, and therefore, as a consequence, that $I_{j+1}$ is an octagon. Suppose, without loss of generality, that $D_{j,1}/\sqrt{2}\ge \delta_\e$.  Let $\PP_{j, 1} = \dseg{p^{j, 1}}{q^{j, 1}}$ and $\PP_{j, 2} = \dseg{p^{j, 2}}{q^{j, 2}}$ be as in Definition \ref{def:discrete_octagon}. For $\e$ sufficiently small, from Proposition \ref{prop:hausdorff_distance_from_boundary} there exist $\bar{p}^1,\bar{p}^2\in\partial^+\I_{j+1}$  such that $d^\e_1(p^{j, 1}, \bar{p}^1) \le \e^{\mu'}$ and $d^\e_1(q^{j, 2}, \bar{p}^2) \le \e^{\mu'}$. We fix a positive constant $c_1\in(\sqrt{3}/2, 1)$. For $\e$ sufficiently small it holds
    \[\overline{p}^1_1-\overline{p}^2_1\ge p^{j, 1}_1-p^{j, 2}_1-c\e^{\mu'} = \frac{D_{j, 1}}{\sqrt{2}}-c\e^{\mu'}\ge c_1\delta_\e,\text{ and similarly }\overline{p}^2_2-\overline{p}^1_2\ge c_1\delta_\e.\]Set $R_{\I_{j+1}}$ as in Definition \ref{def:staircase_set}, and $T^{j+1}:=R_{\I_{j+1}}\setminus\I_{j+1} = \cup_{i=1}^4T^{j+1}_i$, where the $T^{j+1}_i$ are the connected components of $T^{j+1}.$ 
    We observe that
    \begin{align*}
    \# T^{j+1}_1
    \ge\frac{1}{2}\frac{|\overline{p}^1_1-\overline{p}^2_1|}{\e}\cdot&\frac{|\overline{p}_2^1-\overline{p}^2_2|}{\e}\\
    &-\#\Bigl\{p\in\I_{j+1}:p_1\ge \overline{p}^2_1,\,p_2\ge \overline{p}^1_2 \;\text{ and }\;d^\e_1(p, \partial\I_j)\le\e^{\mu'}\Bigl\}
    \end{align*}
    and that
    \[\#\Bigl\{p\in\I_{j+1}:p_1\ge \overline{p}^2_1,\,p_2\ge \overline{p}^1_2 \;\text{ and }\;d^\e_1(p, \partial\I_j)\le\e^{\mu'}\Bigl\}\le c\frac{\e^{\mu'}}{\e}\cdot \frac{|\overline{p}^1-\overline{p}^2|}{\e}\le c\e^{\mu'-1}\frac{\delta_\e}{\e}.\]
    Recalling that $\delta_\e\gg\e^{\mu'}$ and that $c_1^2>3/4$, we finally deduce that
    \begin{align*}
        \# T^{j+1}_1
    \ge\frac{1}{2}\frac{|\overline{p}^1_1-\overline{p}^2_1|}{\e}\cdot\frac{|\overline{p}^1_2-\overline{p}^2_2|}{\e}-
    c\e^{\mu'-1}\frac{\delta_\e}{\e}\ge \frac{c_1^2\delta_\e^2}{2\e^2} - c\e^{\mu'-1}\frac{\delta_\e}{\e}>\frac{3}{4}\frac{\delta_\e^2}{2\e^2}\ge\frac{3}{2}C_\e>C_\e
    \end{align*}
    for $\e$ sufficiently small.
    Therefore by Lemma \ref{lemma:5_optimal_placement_surfactant} we conclude that $\I_{j+1}\cup Z_{j+1}$ cannot be a rectangle.\vspace{5pt}\\    
    \emph{Second step}. The fact that $\zero_{j+1}\subset\zero_j$ is a direct consequence of Lemma \ref{lemma:5_preliminaries} since we showed in the first step of this proof that $I_{j+1}\cup Z_{j+1}$ is not a rectangle. Statement (ii) is a direct consequence of (i).
    Finally we deduce system \eqref{eq:movement_parallel_sides_surrounding_surfactant} arguing as in the proof of Theorem \ref{teo:free_surfactant_movement}.
\end{proof}

\begin{remark}\label{rem:energy_is_perimeter}
    We observe that if $u_{j+1}$ is as in Proposition \ref{prop:surrounded_shape_minimizer_beginning}, then from \eqref{eq:5.5} of Lemma \ref{lemma:5_optimal_placement_surfactant} we have that 
    \[\E_\e(u_{j+1}) = 2(1-k)Per(A_{j+1})+const(\e),\]
    where $const(\e)$ is a term which depends only on $\e$.
    Since for $\e$ fixed this quantity is constant, it follows that, as long as the diagonal sides of $u_j$ are sufficiently long, i.e. as long as all the surfactant can be contained in $R_{I_{j+1}\cup\partial^+I_{j+1}}\setminus I_{j+1}$, then $u_{j+1}$ is also a minimizer for
 $$\tilde{\mathcal{F}}^{\gamma, \tau}_\e(u_{j+1}) := 2(1-k)Per(A_{\I_{j+1}})+\frac{1}{\tau}\dis^1_\e(u_{j+1}, u_j),$$ which is the same functional considered in \cite{BGN} (up to the multiplicative factor $2(1-k)$ in front of the perimeter energy).
\end{remark}

\paragraph{Second stage of the minimizing movements: the minimizer is an $\e$-quasi-rectangle.}
\begin{remark}\label{rem:5.surrounded_shape_minimizer_beginning}
    The proof of last Proposition shows that as long as the set $\I_{j+1}\cup\Z_{j+1}$ is not a rectangle, then $\I_{j+1}$ is an octagon, and coefficients $\alpha_{j, i}$ satisfy (\ref{eq:movement_parallel_sides_surrounding_surfactant}).
    
    Moreover, again from Proposition \ref{prop:surrounded_shape_minimizer_beginning}, we know that if the motion does not interrupt, at a certain point the set $\I_j$ will become an octagon whose diagonal sides verify $D_{j, i}/\sqrt{2}\le\max\{4\e\sqrt{C_\e}, \e^\mu\}.$ Setting $\mu = 1/8$ in particular we deduce that, if the motion does not interrupt, then after a finite number of steps the set $\I_{j+1}$ becomes an $\e$-quasi-rectangle.
    At this point a precise description of the set $\I_{j+1}$ is difficult. In fact in this case it can happen that the surfactant completely fills the set $T^{j+1}:=(R_{\I_{j+1}\cup\partial^+\I_{j+1}})\setminus\I_{j+1}$. As a result we cannot argue as in the first step of the proof of Proposition \ref{prop:surrounded_shape_minimizer_beginning} to deduce the precise shape of the set $\I_{j+1}.$
\end{remark}

In the following of this section, we study the minimizing movements of a special class of minimizers, and we show that with this choice the motion of the $\e$-quasi-rectangle mentioned in Remark \ref{rem:5.surrounded_shape_minimizer_beginning} can be approximated ``from the inside'' by the motion of a proper inscribed octagon, and ``from the outside'' by the motion of the circumscribed rectangle. 

\begin{definition}\label{def:good_minimizer}
    Let $u_j$ be an $\e$-quasi-rectangle, and let $u_{j+1}$ be a minimizer of $u\mapsto \enF(u, u_j)$.
    We say that $u_{j+1}$ is a \emph{good minimizer} if $I_{j+1} = \tilde{I}_{j+1}\cup(\cup_{i=1}^4\PP_{j+1, i})\cup (\cup_{i=1}^4 \Delta_{j+1, i})$ is an $\e$-quasi-rectangle according to Definition \ref{def:quasi_rectangle} such that for every $p\in\zero_{j+1}\cap R_{\tilde{I}_{j+1}}$ it holds
    \begin{equation}\label{eq:good_minimizer}
        d^\e_1(p, \partial\I_j)\le\min\{\alpha_{j,i},\,i=1,\,\dots,\,4\}.
    \end{equation}
    If $I_j$ is an octagon which is not an $\e$-quasi-rectangle (i.e. if at least one of its diagonal sides is longer than $\max\{\sqrt{2}\e\lceil4\sqrt{C_\e}\rceil, \sqrt{2}\e\lceil\e^{1/8-1}\rceil\}$), as a convention we say that any minimizer $u_{j+1}$ is a good minimizer.
\end{definition}

In the following Lemma we show that if $\gamma<2$ and if $I_j$ is an $\e$-quasi-rectangle surrounded by surfactant, then also $I_{j+1}$ is an $\e$-quasi-rectangle. Statement (i) gives an upper bound on the sides displacements which depends only on the size of the quasi-rectangle $I_j$. In (ii) instead, we show that as long as the $I_j$ is sufficiently large then the length of the parallel sides of the $\e$-quasi-rectangle are bounded from below. We point out that (ii) together with Proposition \ref{prop:connectedness} ensures that as long as the $\e$-quasi-rectangle $I_j^\e$ is large enough then, for sufficiently small values of $\e$, the set $I_{j+1}^\e$ is connected. Finally, (iii) shows the existence of a good minimizer.
\begin{lemma}\label{lemma:5.10}
    Let $\gamma<2$. Let $u_j\in\A_\e$  be such that $I_j$ satisfies assumption \eqref{H} of Subsection \ref{subsec:connectedness} . Suppose that $I_j = \tilde{I}_j\cup(\cup_{i=1}^4\PP_{j, i})\cup(\cup_{i=1}^4\Delta_{j, i})$ is an $\e$-quasi-rectangle, and denote by $\ell_j$ the length of the basis of the rectangle $R_{\tilde{I}_j}$. Let $u_{j+1}\in\A_\e$ be a minimizer of $\enF(\cdot, u_j)$. Then there exists $\bar{\e}$ such that for every $\e\le\bar{\e}$ the following properties hold true for the set $\I_{j+1}$.
    
    \begin{itemize}
        \item[(i)] There exists a constant $c_1$ depending only on $\ell_j$ such that $\max\{\alpha_{j,i}:\:i=1, \dots, 4\}\le c_1$.
        \item[(ii)] There is a constant $c_2$ depending only on $\ell_j$ such that $P_{j+1,i}\ge c_2$.
        \item[(iii)] There exists a minimizer $u_{j+1}$
        which is a good minimizer.
    \end{itemize}
\end{lemma}
\begin{proof}
    We denote by $\delta_\e:=\max\{\e\lceil4\sqrt{C_\e}\rceil, \e\lceil\e^{1/8-1}\rceil\}.$\\
    \emph{First claim.} We prove the claim for the displacement of the lowest parallel side, that is for $\alpha_{j, i} = \alpha_{j,1}$. For $s=j,j+1$ we denote by $H^{s, i} = \dseg{p^{s,i}}{q^{s, i}}$ for $i=1, \dots, n_h^s$ the horizontal slices of the sets $\I_{s}$. Fix $\mu\in(1/8,1/4)$. In view of Propositions \ref{prop:hausdorff_distance_from_boundary} and \ref{prop:connectedness}, there exists $\bar{\e}$ such that $\I_{j+1}$ is a connected staircase set contained into $\I_j$ with $d_{\hh}(\partial A_{j+1}, A_j)\le\e^\mu$ for every $\e\le\bar{\e}$. Observe moreover that  since $I_j$ is an $\e$-quasi-rectangle it holds $\e\#H^{j, 2}\ge \ell_j-2\delta_\e.$ First of all, we now show that $\e\#H^{j+1, 3}\ge \ell_j-4\delta_\e$. To this end, it is sufficient to prove that \[p^{j+1, 3}_1\le p^{ j, 2}_1+2\delta_\e,\;\;\text{and}\;\; q^{j+1, 3}_1\ge q^{j, 2}_1-2\delta_\e.\]
    We argue by contradiction. Suppose without loss of generality that the first inequality does not hold. 
    We set
    \[\bar{t}:=\min\{t>3:\:p^{j+1, t}_1<p^{j+1, 3}_1\}.\]
    We observe that such $\bar{t}\in\NN$ is well defined because $d_{\mathcal{H}}(\partial A_j, \partial A_{j+1})\le\e^{\mu}.$
From the geometry of $I_j$ we deduce that 
    \[d^\e_1(p^{ j+1, \bar{t}-1}-\e e_1, \partial\I_j)\ge d^\e_1(p^{ j+1, 3}-\e e_1, \partial\I_j)\ge d^\e_1(p^{j+1, 1},\partial \I_j)+\e.\]
    Since $p^{j+1, \bar{t}-1}-\e e_1\in\zero_{j+1}$ and $p^{j+1, 1}\in\one_{j+1}$ this is a contradiction in view of \eqref{eq:5.7} of Lemma \ref{lemma:5_preliminaries}.
    Summing up, we have that $|\e \#H^{j+1, 3} - \e\# H^{j, 2}|\le c\delta_\e$, i.e. $H^{j+1, 3}$ and $H^{j, 2}$ are parallel slices having approximately the same length. We only sketch the final part of the proof of (i), which follows a similar strategy as the one of the proof of Lemma \ref{lemma:distance_horizontal_sides}. We suppose for simplicity that the first three horizontal slices $H^{j+1, 1}, H^{j+1, 2}, H^{j+1, 3}$ of $I_{j+1}$ satisfy 
    \[p^{j+1, 3}<p^{j+1, 2}<p^{j+1, 1},\quad\text{ and }\quad q^{j+1, 1}<q^{j+1, 2}<q^{j+1, 3}.\]
    Consider the set 
    \begin{equation}\label{eq:new_S}
    	S:=\Bigl((H^{j+1, 3}-\e e_2)\cup (H^{j+1, 2}-\e e_2)\cup (H^{j+1, 1}-\e e_2)\Bigl)\setminus I_{j+1}.
    \end{equation}
    Without loss of generality we can suppose that $S\subset I_j$. Following the same strategy of the second step of Lemma \ref{lemma:distance_horizontal_sides} we can construct a competitor $\tilde{u}_{j+1}$ with the following properties:
    \begin{equation*}
    	\begin{aligned}
    		\{\tilde{u}_{j+1}=1\} = I_{j+1}\cup S,\quad \E_\e(\tilde{u}_{j+1})\le \E_\e(u_{j+1})+2\e(1-k)\;\;\text{ and }\;\;\# \{\tilde{u}_{j+1}=0\} = C_\e.\\
    	\end{aligned}
    \end{equation*}
    By construction for every $p\in S$ it holds 
    \[d^\e_1(p, \partial I_j)\ge \e(\alpha_{j, 1}-1),\]
    and by the first part of the proof we have $\e\#S\ge \ell_j-4\delta_\e.$ Therefore we conclude 
    \[\frac{1}{\tau}\dis^1_\e(\tilde{u}_{j+1}, u_j)\le \frac{1}{\tau}\dis^1_\e(u_{j+1}, u_j)-\frac{(\e\#S)\cdot(\alpha_{j, 1}-1)\e}{\zeta}\le \frac{1}{\tau}\dis^1_\e(u_{j+1}, u_j) - \frac{(\ell_j-4\delta_\e)\cdot(\alpha_{j, 1}-1)\e}{\zeta}.\]
    By minimality of $u_{j+1}$ we deduce that 
    \[\frac{(\ell_j-4\delta_\e)(\alpha_{j, 1}-1)\e}{\zeta}\le2\e(1-k)\implies \alpha_{j, 1}\le \frac{2(1-k)\zeta}{\ell_j-4\delta_\e}+1<\frac{4(1-k)\zeta}{\ell_j}+1,\]
    where in last inequality we used that $\delta_\e\to 0$ as $\e\to 0$. Claim (i) follows.\vspace{5pt}\\
    \emph{Second claim.} From the first part we know that for every $p\in H^{j+1, 1}$ it holds $d^\e_1(p, \partial\I_j)\le(\alpha_{j, 1}+1)\e\le (c_1+1)\e.$ We set $L:=\#H^{j+1, 1}$, then $L\e=P_{j+1,1}$ . Arguing as in the second step of the proof of Proposition \ref{prop:connectedness}, we can construct a competitor $\tilde{u}_{j+1}$ such that $\#\{\tilde{u}_{j+1}=0\} = C_\e$, and 
    \[\{\tilde{u}_{j+1}=1\} = I_{j+1}\setminus H^{j+1, 1},\;\;H^{j+1, 1}\subset\{\tilde{u}_{j+1}=0\},\;\text{ and }\;\E_\e(\tilde{u}_{j+1})\le \E_\e(u_{j+1})-\e(1-k).\]
    By minimality of $u_{j+1}$ we deduce that 
    \[0\ge \E_\e(u_{j+1})+\frac{1}{\tau}\dis^1(u_{j+1}, u_j)-\E_\e(\tilde{u}_{j+1})-\frac{1}{\tau}\dis^1(\tilde{u}_{j+1}, u_j)\ge \e(1-k)-\frac{\e L\cdot (c_1+1)\e}{\zeta},\]
    and therefore
    \[L\e\ge\frac
    {\zeta(1-k)}{c_1+1},\]
    which completes the proof of (ii).
    \vspace{5pt}\\
    \emph{Third claim.} We write $\I_{j+1}$ as
    $\I_{j+1} = \hat{\I}_{j+1}\cup(\cup_{i=1}^4\PP_{j+1,i}),$
    where the sides $\PP_{j+1,i}$ are as in Definition \ref{def:staircase_set}. From Lemma \ref{lemma:5_preliminaries} we have that 
    for every $p\in R_{\hat{\I}_{j+1}}\cap\zero_{j+1}$ it holds \begin{equation}\label{eq:6.qualcosa}
      d^\e_1(p, \partial \I_{j+1})\le (\min\{\alpha_{j,i}:\:i=1, \dots, 4\}+1)\e.      \end{equation} The geometry of $I_j$ ensures that \eqref{eq:6.qualcosa} holds also for every $p\in\partial^+\I_{j+1}$.
    We now want to construct a good minimizer. First we prove that there exists $\tilde{u}_{j+1}$ satisfying \eqref{eq:good_minimizer}, and later we show that the set $\{\tilde{u}_{j+1}=1\}$ is an $\e$-quasi-rectangle.
    Without loss of generality we may suppose that $\alpha_{j, 1}=\min\{\alpha_{j,i},\,i=1, \dots, 4\}.$ We rename $\hat{I}_{j+1, 0} := \hat{I}_{j+1}$, and $\zero_{j+1, 0}:=\zero_{j+1}$, and we consider the set \begin{equation}\label{eq:6.S}
        S_0:=\bigl\{p\in R_{\hat{\I}_{j+1}}\cap\I_{j}\cap\zero_{j+1}:\:d^\e_1(p, \partial\I_j) = \alpha_{j,1}+1\bigl\}.
    \end{equation}
     We define inductively \begin{align*}\hat{\I}_{j+1,i+1}&:=\hat{\I}_{j+1, i}\cup S_i,\\
        \zero_{j+1, i+1} &:= \{p\in R_{\hat{\I}_{j+1}}\cap\partial^+\hat{\I}_{j+1, i+1}:\:\#\mathcal{N}(p)\cap\partial^-\hat{\I}_{j+1, i+1} = 2\},\\        S_{i+1}&:=\bigl\{p\in R_{\hat{\I}_{j+1}}\cap\I_{j}\cap\zero_{j+1, i+1}:\:d^\e_1(p, \partial\I_j) = \alpha_{j, 1}+1\bigl\}.
    \end{align*}  We interrupt the construction at the first index $\bar{s}$ such that $S_{\bar{s}+1} = S_{\bar{s}}$.
    We have \[\#(\cup_{i=1}^{\bar{s}}S_{i})\le \# (R_{\hat{\I}_{j+1}}\setminus \I_{j+1})\cap I_j\le c\bigl(\max\{\alpha_{j, i}:i=1, \dots, 4\}+1\bigr)\frac{\delta_\e}{\e}\le c\cdot c_1\frac{\delta_\e}{\e}\]
    for some positive constant $c$, where $c_1$ is the constant of claim (i).
    From the second step we have that $\#H^{j+1, 1}\ge c_2\e^{-1}$, hence since $\delta_\e\to 0$ as $\e\to 0$ we deduce that for $\e$ sufficiently small \[\#H^{j+1, 1}\ge \# (R_{\hat{\I}_{j+1}}\setminus \I_{j+1})\cap I_j\ge \#(\cup _{i=1}^{\bar{s}}S_{\bar{s}}).\]
    Consider any point $p^0\in S_0$. Since $d^\e_1(p, \partial I_j)\ge d^\e_1(p^{j+1, 1}, \partial I_j)$ the competitor $\tilde{u}_{j+1, 1}$ obtained by replacing $u_{j+1}(p)\mapsto 1$ and $u_{j+1}(p^{j+1, 1})\mapsto 0$ is still a minimizer of $\enF(\cdot, u_j)$. Let 
    $n:=\sum_{i=0}^{\bar{s}}\# S_{i}$. Arguing inductively, we can build a sequence of minimizers $\tilde{u}_{j+1, i}:\:i=1, \dots, n$ obtained by replacing $\tilde{u}_{j+1, i-1}(p^{j+1, 1}+(i-1)\e e_1)\mapsto 0$, and by replacing a suitable element $p^i\in\cup_{i=1}^{\bar{s}}S_s$ with $\tilde{u}_{j+1, i-1}(p^{i})\mapsto 1$. By construction the function $\tilde{u}_{j+1, n}$ verifies \eqref{eq:good_minimizer}. \\
    Now we rename for simplicity $\tilde{u}_{j+1}:=\tilde{u}_{j+1, n}$, and we are left to prove that $\{\tilde{u}_{j+1}=1\}$ is an $\e$-quasi-rectangle. Since $I_j$ is an $\e$-quasi-rectangle, we may write $I_j = \tilde{I}_j\cup (\cup_{i=1}^4\PP_{j, i}) \cup (\cup_{i=1}^4\Delta_{j, i})$ as in Definition \ref{def:quasi_rectangle}, so that $\tilde{I}_j$ is an octagon. Recalling that we denoted by $\e\alpha_{j, i}$ the distances between the parallel sides of $I_{j+1}$ and $I_j$, we define a new octagon $\tilde{I}_{j+1}$ (with sides $\tilde{P}_{j+1, i}$ and $\tilde{D}_{j+1, i}$) through the relations \eqref{eq: spostamenti} and \eqref{eq: Lung. lati} with $\tilde{P}_{j, i},\,\tilde{P}_{j+1, i}$ and $\tilde{D}_{j, i},\,\tilde{D}_{j+1, i}$ in place of $P_{j, i},\,P_{j+1, i}$ and $D_{j, i},\,D_{j+1, i}$ respectively. We set $\beta_{j, i}$ such that $\tilde{D}_{j+1, i} = \tilde{D}_{j, i}$, so that $\tilde{I}_{j+1}$ is an $\e$-quasi-rectangle. We observe that, denoting by $\tilde{\DD}_{j+1,i}$ the diagonal sides of $\tilde{I}_{j+1}$, by construction, for every $p\in\tilde{\DD}_{j+1, i}$ it holds $d^\e_1(p, \partial I_j)\ge \alpha_{j, 1}+1.$ Since $I_{j+1}$ satisfies \eqref{eq:good_minimizer} we conclude that $\tilde{I}_{j+1}\subset I_{j+1},$ and the claim follows.
\end{proof}

\begin{proposition}\label{prop:surrounded_shape_minimizer_final_part}    Let $\gamma<2$ and let $\delta_\e = \max\{\e\lceil4\sqrt{C_\e}\rceil, \e\lceil\e^{1/8-1}\rceil\}\}$. Let $u_j\in\A_\e$ be such that $I_j$ is an $\e$-quasi-rectangle $I_{j} = \tilde{I}_{j}\cup(\cup_{i=1}^4\PP_{j, i})\cup(\cup_{i=1}^4\Delta_{j, i})$ satisfying assumption \eqref{H}, and let $u_{j+1}$ be a good minimizer for $\enF(\cdot, u_j)$. Then, there exists $\bar{\e}$ such that for every $\e\le\bar{\e}$ the set $\I_{j+1}$ is an $\e$-quasi-rectangle, and there exists a constant $c$
depending only on $\bar{c}$ (the constant of \eqref{H}) such that the displacements of the parallel sides are 
\begin{equation}\label{eq:movement_parallel_sides_surrounding_surfactant_final_steps}
        \alpha_{j, i}\begin{cases}
            \in\Bigl\{\Bigl\lfloor\frac{2\zeta(1-k)}{\tilde{P}_{j,i}}\Bigl\rfloor,\Bigl\lceil\frac{2\zeta(1-k)}{\tilde{P}_{j,i}}\Bigl\rceil\Bigl\}&\text{ if }\text{ dist}\bigl(\frac{2\zeta(1-k)}{\tilde{P}_{j,i}}, \NN\bigl)\ge c\delta_\e\\[0.5em]
            \in\Bigl\{\Bigl[\frac{2\zeta(1-k)}{\tilde{P}_{j, i}}\Bigl]-1, \Bigl[\frac{2\zeta(1-k)}{\tilde{P}_{j, i}}\Bigl], \Bigl[\frac{2\zeta(1-k)}{\tilde{P}_{j,i}}\Bigl]+1\Bigl\}&\text{ otherwise.}
            \end{cases}
        \end{equation}
\end{proposition}
\begin{proof}
    To simplify the notation we write $\alpha_i$ instead of $\alpha_{j, i}$ for $i=1, \dots, 4$. We prove the claim for $\alpha_{1}$ assuming, for simplicity, that $\alpha_1>0$. We can prove that \eqref{eq:movement_parallel_sides_surrounding_surfactant_final_steps} holds also in case $\alpha_1=0$ arguing as in the proof of Theorem \ref{teo:free_surfactant_movement}. 
    
    For $s=j, j+1$ we denote by $H^{s,i} = \dseg{p^{s, i}}{q^{s, i}},\,i=1, \dots, n_h^s$ the horizontal slices of $\I_{s}$ and let
    \[\PP_{s, 1}^{-}= \dseg{p^{s, 2}-\e e_2}{p^{s, 1}-\e e_1},\;\text{ and }\;\PP_{s, 1}^{+}= \dseg{q^{s, 1}+\e e_1}{q^{s, 2}-\e e_2}.\]
    Since by assumption $u_{j+1}$ is a good minimizer, then, by Definition \ref{def:good_minimizer}, $I_{j+1}$ is an $\e$-quasi-rectangle. In what follows we provide the complete proof in the case that $\I_{j+1}\cup\Z_{j+1}$ is a rectangle, while we comment later on the general case. 
    
    Suppose then that $\I_{j+1}\cup\Z_{j+1}$ is a rectangle. We denote by $B = \dseg{p^B}{q^B}$ the lowest slice of $\Z_{j+1},$ and we observe that for $\e$ sufficiently small it holds
    \begin{equation}\label{eq:5.12}
        \e(\#B-\#H^{j+1,2})\le c\delta_\e\;\;\text{ and }\;\;\e|\#H^{j,2}-\#H^{j+1,2}|\le c\delta_\e,
    \end{equation}
    for a positive constant $c$. We observe that, since $u_{j+1}$ is a minimizer, we can assume that either 
    \begin{equation}\label{eq:5.14}
        p^{j,1}_1\le p^{j+1,1}_1\le q^{j+1,1}_1\le q^{j,1}_1,\; \text{ or }\;p^{j+1,1}_1\le p^{j,1}_1\le q^{j,1}_1\le q^{j+1,1}_1      
    \end{equation}
    (otherwise we could strictly reduce the dissipation with a horizontal translation of the slice $H^{j+1,1}$).  
    We discuss the first situation of (\ref{eq:5.14}) (the other case being analogous).
    
    As in the proof of Theorem \ref{teo:free_surfactant_movement} we compare $u_{j+1}$ with two competitors $\tilde{u}_{j+1}^\pm$. 
    Let $N^-:=\#B-\#H^{j+1, 2}$ and denote by $S^-$ the segment
    \[S^-:=\dseg{p^{j+1, 1}+\e e_2}{p^{j+1, 1}+(N^--1)\e e_1+\e e_2}.\]
    We define a competitor $\tilde{u}_{j+1}^-$ as    \begin{equation}\label{eq:5.13}
        \tilde{u}^-_{j+1}(p):=        \begin{cases}
            0&\text{ if }p\in \PP_{j+1, 1}\cup (\PP^+_{j+1, 1}+\e e_2)\cup (\PP^-_{j+1, 1}+\e e_2)\cup S^-,\\
            -1&\text{ if }p\in B,\\
            u_{j+1}(p)&\text{ otherwise}.
        \end{cases}
    \end{equation}
    From Lemma \ref{lemma:energy_surfactant}-(ii) we deduce that $\E_\e(\tilde{u}^-_{j+1}) = \E_\e(u_{j+1})-2(1-k)\e.$ 
    We write $I_j = \tilde{I}_j\cup (\cup_{i=1}^4\PP_{j, i})\cup(\cup_{i=1}^4\Delta_{j, i})$, and we denote by $\tilde{P}_{j, i}$ the lengths of the parallel sides of $\tilde{I}_j$. From (\ref{eq:5.12}), setting $\e\#\PP_{s, 1}^\pm:=P^\pm_{s, 1}$ for $s=j, j+1$, we have that the variation of dissipation is 
    \begin{align*}    \dis^1(\tilde{u}^-_{j+1}, u_j)-&\dis^1(u_{j+1}, u_j)\le 
    \frac{(P^-_{j, 1}+P_{j+1, 1}+P^+_{j, 1})(\alpha_1+1)\e}{\zeta}+\\
    &+\frac{(P_{j, 1}-P_{j+1, 1})(\alpha_1+2)\e}{\zeta}+c\e\delta_\e = \frac{\tilde{P}_{j, 1}(\alpha_{1}+1)\e
    }{\zeta}+\frac{(P_{j, 1}-P_{j+1, 1})\e}{\zeta} + c\e\delta_\e.
    \end{align*}
    By minimality of $u_{j+1}$ we deduce \begin{equation}\label{eq:5.15_left_estimate}        \alpha_1\ge\frac{2\zeta(1-k)}{\tilde{P}_{j, 1}}-\frac{P_{j, 1}-P_{j+1, 1}}{\tilde{P}_{j, 1}}-1-c\delta_\e. 
    \end{equation}
    Now we construct the competitor $\tilde{u}_{j+1}^+$. We set $N^+:=\#H^{j+1,2}$, and we consider two points $\{\bar{p}, \bar{q}\}\subset B-\e e_2$ such that $\bar{p}_1 = p^{j+1, 2}_1$ and $\bar{q}_1 = q^{j+1, 2}_1$, so that 
    $\#\dseg{\bar{p}}{\bar{q}} = \# H^{j+1, 2}.$
    We define \begin{equation}\label{eq:5.16} \tilde{u}^+_{j+1}(p)= 
        \begin{cases}
            0&\text{ if }p\in \dseg{\bar{p}}{\bar{q}},\\
            1&\text{ if }p\in \PP^-_{j+1, 1}\cup(\PP_{j+1, 1}-\e e_2)\cup\PP^+_{j+1, 1},\\
            u_{j+1}(p)&\text{ otherwise}.
        \end{cases}
    \end{equation}
    In this case we find $\E_\e(\tilde{u}^+_{j+1})-\E_\e(u_{j+1}) = 2\e(1-k)$ and the estimate
    \[\dis^1(\tilde{u}^+_{j+1}, u_j)-\dis^1(u_{j+1}, u_j)\le -\frac{1}{\zeta}\Bigl(\tilde{P}_{j, 1}\alpha_1\e+\e\bigl(P_{j, 1}-P_{j+1, 1}\bigl)\Bigl) + c\e\delta_\e.\]
    By minimality of $u_{j+1}$ then we have    \begin{equation}\label{eq:5.17_rigth_estimate}
        \alpha_1\le \frac{2\zeta(1-k)}{\tilde{P}_{j, 1}}-\frac{P_{j, 1}-P_{j+1, 1}}{\tilde{P}_{j, 1}}+c\delta_\e.
    \end{equation}
    The claim follows directly from (\ref{eq:5.15_left_estimate}) and (\ref{eq:5.17_rigth_estimate}), observing that $\alpha_1$ is integer, and that
    \[0\le\frac{P_{j, 1}-P_{j+1, 1}}{\tilde{P}_{j, 1}} \le 1.\]
    In the general case $\I_{j+1}\cup\Z_{j+1}$ is not a rectangle then the proof follows the same line. This time, we can define $\tilde{u}^+_{j+1}$ arguing as in the proof of Lemma \ref{lemma:distance_horizontal_sides}, while to define $\tilde{u}_{j+1}^-$ we can argue as in the proof of Proposition \ref{prop:connectedness}-(Step 2).
\end{proof}

\begin{theorem}
\label{teo:final teo A_j surrounded}
    Let $A\subset\rr^2$ be a not degenerate octagon, and for every $\e>0$ let $A_\e$ be an octagon such that $d_{\mathcal{H}}(A, A_\e)\to 0$ as $\e\to 0$. Let $(u_0^\e)_\e$ be such that $u_0^\e\in\A_\e$ for every $\e$, and $\I^\e_0= A_\e\cap\e\ZZ^2.$ Let $\gamma<2$ and suppose that $C_\e:=\#\Z^\e_0$ satisfies \eqref{eq:5.1} and $\e^2C_\e\to 0$ as $\e\to 0$. Let $u^\e_j$ be a minimizing movement with initial datum $u^\e_0$ and such that $u^\e_j$ is a good minimizer for every $j$. For every $t\ge 0$ we denote by $A_\e(t)$ the set 
    \begin{equation}\label{eq:5_flow}
    A_\e(t) = \bigcup_{i\in\I^\e_{\lfloor t/\tau\rfloor}}Q_\e(i).
    \end{equation}
    Then,  it holds that $A_\e(t)$ converges as $\e\to 0$ locally uniformly in time to $A(t).$     
    $A(t)$ is an octagon for every $t\ge 0$ such that $A(0) = A$; denote by $\PP_i(t),\,\DD_i(t)$ its parallel and sloped sides, and $P_i(t),\,D_i(t)$ their lengths. Then, every side of $A(t)$ moves inwards in the direction which is orthogonal to the side. For every $t$ such that $A(t)$ is not a rectangle, the velocity $v_{\PP_i}(t)$ of the $i$-th side satisfies the following differential inclusion \begin{equation}\label{eq:5.2_sides_velocity}
        v_{\PP_i}(t) \begin{cases}
             = \frac{1}{\zeta}\Bigl\lfloor \frac{2\zeta(1-k)}{P_i(t)}\Bigl\rfloor&\text{ if }\frac{2\zeta(1-k)}{P_i(t)}\not\in\NN,\text{ and }\\[.5em]
             \in\Bigl[\frac{2(1-k)}{P_i(t)}-\frac{1}{\zeta},\,\frac{2(1-k)}{P_i(t)}\Bigl]&\text{ otherwise}.
        \end{cases}
    \end{equation}
    If there exists $\bar{t}$ such that $A(\bar{t})$ is a rectangle, then $A(t)$ is a rectangle for every $t\ge \bar{t}$, and the velocities $v_{\PP_i}$ satisfy the following differential inclusion
\begin{equation}\label{eq:movement_parallel_sides_surrounding_surfactant_final_steps_1}
        v_{\PP_i}(t)\begin{cases}
            \in\Bigl[\frac{1}{\zeta}\Bigl\lfloor\frac{2\zeta(1-k)}{P_i(t)}\Bigl\rfloor,\frac{1}{\zeta}\Bigl\lceil\frac{2\zeta(1-k)}{P_i(t)}\Bigl\rceil\Bigl]&\text{ if }\frac{2\zeta(1-k)}{P_i(t)}\not\in\NN,\\[.5em]          \in\Bigl[\frac{2(1-k)}{P_i(t)}-\frac{1}{\zeta},\frac{2(1-k)}{P_i(t)}+\frac{1}{\zeta}\Bigl]&\text{ otherwise.}
            \end{cases}
        \end{equation}
    Moreover as long as $D_i(t)>0$ the velocity of the side $\DD_i(t)$ is zero. If $D_i(\hat{t})=0$ for a certain $\hat{t}>0$, then $D_i(t)=0$ for every $t\ge\hat{t}$. 
\end{theorem}
\begin{proof}
The proof directly follows from Propositions \ref{prop:connectedness}, \ref{prop:surrounded_shape_minimizer_beginning} and \ref{prop:surrounded_shape_minimizer_final_part}, arguing as in the proof of Theorem \ref{teo:free_surfactant_movement}. We only point out that the connectedness is ensured by Proposition \ref{prop:connectedness} together with Lemma \ref{lemma:5.10}-(ii).
\end{proof}

\end{document}